\documentclass[a4paper]{amsart}

\usepackage{amssymb}

\usepackage[hidelinks]{hyperref}
\usepackage[textsize=small]{todonotes}
\usepackage{xcolor}
\usepackage{enumerate}
\usepackage{lipsum}  
\usepackage[normalem]{ulem} 
\usepackage{rotating}
\usepackage{framed}

\theoremstyle{plain}
\newtheorem{theorem}{Theorem}[section]
\newtheorem{lemma}[theorem]{Lemma}
\newtheorem{corollary}[theorem]{Corollary}
\newtheorem{proposition}[theorem]{Proposition}

\theoremstyle{definition}
\newtheorem{definition}[theorem]{Definition}

\newtheorem{example}[theorem]{Example}

\newtheorem{problem}[theorem]{Problem}
\newtheorem{question}[theorem]{Question}

\theoremstyle{remark}
\newtheorem*{remark}{Remark}

\newcommand{\bN}{\mathbb{N}}
\newcommand{\N}{\bN}
\newcommand{\bZ}{\mathbb{Z}}
\newcommand{\Z}{\bZ}

\newcommand{\cA}{\mathcal{A}}
\newcommand{\cB}{\mathcal{B}}
\newcommand{\cC}{\mathcal{C}}

\newcommand{\cD}{\mathcal{D}}

\newcommand{\cF}{\mathcal{F}}

\newcommand{\cG}{\mathcal{G}}
\newcommand{\cH}{\mathcal{H}}
\newcommand{\cI}{\mathcal{I}}
\newcommand{\I}{\cI}
\newcommand{\cJ}{\mathcal{J}}
\newcommand{\J}{\cJ}
\newcommand{\cK}{\mathcal{K}}
\newcommand{\K}{\cK}
\newcommand{\cL}{\mathcal{L}}

\newcommand{\cN}{\mathcal{N}}
\newcommand{\cP}{\mathcal{P}}
\newcommand{\cR}{\mathcal{R}}
\newcommand{\cT}{\mathcal{T}}

\newcommand{\cU}{\mathcal{U}}

\newcommand{\cW}{\mathcal{W}}

\newcommand{\continuum}{\mathfrak{c}}

\newcommand{\pnumber}{\mathfrak{p}}

\newcommand{\fin}{\mathrm{Fin}}
\newcommand{\Fin}{\fin}

\newcommand{\conv}{\mathrm{conv}} 
\newcommand{\nwd}{\mathrm{nwd}}

\newcommand{\Hindman}{\mathcal{H}} 

\newcommand{\Diff}{\mathcal{D}}
\newcommand{\Ramsey}{\mathcal{R}}
\newcommand{\vdW}{\mathcal{W}}

\DeclareMathOperator{\FS}{FS}
\DeclareMathOperator{\closure}{cl}

\DeclareMathOperator{\FinBW}{FinBW}
\DeclareMathOperator{\hFinBW}{hFinBW}

\begin{document}


\title{A unified approach to Hindman, Ramsey and van der Waerden spaces}


\author[R.~Filip\'{o}w]{Rafa\l{} Filip\'{o}w}
\address[Rafa\l{}~Filip\'{o}w]{Institute of Mathematics\\ Faculty of Mathematics, Physics and Informatics\\ University of Gda\'{n}sk\\ ul.~Wita Stwosza 57\\ 80-308 Gda\'{n}sk\\ Poland}
\email{Rafal.Filipow@ug.edu.pl}
\urladdr{\url{http://mat.ug.edu.pl/~rfilipow}}

\author[K.~Kowitz]{Krzysztof Kowitz}
\address[Krzysztof Kowitz]{Institute of Mathematics\\ Faculty of Mathematics\\ Physics and Informatics\\ University of Gda\'{n}sk\\ ul.~Wita  Stwosza 57\\ 80-308 Gda\'{n}sk\\ Poland}
\email{Krzysztof.Kowitz@phdstud.ug.edu.pl}

\author[A.~Kwela]{Adam Kwela}
\address[Adam Kwela]{Institute of Mathematics\\ Faculty of Mathematics\\ Physics and Informatics\\ University of Gda\'{n}sk\\ ul.~Wita  Stwosza 57\\ 80-308 Gda\'{n}sk\\ Poland}
\email{Adam.Kwela@ug.edu.pl}
\urladdr{\url{https://mat.ug.edu.pl/~akwela}}


\date{\today}






\begin{abstract}
For many years, there have been conducting research (e.g.~by Bergelson, Furstenberg, Kojman, Kubi\'{s}, Shelah, Szeptycki, Weiss) into sequentially compact  spaces that are, in a sense, topological counterparts of some combinatorial theorems, for instance 
Ramsey’s theorem for coloring graphs, 
Hindman’s finite sums theorem  and 
van der Waerden’s arithmetical progressions theorem. These  spaces are defined with the aid of different kinds of convergences: IP-convergence, R-convergence and ordinary convergence. 

The first aim of this paper is to present a unified approach to these various types of convergences and spaces. Then, using this unified  approach, we prove some general theorems about existence of the considered spaces and show that all  results obtained so far in this subject can be derived from our theorems.  

The second aim of this paper is to obtain  new results about the specific types of these spaces. For instance, we construct a Hausdorff Hindman space that is not an $\I_{1/n}$-space and a Hausdorff differentially compact space that is not Hindman. 
Moreover, we compare Ramsey spaces with other types of spaces. For instance, we construct a Ramsey space that is not Hindman and a Hindman space that is not Ramsey.  

The last aim of this paper is to provide a characterization  that shows when  there exists a space of one considered type that is not of the other  kind. 
This characterization is expressed  in purely combinatorial manner with the aid of 
 the so-called Kat\v{e}tov order that has been extensively examined  for many years so far.

This paper may interest the general audience of mathematicians as the results we obtain are on the intersection of topology, combinatorics, set theory and  number theory.

\end{abstract}


\maketitle


\setcounter{tocdepth}{1}
\tableofcontents


\section{Introduction}

For more than twenty years, many mathematicians have been examining  sequentially compact  spaces that are, in a sense, 
topological counterparts of some combinatorial theorems, for instance 
Ramsey’s theorem for coloring graphs, 
Hindman’s finite sums theorem  and 
van der Waerden’s arithmetical progressions theorem
\cite{
BergelsonZelada,
MR2948679,
MR3461181,
MR2904054,
MR3097000,
MR4356195,
MR2320288,
MR2961261,
MR3276758,
MR2471564,
MR603625,
MR531271,
MR2052425,
MR1887003,
MR1866012,
MR1950294,
MR4358658,
MR4552506,
MR3594409,
MR4584767,
Shi2003Numbers}. 
These  spaces are defined with the aid of different kinds of convergences: IP-convergence, R-convergence and ordinary  convergence.

We start our brief overview of these spaces with the ones defined using ordinary convergence. A topological space $X$ is called:
\begin{itemize}
    \item \emph{van der Waerden} \cite{MR1866012} if for every sequence
$\langle x_n\rangle_{n\in \N}$ in $X$  there exists a convergent subsequence 
$\langle x_{n}\rangle_{n\in A}$ with $A$ being an AP-set (i.e.~$A$ contains arithmetic progressions of arbitrary finite length);
    \item an \emph{$\I_{1/n}$-space} \cite{MR2471564} if for every sequence
$\langle x_n\rangle_{n\in \N}$ in $X$  there exists a convergent subsequence 
$\langle x_{n}\rangle_{n\in A}$ with $A$ having the property that the series of reciprocals of elements of $A$ diverges. 
\end{itemize}

In fact both mentioned classes of spaces are special cases of a more general notion. A nonempty family $\I\subseteq\cP(\N)$ of subsets of $\N$ is an \emph{ideal on $\N$} if it is closed under taking subsets and finite unions of its elements, $\N\notin\I$ and $\I$ contains all finite subsets of $\N$ (it is easy to see that the family $\I_{1/n}=\{A\subseteq\N: \sum_{n\in A} 1/n<\infty\}$ 
is an ideal on $\N$, and it follows from van der Waerden's theorem \cite{vanderWearden} that the family 
$\cW = \{A\subseteq\N: \text{$A$ is not an AP-set}\}$ is an ideal on $\N$). If $\I$ is an ideal on $\N$ then a topological space $X$ is called an \emph{$\I$-space} \cite{MR2471564} if for every sequence
$\langle x_n\rangle_{n\in \N}$ in $X$  there exists a converging subsequence 
$\langle x_{n}\rangle_{n\in A}$ with $A\notin \I$.
In particular, van der Waerden spaces coincide with $\cW$-spaces.

Now we want to turn our attention to spaces defined with the aid of different kinds of convergence. We start with Hindman spaces. A set $A\subseteq \N$ is an \emph{IP-set} \cite{MR531271}  if there exists an infinite set $D\subseteq\N$ such that $\FS(D)\subseteq A$ where $\FS(D)$ denotes the set of all finite sums of distinct elements of $D$.
The family $\Hindman = \{A\subseteq\N: \text{$A$ is not an IP-set}\}$ is an ideal on $\N$ (it follows from Hindman's theorem \cite{MR349574}). 

An \emph{IP-sequence} in $X$ is a sequence indexed by $\FS(D)$ for some infinite $D\subseteq\N$.
An IP-sequence $\langle x_n\rangle_{n\in \FS(D)}$ in a topological space $X$ is \emph{IP-convergent} \cite{MR531271} to a point $x\in X$ if for every neighborhood $U$ of  $x$ there exists $m\in \N$ so that $x_n \in  U$ for every $n\in \FS(D\setminus \{0,1,\dots,m\})$ (then  $x$ is called the \emph{IP-limit} of the sequence).

Since only finite spaces are $\cH$-spaces \cite{MR1887003}, Kojman replaced the ordinary convergence with IP-convergence (introduced by Furstenberg and Weiss \cite{MR531271}) to define a meaningful topological counterpart of Hindman's finite sums theorem. Namely, a topological space $X$ is called \emph{Hindman} \cite{MR1887003} if for every sequence
$\langle x_n\rangle_{n\in \N}$ in $X$  there exists an infinite set $D\subseteq \N$ such that the subsequence $\langle x_n\rangle_{n\in \FS(D)}$ IP-converges to some $x\in X$. 

We finish our brief overview of classes of sequentially compact spaces with Ramsey spaces. Let  $[A]^2$  denote the set of all pairs of elements of $A$.
A sequence    
$\langle x_n\rangle_{n\in [D]^2}$ in $X$ (indexed by pairs of natural numbers from some infinite set $D\subseteq\N$)  
\emph{R-converges} \cite{MR2948679, BergelsonZelada}
to a point $x\in X$ if for every neighborhood  $U$ of $x$ there is a finite set $F$ such that $x_{\{a,b\}} \in U$ for all distinct $a,b\in D\setminus F$.
A topological space $X$ is called \emph{Ramsey} \cite{MR4552506} 
 if for every sequence  
$\langle x_n\rangle_{n\in [\N]^2}$ in $X$ there exists an infinite set $D\subseteq\N$ such that the subsequence $\langle x_n\rangle_{n\in [D]^2}$ R-converges to some $x\in X$.

We say that an ideal $\I$ (on $\N$) is below an ideal $\J$  in the \emph{Kat\v{e}tov order} \cite{MR250257}  
 if there is a function $f:\N\to \N$ such that $f^{-1}[A]\in\J$ for every $A\in \I$. 
Note that Kat\v{e}tov order  has been extensively examined (even in its own right) for many years so far 
\cite{
MR3034318,
MR1335140,
MR3600759,
MR4247792,
MR4378082,
MR4036731,
MR3513296,
MR3692233,
MR2777744,
MR3696069,
MR2017358,
MR3019575,
alcantara-phd-thesis,
MR3555332,
MR3550610,
MR3778962,
MR4312995}.
 

\smallskip
 
There are three objectives of this paper. 
The first aim is to present a unified approach to these various types of convergences and spaces. This is achieved in sections in Part~\ref{Part:Partition_regular} with the aid of partition regular functions (Definition~\ref{def:partition-regular}), a convergence with respect to partition regular functions (Definition~\ref{def:rho-convergence}) and a subclass of sequentially compact spaces defined using this new kind of convergence  (see Definition~\ref{def:FinBW-for-rho}). Then using this approach, we prove some general theorems about those classes of spaces (Theorem \ref{thm:FinBW-hFinBW-homogeneous}) and show that all  results obtained so far in this subject can be derived from our theorems (see sections in Parts~\ref{part:FinBW-spaces} and \ref{part:Distinguishing-FinBW-spaces}).  

The second aim of this paper is to obtain  new results concerning specific types of these spaces: Ramsey spaces, Hindman spaces, van der Waerden spaces and $\I_{1/n}$-spaces. For instance, we construct a Hausdorff Hindman space that is not an $\I_{1/n}$-space   
(Corollary~\ref{cor:Fspaces:Pminus-Mrowka}(\ref{cor:Fspaces:Pminus-Mrowka:Hindman-not-summable:existence})) -- this gives  a positive answer to a  question posed by Fla\v{s}kov\'{a}~\cite{Flaskova-slides-2007} (so far only non-Hausdorff answer to this question was known \cite[Theorem~2.5]{MR4356195}). 
We also construct  a Hausdorff so-called differentially compact space that is not Hindman (Corollary~\ref{cor:distinguishing-Hindman-Ramsey-Diff-spaces}(\ref{cor:distinguishing-Hindman-Ramsey-Diff-spaces:diff-not-Hindman-space})) which  yields the  negative  answer to
a question posed by Shi~\cite[Question~4.2.2]{Shi2003Numbers} and other authors \cite[Problem 1]{MR3097000},\cite[Question 3]{MR4358658}. 
Moreover, we compare Ramsey spaces with other types of spaces (so far Ramsey spaces were only examined in their own right without comparing them with other kinds of spaces \cite{MR2948679,MR4552506,corral-guzman-lopez}). For instance, we construct a Ramsey space that is not Hindman and a Hindman space that is not Ramsey (Corollary~\ref{cor:distinguishing-Hindman-Ramsey-Diff-spaces}).

The final aim of this paper is to provide a characterization  that shows when there exists a space of one considered type that is not of the other  type (Theorem~\ref{thm:characterization-for-rho} and other results in Part~\ref{part:characterization}). This characterization is expressed in purely combinatorial manner with the aid of the  Kat\v{e}tov order or its counterpart in the realm of partition regular functions (Definition~\ref{Katetov-order-for-rho}).


\section{Preliminaries}

In the paper we are exclusively interested in Hausdorff topological spaces 
with one exception (Sections~\ref{sec:nonHausdorf-world} and \ref{sec:nonHausdorf-world2}) where we were unable to obtain results for Hausdorff spaces but succeeded  in constructing  a topological space with unique limits of sequences.

Following von Neumann, we identify an ordinal number $\alpha$  with the set of all ordinal numbers less than $\alpha$. 
In particular, the  smallest infinite ordinal number $\omega=\{0,1,\dots\}$ is equal to the set $\N$ of all natural numbers, and each natural number $n = \{0,\dots,n-1\}$ is equal  to the set of all natural numbers less than $n$.
Using this identification, we can for instance write $n\in k$ instead of $n<k$ and $n<\omega$ instead of $n\in \omega$ or $A\cap n$ instead of $A\cap \{0,1,\dots,n-1\}$. 

If $A\subseteq\omega$ and $n\in \omega$, we write
$A+n = \{a+n:a\in A\}$ and $A-n=\{a-n:a\in A, a>n\}$.

We write  
$[A]^2$ to denote the set of all unordered pairs of elements of $A$, 
$[A]^{<\omega}$ to denote the family of all finite subsets of $A$,
$[A]^\omega$ to denote the family of all infinite countable subsets of $A$
and
$\cP(A)$ to denote the family of all subsets of $A$.

We say that a family $\cA$ of subsets of a set $\Lambda$ is an \emph{almost disjoint family on $\Lambda$} if 
\begin{enumerate}
\item $|A|=|\Lambda|$ for every $A\in \cA$ and 
\item $|A\cap B|<|\Lambda|$ for all  distinct elements $A,B\in \cA$.
\end{enumerate}

By  $A \sqcup B$
we denote the \emph{disjoint union} of sets $A$ and $B$:
$$A \sqcup B = (A\times\{0\})\cup (B\times\{1\})  = \{(x,0): x\in A\}\cup\{(y,1):y\in B\}.$$
For families of sets $\cA\subseteq\cP(\Lambda)$ and $\cB\subseteq \cP(\Sigma)$, we write
$\cA \oplus \cB =\{ 
A\sqcup B:A\in \cA,B\in \cB\}.$

A nonempty family $\I\subseteq\cP(\Lambda)$ of subsets of $\Lambda$ is an \emph{ideal on $\Lambda$} if it is closed under taking subsets and finite unions of its elements, $\Lambda\notin\I$ and $\I$ contains all finite subsets of $\Lambda$.
By $\fin(\Lambda)$ we denote the family of all finite subsets of  $\Lambda$. For $\Lambda=\omega$, we write $\Fin$ instead of $\Fin(\omega)$. 
For an ideal $\I$ on  $\Lambda$, we write $\I^+=\{A\subseteq \Lambda: A\notin\I\}$ and call it the \emph{coideal of $\I$}, and we write $\I^*=\{\Lambda\setminus A: A\in\I\}$ and call it the \emph{filter dual to $\I$}.
For an ideal $\I$ on $\Lambda$
and $A\in \I^+$, it is easy to see that  
$\I\restriction A=\{A\cap B:B\in \I\}$
is an ideal on $A$. 

In our research the following ideal on $\omega^2$ plays an important role:
$$\Fin^2=\left\{C\subseteq\omega^2: 
\{n\in \omega: \{k\in\omega:(n,k)\in C\}\notin \Fin\}\in \Fin\right\}.$$

We say that a function $f:\Lambda\to \Sigma$
is \emph{$\I$-to-one} if $f^{-1}(\sigma)\in \I$ for every $\sigma\in \Sigma$.

A set $A\subseteq X$ 
is $F_\sigma$ ($G_\delta$, $F_{\sigma\delta}$, etc., resp.) in a topological space $X$  if $A$ is a union of a countable family of closed sets ($A$ is an intersection of a countable family of open sets, $A$ is an intersection  of a countable family of $F_\sigma$ sets, etc., resp.).

For a function $f:X\to Y$ and a set $A\subseteq X$, we write $f\restriction A$ to denote the restriction of $f$ to the set $A$.


\part{Partition regular operations}
\label{Part:Partition_regular}



\section{Partition regular operations and ideals associated with them}

Below we introduce a notion that proved to be a convenient tool allowing  to grasp the common feature of different kinds of convergences related to  Hindman, Ramsey and van der Waerden spaces.

\begin{definition}
\label{def:partition-regular}
Let $\Lambda$ and $\Omega$ be  countable infinite sets.
Let $\cF$ be a nonempty family of infinite subsets of $\Omega$ such that $F\setminus K\in \cF$ for every $F\in \cF$ and a finite set $K\subseteq \Omega$.
We say that a function $\rho:\cF\to [\Lambda]^\omega$ 
is \emph{partition regular} if 
\begin{description}
\item[(M)\label{def:partition-regular:monotone}] 
$\forall E,F\in \cF\, \left( E\subseteq F \implies \rho(E)\subseteq \rho(F)\right)$, 

\item[(R)\label{def:partition-regular:partition-regular}]
$\forall F\in \cF\,\forall A,B \subseteq \Lambda \,\left(\rho(F) =  A\cup B\implies 
\exists E\in \cF \,(\rho(E)\subseteq A \lor \rho(E)\subseteq B)
\right).$

\item[(S)\label{def:partition-regular:finitely-supported}]
$\forall F\in \cF\,\exists E\in \cF\,(E\subseteq F \land \forall a\in \rho(E)\,\exists K\in[\Omega]^{<\omega}(a\notin \rho(E\setminus K)))$.
\end{description}
\end{definition}

In our considerations, we use the following easy observation concerning condition  \nameref{def:partition-regular:finitely-supported} of Definition \ref{def:partition-regular}. 

\begin{proposition}
\label{prop:finite-support}
Let 
$\rho:\cF\to[\Lambda]^\omega$ (with $\cF\subseteq[\Omega]^\omega$) be a partition regular function. Then for every $F\in \cF$ there is $E\in \cF$ such that $E\subseteq F$ and for every finite set $L\subseteq \Lambda$ there exists a finite set $K\subseteq \Omega$
such that $\rho(E\setminus K)\subseteq \rho(E)\setminus L$.
\end{proposition}

\begin{proof}
For $F\in \cF$, let $E\in \cF$ be as in condition \nameref{def:partition-regular:finitely-supported} of Definition \ref{def:partition-regular}. 
Let $L\subseteq \Lambda$ be a finite set.
For every $a\in \rho(E)$, we take a finite set $K_a$ such that $a\notin \rho(E\setminus K_a)$.
Then $K=\bigcup\{K_a:a\in \rho(E)\cap L\}$ is finite 
and 
$\rho(E\setminus K)\subseteq \rho(E)\setminus L$.
\end{proof}

The following easy proposition reveals basic relationships between  partition regular functions and ideals.

\begin{proposition}\ 
\label{prop:rho-versus-ideal}
\begin{enumerate}
\item If $\rho:\cF\to[\Lambda]^\omega$ is partition regular, then 
$$\I_{\rho} = \{A\subseteq \Lambda: \forall  F\in \cF\, (\rho(F)\not\subseteq A)\}.$$
is an ideal on $\Lambda$. 
\label{prop:rho-versus-ideal:rho-gives-ideal}

    \item 
If $\I$ is an ideal on $\Lambda$, then 
the function 
$$\rho_{\I}:\I^+\to[\Lambda]^\omega \text{\ \ \  given by \ \ } \rho_{\I}(A)=A$$ 
is partition regular and  $\I=\I_{\rho_{\I}}.$ \label{prop:rho-versus-ideal:ideal-gives-rho} 
\end{enumerate}
\end{proposition}

\begin{remark}
If $\rho$ is partition regular and $\tau =\rho_{\I_\rho}$, then 
$\I_{\tau} = \I_\rho$ but, as we will see, in general, $\rho \neq \tau$.  More important,  $\tau$ may miss some crucial properties which $\rho$ possesses (e.g.~$P$-like properties -- see Proposition~\ref{prop:properties-of-ideal-versus-rho}(\ref{prop:Plike-properties-for-known-rho:FS-r-Delta:Pminus})(\ref{prop:Plike-properties-for-known-rho:Hindman-Ramsey-Diff})).\label{prop:rho-versus-ideal:rho-gives-ideal:rho}
\end{remark}

Below we present the most important examples of partition regular functions that were our prototypes while we were thinking on a unified approach to  Hindman, Ramsey and van der Waerden spaces.

\subsection{Hindman’s finite sums theorem}

Let the function $\FS :[\omega]^\omega\to [\omega]^\omega$ be given by 
$$\FS(D)  = \left\{ \sum_{n\in \alpha} n : \alpha\in [D]^{<\omega}\setminus\{\emptyset\}\right\}$$
i.e.~$\FS(D)$ is the set of all finite non-empty sums of distinct elements of $D$. 

A set $D \subseteq \omega$ is \emph{sparse} \cite[p.~1598]{MR1887003} if for each $n\in \FS(D)$ there exists the unique set $\alpha \subseteq D$ such that $n = \sum_{i\in \alpha} i$. This unique set will be denoted by $\alpha_D(n)$.
For instance, the set $E = \{2^i:i\in\omega\}$ is sparse, and in the sequel, we write $\alpha(n)$ instead of $\alpha_E(n)$. 


A sparse set $D\subseteq \omega$ is \emph{very sparse} \cite[p.~894]{MR4356195} if $\alpha_D(x)\cap \alpha_D(y)\neq\emptyset$ implies  $x+y\notin \FS(D)$ for every $x,y\in \FS(D)$.

\begin{theorem}[Hindman] 
The function 
 $\FS$ is partition regular
 and 
 the family 
$$\Hindman = \I_{\FS} = \{A\subseteq \omega: \forall  D\in [\omega]^\omega\, (\FS(D)\not\subseteq A)\}$$
is an ideal on  $\omega$. The ideal $\Hindman$ is called \emph{the Hindman ideal} \cite[p.~109]{MR2471564}.
It is known that sets from  $\Hindman^+$ (that  are called \emph{IP-sets})  are examples of  so-called Poincar\'{e} sequences\footnote{A set $W\subseteq\Z$  is called a \emph{Poincar\'{e} sequence} \cite[Definition~3.6 at p.~72]{MR603625} if for any measure preserving system $(X,\mathcal{B},\mu, T)$ and $A\in \mathcal{B}$ with $\mu(A)>0$ we have $\mu (T^{-n}[A]\cap A)>0$ for some $n\in W$, $n\neq0$.} that play an important role in the study of recurrences in topological dynamics~\cite[p.~74]{MR603625}.
\end{theorem}

\begin{proof}
It is easy to see that condition \nameref{def:partition-regular:monotone}
of   Definition~\ref{def:partition-regular} is satisfied for $\FS$.
Condition \nameref{def:partition-regular:partition-regular} of   Definition~\ref{def:partition-regular}
 holds for $\FS$ as in this case it is the well known Hindman's finite sums theorem \cite[Theorem~3.1]{MR349574},\cite[Theorem~3.5]{MR2757532}. 
 To see that condition  \nameref{def:partition-regular:finitely-supported} of   Definition~\ref{def:partition-regular} holds for $\FS$, it is enough to notice \cite[p.~1598]{MR1887003} that  every infinite set $F\subseteq\omega$ has an infinite sparse subset $G\subseteq F$ which obviously satisfies  condition \nameref{def:partition-regular:finitely-supported}.
Finally,  Proposition~\ref{prop:rho-versus-ideal}(\ref{prop:rho-versus-ideal:rho-gives-ideal}) shows that $\cH$ is an ideal on $\omega$.
\end{proof}

The following lemma will be used in some proofs regarding properties of the function $\FS$.

\begin{lemma}[{\cite[Lemma~7]{MR1887003}}]
\label{lem:KOJMAN}
If $D$ is an infinite  sparse set, then there exists a set  $S=\{s_i:i\in\omega\}\subseteq \FS(D)$ such that for every $i\in\omega$ we have 
$s_i<s_{i+1}$ and 
$$\max \alpha_D(s_i)<\min \alpha_D(s_{i+1}) \text{ and } \max \alpha(s_i)<\min \alpha(s_{i+1}).$$
\end{lemma}

\subsection{Ramsey’s theorem for coloring graphs}

\begin{theorem}[Ramsey]
Let $r:[\omega]^\omega\to \left[[\omega]^2\right]^\omega$
be given by 
$$r(H) = [H]^2 = \{\{x,y\}\subseteq [\omega]^2: x,y\in H, x\neq y\}$$
i.e.~$r(H)$ is the set of all unordered pairs of elements of $H$.
Then $r$ is partition regular and the family 
$$\Ramsey = \I_r = \{A\subseteq[\omega]^2: \forall H\in [\omega]^\omega\, ([H]^2\not\subseteq A)\}$$
is an ideal on $[\omega]^2$.
The ideal $\Ramsey$ is called the \emph{Ramsey ideal} \cite{alcantara-phd-thesis, MR3019575}.
(If we identify a set $A\subseteq [\omega]^2$ with a graph $G_A=(\omega,A)$, the ideal $\Ramsey$ can be seen as an ideal consisting of graphs without infinite complete subgraphs).
\end{theorem}

\begin{proof}
It is easy to see that condition \nameref{def:partition-regular:monotone}
of   Definition~\ref{def:partition-regular} is satisfied for $r$.
Condition \nameref{def:partition-regular:partition-regular} of   Definition~\ref{def:partition-regular}
 holds for $r$ as in this case it is the well known Ramsey's theorem for coloring graphs \cite[Theorem~A]{MR1576401},\cite[Theorem~1.5]{MR1044995}. 
 To see that condition  \nameref{def:partition-regular:finitely-supported} of   Definition~\ref{def:partition-regular} holds for $r$, 
 it is enough to notice that  for every 
 $\{a,b\}\in [F]^2$ we have $\{a,b\}\notin [F\setminus \{a,b\}]^2$.
 Finally,  Proposition~\ref{prop:rho-versus-ideal}(\ref{prop:rho-versus-ideal:rho-gives-ideal}) shows that $\Ramsey$ is an ideal on $[\omega]^2$.
\end{proof}

\subsection{The positive differences and the associated ideal}

Let the function $\Delta :[\omega]^\omega\to [\omega]^\omega$ be given by 
$$\Delta(E)  = \left\{a-b: a,b\in E, a>b \right\}$$
i.e.~$\Delta(E)$ is the set of all positive differences  of distinct elements of $E$. 

We say that a set  $E\subseteq\omega$ is \emph{$\Diff$-sparse} \cite[p.~2009]{MR3097000} if for every $a\in \Delta(E)$ there are unique elements $b,c\in  E$ such that  $a = b-c$. 

\begin{proposition}
The function  $\Delta$ is partition regular 
and the family 
$$ \Diff = \I_{\Delta} = \{A\subseteq \omega: \forall  E\in [\omega]^\omega\, (\Delta(E)\not\subseteq A)\}$$
is an ideal on  $\omega$ such that  $\Diff\subsetneq \Hindman$. 
It is known that sets from  $\Diff^+$ are examples of so-called Poincar\'{e} sequences \cite[p.~74]{MR603625}.
\end{proposition}

\begin{proof}

It is easy to see that condition \nameref{def:partition-regular:monotone}
of   Definition~\ref{def:partition-regular} is satisfied for $\Delta$.
It is known \cite[Proposition~4.1]{MR3097000} that condition \nameref{def:partition-regular:partition-regular} of Definition~\ref{def:partition-regular}
 holds for $\Delta$. 
 To see that condition \nameref{def:partition-regular:finitely-supported} of   Definition~\ref{def:partition-regular} holds for $\Delta$, it is enough to notice 
\cite[Proposition~4.3(2)]{MR3097000} that every infinite set $F\subseteq \omega$ has an infinite $\Diff$-sparse subset $G\subseteq F$ which obviously satisfies condition \nameref{def:partition-regular:finitely-supported}.
Finally,  Proposition~\ref{prop:rho-versus-ideal}(\ref{prop:rho-versus-ideal:rho-gives-ideal}) shows that $\Diff$ is an ideal on $\omega$ and it is known \cite[Proposition~4.2.1]{Shi2003Numbers},\cite[Proposition~4.1]{MR3097000} that $\Diff\subsetneq\Hindman$.
\end{proof}

\subsection{The summable ideal}

\begin{proposition}
The family 
$$\I_{1/n}=\left\{A\subseteq\omega: \sum_{n\in A}\frac{1}{n+1}<\infty\right\}$$ 
is an ideal on $\omega$. 
The ideal $\I_{1/n}$ is called the \emph{summable ideal} \cite[Definition~1.6]{MR1124539},\cite[Example~3]{MR0363911},\cite[p.~238]{MR588216},\cite[p.~411]{MR748847}.
The function $\rho_{\I_{1/n}}:\I_{1/n}^+\to[\omega]^\omega$ given by $\rho_{\I_{1/n}}(A)=A$ is partition regular and 
$\I_{1/n} = \I_{\rho_{\I_{1/n}}}$.
\end{proposition}

\begin{proof}
It is easy to show that $\I_{1/n}$ is an ideal on $\omega$, whereas Proposition~\ref{prop:rho-versus-ideal}(\ref{prop:rho-versus-ideal:ideal-gives-rho}) gives the required properties of $\rho_{\I_{1/n}}$.
\end{proof}

\subsection{van der Waerden’s arithmetical progressions theorem}

\begin{theorem}[van der Waerden]
A set $A\subseteq \omega$ is called an \emph{AP-set} if it contains an arithmetic progressions of arbitrary finite length. 
The family 
$$\vdW = \{A\subseteq\omega: \text{$A$ is not an AP-set}\}$$
is an ideal on $\omega$. 
The ideal $\vdW$ is called the \emph{van der Waerden ideal} \cite[p.~107]{MR2471564}.
The function $\rho_\vdW:\vdW^+\to[\omega]^\omega$ given by $\rho_\vdW(A)=A$ is partition regular and  $\vdW = \I_{\rho_\vdW}$.
\end{theorem}

\begin{proof}
It is easy to see that all conditions from the definition of an ideal but additivity are satisfied, whereas additivity is the well known van der Waerden's arithmetical progressions theorem \cite{vanderWearden},\cite[Theorem~2.1]{MR1044995}. 
Finally, Proposition~\ref{prop:rho-versus-ideal}(\ref{prop:rho-versus-ideal:ideal-gives-rho})
gives the required properties of $\rho_\vdW$.
\end{proof}

\subsection{Ideals on directed sets}

Finally, we introduce a class  of partition regular functions which are connected with ideals on directed sets \cite{MR3461173,MR3461181}.
Recall, that  $(\Lambda,<)$ is a \emph{directed set} if the relation $<$ is an upward directed strict partial order on $\Lambda$.

Let $(\Lambda,<)$ be a \emph{directed set}  such that $\Lambda$ is infinite countable.
A set $B\subseteq\Lambda$ is \emph{cofinal} in $(\Lambda,<)$ if for every $\lambda\in \Lambda$ there is $b\in B$ with $\lambda <  b$.
A family $\I$ of subsets of $\Lambda$ is an \emph{ideal on $(\Lambda,<)$} 
\cite[Definition~2.2]{MR3461181}
if $\I$ is an ideal on $\Lambda$ and $\I$ contains all sets which are \emph{not} cofinal.
A family $\cB$ of subsets of $\Lambda$ is a \emph{coideal basis on $(\Lambda,<)$} 
\cite[Definition~2.4]{MR3461181}
if  $\cB\neq\emptyset$, all sets in $\cB$ are cofinal and 
if $C\cup D\in\cB$, then there exists $B\in\cB$ such that  $B\subseteq C$ or $B\subseteq D$. In particular, for every ideal $\I$ on $(\Lambda,<)$ the family $\I^+$ is a coideal basis on $(\Lambda,<)$.
It is known \cite[Proposition~2.7]{MR3461173} that  
 $\I$ is an ideal on $(\Lambda,<)$ if and only if there exists a coideal basis $\cB$ on $(\Lambda,<)$ such that $\I=\{A\subseteq\Lambda: \forall B\in \cB\,(B\not\subseteq A)\}$.

 The following easy proposition reveals basic relationships between  partition regular functions and ideals on directed sets.

\begin{proposition}\ 
\label{prop:rho-versus-ideal-on-directed-sets}
Let $(\Lambda,<)$ be a directed set. 
\begin{enumerate}
\item If $\rho:\cF\to[\Lambda]^\omega$ is a partition regular function such that $\rho(F)$ is cofinal for every $F\in \cF$, then 
$$\I_{\rho} = \{A\subseteq \Lambda: \forall  F\in \cF\, (\rho(F)\not\subseteq A)\}.$$
is an ideal on $(\Lambda,<)$. 
\label{prop:rho-versus-ideal-on-directed-sets:rho-gives-ideal}

    \item 
For a coideal basis $\cB$ on $(\Lambda,<)$ (in particular for $\cB=\I^+$, where $\I$ is an ideal on $(\Lambda,<)$), we define
$$ \widehat{\cB} = \{B\setminus K:B\in \cB, K\in[\Lambda]^{<\omega}\}.$$
Then the function $\rho_{\cB}:\widehat{\cB}\oplus\Fin(\Lambda)^*\to[\Lambda]^\omega$ given by
$$\rho_{\cB}((B\setminus K)\sqcup C)=(B\setminus K)\cap\{\lambda\in\Lambda:\forall \lambda'\in (\Lambda\setminus C)\, (\lambda'<\lambda)\}$$ 
is a partition regular function such that $\rho_{\cB}((B\setminus K)\sqcup C)$ is cofinal for every $(B\setminus L)\sqcup C\in \widehat{\cB}\oplus\Fin(\Lambda)^*$ and  $\I_{\rho_{\cB}} = \{A\subseteq\Lambda: \forall B\in \cB\,(B\not\subseteq A)\}.$\label{prop:rho-versus-ideal-on-directed-sets:ideal-gives-rho:coideal-basis} 
\end{enumerate}
\end{proposition}


\section{Restrictions and small accretions}


\subsection{Restrictions of partition regular operations}
\label{sec:restriction-of-rho}

For $B\notin \I_\rho$, we define a family 
$\cF\restriction B = \{E\in \cF: \rho(E)\subseteq B\}$
and a function 
$\rho\restriction B :\cF\restriction B\to [B]^\omega$ by 
$(\rho\restriction B)(E) = \rho(E)$ (i.e.~$\rho\restriction B = \rho\restriction (\cF\restriction B)$).
The following easy proposition reveals relationships between restriction of a function $\rho$ and restriction of an ideal $\I_\rho$.

\begin{proposition}
\label{prop:restrictions-of-rho}
If $\rho:\cF\to[\Lambda]^\omega$ is  partition regular and $B\notin \I_\rho$, then 
$\rho\restriction  B$ is partition regular
and $\I_{\rho\restriction B} = \I_{\rho}\restriction B$.
\end{proposition}


\subsection{Small accretions of partition regular operations}

We will need the following notion in the last part of the paper for characterization that shows when there exists a space of one considered type that is not of the other type (Theorem~\ref{thm:characterization-for-rho}). 

\begin{definition}
\label{def:small-accretions}
Let 
$\rho:\cF\to[\Lambda]^\omega$ (with $\cF\subseteq[\Omega]^\omega$) be a  partition regular function.
\begin{enumerate}
    \item 
A set $F\in \cF$ has \emph{small accretions} if 
$\rho(F)\setminus\rho(F\setminus K)\in\I_\rho$ for every finite set $K$.

\item $\rho$  has \emph{small accretions} if for every $E\in \cF$ there is $F\in \cF$ such that $F\subseteq E$ and $F$ has small accretions.
\end{enumerate}
\end{definition}

\begin{proposition}
\label{prop:ideal-rho-is-sparse}
\label{prop:ideal-rho-are-sparse-and-accretions}
\label{prop:FS-properties:accretion}
If  $\rho\in \{\FS,r,\Delta\}\cup\{\rho_\I:\text{$\I$ is an ideal}\}$, then  $\rho$ has small accretions.
\end{proposition} 

\begin{proof}[Proof for $\rho = \rho_\I$ where  $\I$ is an ideal]
The function $\rho$ has small accretions, since for every $A\in \I^+$ and finite $K\subseteq \Lambda$ we have 
$\rho_\I(A)\setminus\rho_\I(A\setminus K)=A\setminus (A\setminus K) \subseteq K \in \I$.
\end{proof}

\begin{proof}[Proof for $\rho = \FS$]
It is known \cite[Lemma 2.2]{MR4356195} that 
every infinite set $E\subseteq \omega$ has an infinite very sparse subset $F\subseteq E$, so 
if we show that every very sparse set has small accretions, the proof will be finished.

Let $F\subseteq \omega$ be an infinite very sparse set and $K\subseteq\omega$ be a finite set.
Assume towards contradiction that $\FS(D)\subseteq \FS(F)\setminus \FS(F\setminus K)=\{x\in \FS(F):\ \alpha_{F}(x)\cap K\neq\emptyset\}$ for some $D\in[\omega]^\omega$. Since $K$ is finite, we can find $x,y\in D$, $x\neq y$, such that $\alpha_{F}(x)\cap \alpha_{F}(y)\neq\emptyset$. But then $x+y\in \FS(D   )\setminus \FS(F)$, a contradiction.
\end{proof}

\begin{proof}[Proof for $\rho = r$]
The function  $r$  has small accretions, since for every  $A\in[\omega]^\omega$ and finite $K\subseteq \omega$ we have $r(A)\setminus r(A\setminus K) = [A]^2\setminus [A\setminus K]^2=\{\{i,j\}:\ i\in A\cap K,j\in A\}\in\cR$.
\end{proof}

\begin{proof}[Proof for $\rho = \Delta$]
It is known \cite[Proposition~4.3(2)]{MR3097000}
that every infinite set $E\subseteq \omega$ has an infinite $\Diff$-sparse subset $F\subseteq E$, so 
if we show that every $\cD$-sparse set has small accretions, the proof will be finished.

Let $F\subseteq \omega$ be an infinite $\Diff$-sparse set and $K\subseteq\omega$ be a finite set.
It is known \cite[Proposition~4.3(1)]{MR3097000} that then $F-n\in \Diff$ for every $n<\min F$, and consequently,
$\{a-b:a\in F\setminus K,b\in F\cap K\}\cap \omega\in \cD$. 
Thus, 
$\Delta(F)\setminus \Delta(F\setminus K)= \{a-b:a\in F\cap K, b\in F, a>b\}\cup (\{a-b:a\in F\setminus K,b\in F\cap K\}\cap \omega)\in \cD$ as a finite union of sets from $\Diff$.
\end{proof}


\section{Topological complexity of partition regular operations}

If $\Lambda$ is a countable infinite set, then we consider $2^\Lambda=\{0,1\}^\Lambda$ as a product (with the product topology) of countably many copies of a discrete topological space $\{0,1\}$. Since $2^\Lambda$ is a Polish space \cite[p.~13]{MR1321597} and  $[\Lambda]^\omega$ is a $G_\delta$ subset of $2^\Lambda$, we obtain that  $[\Lambda]^\omega$ is a Polish space as well \cite[Theorem~3.11]{MR1321597}. 
In particular, if $\Lambda$ and $\Omega$ are countable infinite and $\cF\subseteq[\Omega]^\omega$, we say that a partition regular function  $\rho:\cF\to [\Lambda]^\omega$
is \emph{continuous} if $\rho$ is a continuous function from a topological subspace $\cF$ into a topological space $[\Lambda]^\omega$.

By identifying subsets of $\Lambda$  with their characteristic functions,
we equip $\cP(\Lambda)$ with the topology of the space $2^\Lambda$ and therefore
we can assign topological notions to ideals on $\Lambda$.
In particular, an ideal $\I$ is \emph{Borel} (\emph{analytic}, \emph{coanalytic}, resp.) if $\I$ is a Borel  (analytic, coanalytic, resp.) subset of $2^\Lambda$. Recall, a set $A\subseteq X$ is \emph{analytic} if there is a Polish space $Y$ and a Borel set $B\subseteq X\times Y$ such that $A$ is a projection of $B$ onto the first coordinate \cite[Exercise~14.3]{MR1321597}, and a set $C\subseteq X$ is coanalytic if $X\setminus C$ is an analytic set.

\begin{proposition}
\label{prop:continuous-rho-gives-coanalytic-ideal}
    If a partition regular function $\rho:\cF\to[\Lambda]^\omega$ (with $\cF\subseteq[\Omega]^\omega$) is continuous and $\cF$ is a closed subset of $[\Omega]^\omega$, then the ideal $\I_\rho$ is coanalytic.
\end{proposition}

\begin{proof}
We will show that $\I_\rho^+ = \cP(\Lambda)\setminus \I_\rho$ is an analytic set.
Let $B = \{(A,F)\in \cP(\Lambda)\times \cF: \rho(F)\subseteq A\}.$
Since $B\subseteq \cP(\Lambda)\times [\Omega]^\omega$ and $\I_\rho^+$ is a projection of $B$ onto the first coordinate, we only need to show that $B$ is a Borel set. 
It suffices to show that $C = (\cP(\Lambda)\times [\Omega]^\omega) \setminus B$ is an open set, since
$$B= ((\cP(\Lambda)\times [\Omega]^\omega) \setminus C)\cap(\cP(\Lambda)\times \cF).$$

Let $(A,F)\in C$. We have two cases: (1) $F\notin \cF$ or (2) $F\in \cF$.

\emph{Case (1)}. 
Since $\cF$ is closed, there is an open set $U\subseteq [\Omega]^\omega$ with $F\in U$ and $U\cap \cF=\emptyset$.
Then $W = \cP(\Lambda) \times U$ is open and $(A,F)\in W \subseteq C$.

\emph{Case (2)}. 
Since $\rho(F)\not\subseteq A$, there is $a\in \rho(F)\setminus A$.
Let $V=\{D\in \cP(\Lambda): a\in D\}$. Then $V$ is an open and closed set, $A\notin V$ and $\rho(F)\in V$.
Since $\rho$ is continuous at the point $F$, there is an open set $U\subseteq [\Omega]^\omega$ such that $F\in U$ and $\rho[U]\subseteq V $.
Then $W = (\cP(\Lambda)\setminus V)\times U$ is 
open and $(A,F)\in W \subseteq C$.
\end{proof}

\begin{proposition}\ 
\label{prop:top-complex-of-ideals}
    \begin{enumerate}
    \item The ideals $\I_{1/n}$ and $\vdW$ are $F_\sigma$.\label{prop:top-complex-of-ideals:summable-vdW}
        \item The functions $\FS$ and  $r$ are continuous.\label{prop:top-complex-of-ideals:continuus} 
        \item The function $\Delta$ is not continuous. In fact, the function $\Delta$ is discontinuous at every  point  $A$ such that $\Delta(A)\neq\omega$. \label{prop:notcontinuus}

\item 
If $\cL = \{A\in [\omega]^{\omega}: \forall n\in\omega\,(e_A(n+1)-e_A(n)>e_A(n))\}$ where 
$e_A : \omega \to A$ is the increasing enumeration of a set $A\subseteq\omega$, then
  $\cI_{\Delta}=\cI_{\Delta\restriction \cL}$, $\cL$ is closed
  and
  $\Delta\restriction \cL$ is continuous.\label{prop:top-complex-of-ideals:restriction-of-Delta} 
      
        \item The ideals $\Hindman$, $\Ramsey$ and $\Diff$ are coanalytic.\label{prop:top-complex-of-ideals:Hindman-Ramsey-Diff}
    \end{enumerate}

\end{proposition}

\begin{proof}

(\ref{prop:top-complex-of-ideals:summable-vdW})
It is known that $\I_{1/n}$ and $\vdW$ are $F_\sigma$ \cite[Example~1.5]{MR1124539}, \cite[Example~4.12]{MR4572258}.

    (\ref{prop:top-complex-of-ideals:continuus})
    \emph{Case of $\FS$}.
    Let $D\in [\omega]^\omega$ and let $U$ be an open basic neighborhood of $\FS(D)$. Then there exists a finite set $G\subseteq\omega$ such that $U = \{B\in [\omega]^\omega:  B\cap \{0,1,\dots,\max G\} = G\}$.
Let  $F =  D \cap \{0,1,\dots,\max G\}$.
Then  $V = \{A\in [\omega]^\omega: A\cap \{0,1,\dots,\max G\}=F\}$
is an open neighborhood of $D$
and 
$\FS[V]\subseteq U$.

    \emph{Case of $r$}. Let $D\in [\omega]^\omega$ and  let $U$ be an open basic neighborhood of $[D]^2$. There exists a finite set $G\subseteq[\omega]^2$ such that $U = \{B\in \left[[\omega]^2\right]^\omega:  B\cap [N]^2 = G\}$, where $N=\max\{\max\{p,q\}:\{p,q\}\in G\}$.
Then  $V = \{A\in [\omega]^\omega: A\cap N=D\}$
is an open neighborhood of $D$ and $r[V]\subseteq U$.

    (\ref{prop:notcontinuus}) Let $A\subseteq \omega$ be such that  $b\notin \Delta(A)$ for some $b\in \omega$. Then  $U=\{B\subseteq \omega: b\notin B\}$ is an open neighborhood of $\Delta(A)$. Let $V$ be an open basic neighborhood of $A$. There is $N\in\omega$ such that $V=\{C\subseteq \omega: C\cap N=A\cap N\}$. Then $C=(A\cap N)\cup (\omega\setminus N)\in V$ and $\Delta(C)=\omega\notin U$. Hence the function $\Delta$ is  discontinuous at the point  $A$.

(\ref{prop:top-complex-of-ideals:restriction-of-Delta})
It is obvious that $\cI_{\Delta}=\cI_{\Delta\restriction \cL}$. To show that $\cL$ is closed, notice that  $[\omega]^\omega\setminus\cL$ is open as for each $A\in[\omega]^\omega\setminus\cL$ there is $n\in\omega$ such that $e_A(n+1)-e_A(n)\leq e_A(n)$ and $U=\{C\in[\omega]^\omega:C\cap(e_A(n+1)+1)=A\cap(e_A(n+1)+1)\}$ is an open neighborhood of $A$ disjoint with $\cL$.

Below we show that 
  $\Delta\restriction \cL$ is continuous. 
 Let $A\in \cL$.
  We are going to show that  the function $\Delta\restriction \cL$ is continuous at the point  $A$.
Let $U$ be a neighborhood of $\Delta(A)$. Without loss of generality, we can assume that there is $N\in \omega$ such that $U=\{B\in [\omega]^{\omega}: B\cap N=\Delta(A)\cap N\}$.
  There  exists $M\in \omega$ such that $e_A(M)>N$.
 Then $V=\{C\in [\omega]^{\omega}: C\cap (e_A(M)+1)=A\cap (e_A(M)+1)\}$ is an open  neighborhood of $A$. 
Once we show that $\Delta[V\cap\cL]\subseteq U$, the proof will be finished.
Let $C\in V\cap\cL$.
Since $A,C\in \cL$, we obtain 
$\Delta(C)\cap (e_A(M)+1)=\Delta(C\cap (e_A(M)+1))=\Delta(A\cap (e_A(M)+1))= \Delta(A)\cap (e_A(M)+1)$.
But $N<e_A(M)$, hence 
    $\Delta(C)\cap N=\Delta(A)\cap N$
    and consequently 
 $\Delta(C)\in U$.

    (\ref{prop:top-complex-of-ideals:Hindman-Ramsey-Diff}) 
    It is known that $\Hindman$ and $\Ramsey$ are coanalytic \cite[Example~4.11]{MR4572258},\cite[Lemma~1.6.24]{alcantara-phd-thesis} (but it also follows from item (\ref{prop:top-complex-of-ideals:continuus}) and Proposition~\ref{prop:continuous-rho-gives-coanalytic-ideal}).
    It follows from item (\ref{prop:top-complex-of-ideals:restriction-of-Delta}) and Proposition \ref{prop:continuous-rho-gives-coanalytic-ideal} that $\Diff$ is coanalytic.
\end{proof}


\section{P-like properties}

\subsection{P-like properties of ideals}

For $A,B\subseteq\Lambda$, we write $A\subseteq^* B$ if 
there is a finite set $K\subseteq \Lambda$ with $A\setminus K \subseteq B$.

Let us recall definitions of P-like properties of ideals that are considered in the literature \cite[p.~2030]{MR3692233}. 
An ideal $\I$ on $\Lambda$ is 
\begin{itemize}

\item 
\emph{$P^-(\Lambda)$} if 
for every $\subseteq$-decreasing sequence $A_n\in \I^+$ with $A_0=\Lambda$ and $A_n\setminus A_{n+1}\in \I$ for each $n\in\omega$ there exists $B\in\I^+$ such that $B \subseteq^*  A_n$  for each $n\in\omega$;
\item 
\emph{$P^-$} if 
for every $\subseteq$-decreasing sequence $A_n\in \I^+$ with $A_n\setminus A_{n+1}\in \I$ for each $n\in\omega$ there exists $B\in\I^+$ such that $B \subseteq^*  A_n$  for each $n\in\omega$;

\item 
\emph{$P^+$} if for every $\subseteq$-decreasing sequence $A_n\in \I^+$ there exists $B\in\I^+$ such that $B \subseteq^* A_n$ for each $n\in\omega$.

\end{itemize}

The following proposition reveals some implications between  P-like properties and provides equivalent forms of the properties $P^-(\Lambda)$ and  $P^-$ that were  considered in the literature \cite{MR4584767} under  the names \emph{weak P-ideals} and \emph{hereditary weak P-ideals}, where the author used them for in-depth research on  $\I$-spaces.

\begin{proposition}\ 
\label{prop:Plike-basic-properties-for-ideals}
\label{prop:weak-Pplus-extends-to-Pplus}
\label{prop:weak-P-equivalent-definitions}
Let  $\I$ be an ideal on an infinite countable set $\Lambda$.
\begin{enumerate}
    \item $\I\text{ is }P^+\implies\I\text{ is }P^-\implies\I\text{ is }P^-(\Lambda)$.\label{prop:Plike-basic-properties-for-ideals:Pplus-implies-weakPplus}\label{prop:Plike-basic-properties-for-ideals:weakPplus-implies-Pminus}\label{prop:Plike-basic-properties-for-ideals:Pminus-implies-Pminus-of-Lambda}\label{prop:Plike-basic-properties-for-ideals:items}
\item The implications from item (\ref{prop:Plike-basic-properties-for-ideals:items}) cannot be reversed.\label{prop:Plike-basic-properties-for-ideals:reversed-items}

\item 
 The following conditions are equivalent.\label{prop:weak-P-equivalent-definitions:charakterization}
\begin{enumerate}
    \item $\I$ is $P^-(\Lambda)$ ($\I$ is  $P^-$, resp.).
\item For every partition $\cA$ of $\Lambda$ (of any set $C\in \I^+$, resp.) into sets from $\I$  
there exists  $B\in \I^+$ such that $B\subseteq\Lambda$ ($B\subseteq C$, resp.) and $B\cap A$ is finite for each $A\in \cA$.
\item $\I$ is a \emph{weak P-ideal} (\emph{hereditary weak P-ideal}, resp.) i.e. 
for every countable family $\cA\subseteq \I$ of subsets of $\Lambda$ (subsets of any $C\in \I^+$, resp.) there exists  $B\in \I^+$ such that $B\subseteq \Lambda$ ($B\subseteq C$, resp.)  and $B\cap A$ is finite for each $A\in \cA$.\label{prop:weak-P-equivalent-definitions:family-of-sets}
\end{enumerate}

\end{enumerate}
\end{proposition}

\begin{proof}
(\ref{prop:Plike-basic-properties-for-ideals:Pplus-implies-weakPplus})
Straightforward.

(\ref{prop:Plike-basic-properties-for-ideals:reversed-items})
The ideal 
$\Fin\oplus\Fin^2 $ 
is $P^-(\omega\sqcup\omega^2)$ (the set $B=\omega\sqcup\emptyset$ works for every sequence) but not $P^-$ (as witnessed by the sets $A_n=\emptyset\sqcup((\omega\setminus n)\times\omega)$).

Below we show an example of a $P^-$ ideal that is  not $P^+$.
  For a set $A\subseteq\omega$, we  define the asymptotic density of $A$ by 
 $\overline{d}(A)=\limsup_{n\to\infty} |A\cap n|/n$.
 Then the ideal $\I_d=\{A\subseteq\omega:\overline{d}(A)=0\}$ is $P^-$ (see e.g.~\cite[Corollary~1.1]{MR54000}). Now we show that  $\I_d$  is not $P^+$.
Take a decreasing sequence $B_n\subseteq \omega$ such that $0<\overline{d}(B_n)< 1/n$ for each $n\in \omega$.
If  $C\subseteq \omega$ is such that $C\subseteq^* B_n$ for all $n\in\omega$, then $\overline{d}(C)\leq \overline{d}(B_n)\to0$ as $n\to\infty$. Hence $C\in \I_d$. This shows that $\I_d$ is not $P^+$.

(\ref{prop:weak-P-equivalent-definitions:charakterization})
Straightforward.
\end{proof}

There are known relationships between topological complexity and P-like properties.

\begin{theorem}[{\cite[Proposition~4.9]{MR4584767},\cite[Lemma 1.2]{MR748847},\cite[Theorem~3.7]{MR3692233}}]\ 
\label{thm:Plike-properties-for-definable-ideals}
\begin{enumerate}
\item Each $G_{\delta\sigma\delta}$ (in particular, $F_{\sigma\delta}$) ideal is $P^-$ (hence $P^-(\Lambda)$).\label{thm:Plike-properties-for-definable-ideals:Fsigmadelta}
  
\item Each $F_\sigma$ ideal is $P^+$ (hence $P^-$ and $P^-(\Lambda)$). \label{thm:Plike-properties-for-definable-ideals:Fsigma}

\item  If $\I$ is an analytic ideal, then the following conditions are equivalent.
\begin{enumerate}
\item There exists a $P^+$ ideal $\J$ with  $\I\subseteq \J$.
\item There exists an $F_\sigma$ ideal $\K$ with  $\I\subseteq \K$.
\end{enumerate}

\end{enumerate}    
\end{theorem}

\subsection{P-like properties of partition regular operations}

\begin{definition}
Let  $\rho:\cF\to[\Lambda]^\omega$ be partition regular.
 For  sets $F\in \cF$ and $B\subseteq\Lambda$,  we write 
$\rho(F)\subseteq^\rho B$
if 
there is a finite set $K\subseteq \Omega$ with $\rho(F\setminus K)\subseteq B$.
\end{definition}

\begin{remark}
We want to stress here that the relation ``$\rho(F)\subseteq^\rho B$'' is in fact a relation between $F$ and $B$ and not between $\rho(F)$ and $B$ because it can happen that $\rho(F)=\rho(G)$ and $\rho(F)\subseteq^\rho B$ but $\rho(G)\not\subseteq^\rho B$.   
We decided that we write $\rho(F)\subseteq^\rho B$ instead of $F\subseteq^\rho B$ as the former  seems more natural for us. 
The same remark applies to other notions involving ``$\rho(F)$'' we defined earlier or we define later  (e.g.~Definitions~\ref{def:P-like-rho} and \ref{def:rho-convergence}). 
\end{remark}

The following properties will prove useful in the studies of classes of sequentially compact spaces defined with the aid of partition regular functions.

\begin{definition}
\label{def:P-like-rho}
Let  $\rho:\cF\to[\Lambda]^\omega$ be partition regular.
 We say that $\rho$ is 
\begin{enumerate}

\item 
\emph{$P^-(\Lambda)$} if for every $\subseteq$-decreasing sequence $A_n\in \I_{\rho}^+$ with $A_0=\Lambda$ and $A_n\setminus A_{n+1}\in \I_{\rho}$ for each $n\in\omega$ 
there exists $F\in\cF$ such that $\rho(F)\subseteq^\rho A_n$
for each $n\in\omega$;

\item 
\emph{$P^-$} if for every $\subseteq$-decreasing sequence $A_n\in \I_{\rho}^+$ with $A_n\setminus A_{n+1}\in \I_{\rho}$ for each $n\in\omega$ 
there exists $F\in\cF$ such that $\rho(F)\subseteq^\rho A_n$
for each $n\in\omega$;

\item 
\emph{$P^+$} if for every $\subseteq$-decreasing sequence $A_n\in \I_{\rho}^+$  
there exists $F\in\cF$ such that $\rho(F)\subseteq^\rho A_n$
for each $n\in\omega$;

\item  
\emph{weak $P^+$}  if
for every $E\in \cF$ there exists $F\in \cF$ such that $\rho(F)\subseteq \rho(E)$ and for every  sequence $\{F_n: n\in\omega\}\subseteq \cF$ such that $\rho(F) \supseteq \rho(F_n)\supseteq \rho(F_{n+1})$ for each $n\in\omega$ 
there exists $G\in\cF$ such that  $\rho(G)\subseteq^\rho \rho(F_n)$
for each $n\in\omega$.

\end{enumerate}
\end{definition}

The following result reveals basic properties of the above defined notions and their connections with P-like properties of ideals.

\begin{proposition}
\label{prop:Plike-basic-properties}
\label{prop:properties-of-ideal-versus-rho}
Let  $\rho:\cF\to[\Lambda]^\omega$ be partition regular with $\cF\subseteq[\Omega]^\omega$.
Let $\I$ be an ideal on $\Lambda$.
\begin{enumerate}

\item \label{prop:Plike-basic-properties:implications}
$\rho\text{ is }P^+\implies\rho\text{ is weak }P^+\implies\rho\text{ is }P^-\implies\rho\text{ is }P^-(\Lambda)$.
\label{prop:Plike-basic-properties:Pplus-implies-weakPplus}\label{prop:Plike-basic-properties:weakPplus-implies-Pminus}\label{prop:Plike-basic-properties:Pminus-implies-Pminus-of-Lambda}

\item 
$\I$ is $P^+$ $\iff$  $\rho_\I$ is $P^+$, for every ideal $\I$. Similar equivalences hold for $P^-$ and $P^-(\Lambda)$, resp.\label{prop:properties-of-ideal-versus-rho:ideal-rho}

\item  
The implications from item (\ref{prop:Plike-basic-properties:implications}) cannot be reversed.\label{prop:Plike-basic-properties:reversed-implications}

\item If $\I_\rho$ is $P^-(\Lambda)$ ($P^-$, $P^+$, resp.), then $\rho$ is $P^-(\Lambda)$ ($P^-$, $P^+$, resp.).\label{prop:Plike-basic-properties:Pminus-ideal-implies-rho}

\item  
The implications from item (\ref{prop:Plike-basic-properties:Pminus-ideal-implies-rho}) cannot be reversed in case of $P^-(\Lambda)$ and $P^-$ properties.\label{prop:Plike-basic-properties:Pminus-ideal-implies-rho:reversed-implications}
\end{enumerate}

\end{proposition}

\begin{proof}

(\ref{prop:Plike-basic-properties:implications}) 
Below we only show that if $\rho$ is  weak $P^+$ then it is $P^-$ since other implications are straightforward.

Let 
 $A_n\in \I_{\rho}^+$ be a $\subseteq$-decreasing sequence with $A_n\setminus A_{n+1}\in \I_{\rho}$ for each $n\in\omega$. Since $A_0\in\I_{\rho}^+$, there is $E\in\cF$ such that $\rho(E)\subseteq A_0$. Using the fact that $\rho$ is weak $P^+$ we can find $F\in\cF$ with $\rho(F)\subseteq\rho(E)$ and such as in the definition of weak $P^+$ property. 
 
We will show that there is a sequence $\{F_n: n\in\omega\}\subseteq \cF$ such that $\rho(F_0)\subseteq \rho(F)$ and $\rho(F_{n+1})\subseteq \rho(F_n)\cap A_{n+1}$ for each $n\in\omega$. Indeed, since $\rho(F)\subseteq A_0$, it suffices to put $F_0=F$. Suppose now that $F_i$ have been constructed for $i\leq n$. 
 Since $\rho(F_n)\cap A_{n+1}  = \rho(F_n) \setminus (A_n\setminus A_{n+1}) \in \I^+_\rho$, there is $F_{n+1}\in \cF$ with  $\rho(F_{n+1})\subseteq\rho(F_n)\cap A_{n+1}$.
 
Since $F$ is as in the definition of weak $P^+$ property, there  exists $G\in\cF$ such that $\rho(G)\subseteq^\rho \rho(F_n)$
for each $n\in\omega$.
Thus,  $\rho(G)\subseteq^\rho  A_n$
for each $n\in\omega$.

 (\ref{prop:properties-of-ideal-versus-rho:ideal-rho}) Straightforward. 

(\ref{prop:Plike-basic-properties:reversed-implications})
The cases of the second and third implications follow from Proposition~\ref{prop:Plike-basic-properties-for-ideals}(\ref{prop:Plike-basic-properties-for-ideals:reversed-items}) 
and item (\ref{prop:properties-of-ideal-versus-rho:ideal-rho}), where the proof of the fact that  $\rho_{\I_d}$ is not weak $P^+$ is just a slight modification of the proof that $\I_d$ is not $P^+$.

Now we show  that the first implication cannot be reversed. Consider the ideal 
$\I= \{A\subseteq\omega\times\omega: \text{$A\cap (\{n\}\times\omega)$ is finite for every $n\in\omega$}\}$. Then $\I$ is not $P^+$ as witnessed by $A_n=(\omega\setminus n)\times\omega$, so $\rho_\I$ is not $P^+$ (by item (\ref{prop:properties-of-ideal-versus-rho:ideal-rho})). However, we will show that $\rho_\I$ is weak $P^+$. Let $E\in\I^+$. Then there is $n\in\omega$ such that $F=E\cap (\{n\}\times\omega)$ is infinite. Then $F\in\I^+$ and it is easy to see that if $F_n\in\I^+$ are such that $F\supseteq F_n\supseteq F_{n+1}$ then one can pick $x_n\in F_n$ for each $n\in\omega$ and $G=\{x_n:n\in\omega\}\in \I^+$ is such that $G\setminus \{x_i:i<n\}\subseteq F_n$ for all $n\in\omega$.

(\ref{prop:Plike-basic-properties:Pminus-ideal-implies-rho})
Proofs in all cases are very similar, so we only present  a proof  for the property $P^-(\Lambda)$.
Let $A_n\in \I_{\rho}^+$ be a $\subseteq$-decreasing sequence with $A_0=\Lambda$ and $A_n\setminus A_{n+1}\in \I_{\rho}$ for each $n\in\omega$. Since $\I_{\rho}$ is $P^-(\Lambda)$, there is $B\not\in \I_{\rho}$ such that for every $n\in \omega$ one can find a finite set $K_n\subseteq \Omega$ such that $B\setminus K_n\subseteq A_n$. From the fact that $B\not\in \I_{\rho}$, there is $F\in \cF$ such that $\rho(F)\subseteq B$. 
Using Proposition~\ref{prop:finite-support}, we can find  $E\in \cF$  such that $E\subseteq F$ and  for every $K_n$ there exists a finite set $L_n\subseteq \Omega$ such that $\rho(E\setminus L_n)\subseteq \rho(E)\setminus K_n$. Then $\rho(E\setminus L_n)\subseteq \rho(E)\setminus K_n \subseteq B\setminus K_n\subseteq A_n$, so $\rho$ is $P^-(\Lambda)$.

(\ref{prop:Plike-basic-properties:Pminus-ideal-implies-rho:reversed-implications})
In Proposition~\ref{prop:Plike-properties-for-known-rho}(\ref{prop:Plike-properties-for-known-rho:FS-r-Delta:Pminus})(\ref{prop:Plike-properties-for-known-rho:Hindman-Ramsey-Diff}) 
we will show that $\rho=\FS$ is weak $P^+$, but $\I_\rho = \Hindman$ is not $P^-(\omega)$.
\end{proof}

We will need the following lemma to show that $FS$, $r$ and $\Delta$ are not $P^+$.

\begin{lemma}\label{lemma:wlP}
Let $\rho:\cF\to[\Lambda]^\omega$ (with $\cF\subseteq[\Omega]^\omega$) be a  partition regular function such that   
there exists a function $\tau:[\Omega]^{<\omega}\to \Lambda$ such that
\begin{enumerate}
        \item $\forall F\in \cF\, \forall \{a,b\}\in [F]^2 \, (\tau\left(\{a,b\}\right)\in \rho(F))$,
        \item $\forall F\in \cF\, \forall c\in \rho(F)\,  \exists S\in [F]^{<\omega}\,(\tau(S)=c)$,
    \item there exists a pairwise disjoint family $\{P_n: n\in \omega\}\subseteq \cF$ such that  the family  $\{\rho(P_n):n\in\omega\}$ is also pairwise disjoint  and 
 the restriction $\tau\restriction \left[\bigcup\{P_n:n\in\omega\}\right]^{<\omega}$
 is one-to-one.\label{lemma:wlP:item-sets-P}
 \end{enumerate}
Then $\rho$ is not $P^+$.
\end{lemma}

\begin{proof}
Let $\{P_n:n\in\omega\}$ be as in item (\ref{lemma:wlP:item-sets-P}) of the lemma. For each $n\in\omega$, we define 
$B_n = \bigcup\{\rho(P_i): i\geq n\}.$
Then $B_n\in \I_{\rho}^+$ and $B\supseteq B_n\supseteq B_{n+1}$ for each $n\in\omega$. 
If we show that there is no 
$G\in \cF$ such that $\rho(G)\subseteq^\rho B_n$ for every $n\in \omega$, the proof will be finished.
Suppose for sake of contradiction that there exists $G\in \cF$ such that for every $n\in \omega$ there exists a finite set $K_n\subseteq \Omega$ with $\rho(G\setminus K_n)\subseteq B_n$.
We have two cases:
\begin{enumerate}
    \item $|G\cap P_{n_0}|=\omega$ for some $n_0\in \omega$,
    \item $|G\cap P_{n}|<\omega$ for all $n\in \omega$.
\end{enumerate}

\emph{Case (1).}  We take distinct  $a,b\in (G\cap P_{n_0})\setminus K_{n_0+1}$. 
Since $a,b\in P_{n_0}\in\cF$, we have $\tau\left(\{a,b\}\right)\in  \rho(P_{n_0})$. 
On the other hand, $a,b\in G\setminus K_{n_0+1}\in \cF$, so $\tau\left(\{a,b\}\right)\in  \rho(G\setminus K_{n_0+1}) \subseteq B_{n_0+1}$. 
Hence, there exists $i\geq n_0+1$ such that $\tau\left(\{a,b\}\right)\in \rho(P_i)$. 
A contradiction with $\rho(P_i)\cap \rho(P_{n_0})=\emptyset$.

\emph{Case (2).} In this case, there exists a strictly increasing sequence $\{k_n : n\in \omega\}$   
such that  we can choose an element $x_{k_n}\in G\cap P_{k_n}$ for each $n\in \omega$. 
Since $x_{k_n}$ are pairwise distinct, there is $N\in \omega$ such that  $x_{k_n}\in G\setminus K_0$ for every $n\geq N$. 
In particular,  $\tau\left(\{x_{k_N},x_{k_{N+1}}\}\right)\in \rho(G\setminus K_0)\subseteq B_0$, and consequently there exists $i\in \omega$ such that $\tau\left(\{x_{k_N},x_{k_{N+1}}\}\right)\in \rho(P_i)$.
Therefore there is a finite set $S\subseteq P_i$ such that $\tau(S)=\tau\left(\{x_{k_N},x_{k_{N+1}}\}\right)$. 
Since $P_n$ are pairwise disjoint and $x_{k_n}\in P_n$,  we obtain that $x_{k_N}\notin P_i$ or $h_{k_N+1}\notin P_i$.
Consequently,  $\{x_{k_N},x_{k_{N+1}}\} \neq S$, so  $\tau\restriction \left[\bigcup\{P_n:n\in\omega\}\right]^{<\omega}$ is not one-to-one, a contradiction. 
\end{proof}

\begin{proposition}\ 
\label{prop:Plike-properties-for-known-rho}
\begin{enumerate}

\label{prop:properties-of-ideal-versus-rho:Pplus}\label{prop:properties-of-ideal-versus-rho:P-like-properties:weakPplus}\label{prop:properties-of-ideal-versus-rho:P-like-properties:Pminus} 

\item The ideals $\vdW$ and $\I_{1/n}$ are $P^+$ (hence, $P^-$ and $P^-(\omega)$) while $\rho_\vdW$ and $\rho_{\I_{1/n}}$ are $P^+$, weak $P^+$, $P^-$ and $P^-(\omega)$. \label{prop:Plike-properties-for-known-rho:vdW}\label{prop:Plike-properties-for-known-rho:summable}\label{prop:Plike-properties-for-known-rho:Fsigma-known}

\item 
If  $\rho\in \{\FS,r,\Delta\}$, then
$\rho$ is not $P^+$,\label{prop:Plike-properties-for-known-rho:FS-r-Delta:Pplus}

\item 
If  $\rho\in \{\FS,r,\Delta\}$, then $\rho$ is   weak $P^+$ (hence $P^-$ and $P^-(\Lambda)$).\label{prop:Plike-properties-for-known-rho:FS-r-Delta:weakPplus}\label{prop:Plike-properties-for-known-rho:FS-r-Delta:Pminus}\label{prop:Plike-properties-for-known-rho:FS-r-Delta:veryweakPplus}

\item 
If  $\I\in \{\Hindman,\Ramsey,\Diff\}$, then $\I$ is  
not $P^-(\Lambda)$ (hence not $P^+$ and not $P^-$).\label{prop:Plike-properties-for-known-rho:Hindman-Ramsey-Diff}

\end{enumerate}
\end{proposition}

\begin{proof}

(\ref{prop:Plike-properties-for-known-rho:Fsigma-known})
It follows from 
Theorem~\ref{thm:Plike-properties-for-definable-ideals}(\ref{thm:Plike-properties-for-definable-ideals:Fsigma}) and 
the fact that  $\cW$ and $\I_{1/n}$ are $F_{\sigma}$ ideals (see Proposition~\ref{prop:top-complex-of-ideals}(\ref{prop:top-complex-of-ideals:summable-vdW})). The ``hence'' part follows from Proposition~\ref{prop:Plike-basic-properties}.

(\ref{prop:Plike-properties-for-known-rho:FS-r-Delta:Pplus})
Below we show that $\rho$ is not $P^+$   separately for each $\rho$.

\emph{Case of $\rho=FS$.}
We define a function $\tau:[\omega]^{<\omega}\to\omega$ by $\tau(S) =\sum_{i\in S}i$. Then we take an infinite sparse set $P$ and a partition $\{P_n:n\in \omega\}$ of $P$ into infinite sets. 
Lemma~\ref{lemma:wlP} shows that $\FS$ is not $P^+$.

\emph{Case of $\rho=r$.}
We define a function $\tau:[\omega]^{<\omega}\to[\omega]^2$ by $\tau(\{a,b\}) = \{a,b\}$ for distinct $a>b$ and $\tau(S) = \{0,1\}$ otherwise. Then we take a partition $\{P_n:n\in \omega\}$ of $\omega$ into infinite sets. 
Lemma~\ref{lemma:wlP} shows that $r$ is not $P^+$.

\emph{Case of $\rho=\Delta$.}
We define a function $\tau:[\omega]^{<\omega}\to\omega$ by $\tau(\{a,b\}) = a-b$ for distinct $a>b$ and $\tau(S) = 0$ otherwise. Then we take an infinite $\Diff$-sparse set $P$ and a partition $\{P_n:n\in \omega\}$ of $P$ into infinite sets. 
Lemma~\ref{lemma:wlP} shows that $\Delta$ is not $P^+$.

(\ref{prop:Plike-properties-for-known-rho:FS-r-Delta:weakPplus})
The ``hence'' part follows from Proposition~\ref{prop:Plike-basic-properties}(\ref{prop:Plike-basic-properties:weakPplus-implies-Pminus}). Below we show that $\rho$ is weak $P^+$ separately for each $\rho$.

\emph{Case of $\rho=\FS$.} 
It is proved in  \cite[Lemma~2.3]{MR4356195} (see also \cite[Example~2.9(2)]{MR3461181}).

\emph{Case of $\rho=r$.} 
For any  $E\in[\omega]^\omega$
we take $F=E$. Let $F_n\in [\omega]^\omega$ be such that $[F]^2\supseteq [F_n]^2\supseteq [F_{n+1}]^2$ for each $n\in\omega$.
We pick $x_n\in F_n\setminus\{x_i:i<n\}$ for each $n\in\omega$. Then $G=\{x_n:n\in\omega\}\in [\omega]^\omega$ and $[G]^2\subseteq^r[F_n]^2$ for each $n\in\omega$.

\emph{Case of $\rho=\Delta$.} 
Fix any $F\in[\omega]^\omega$. Inductively pick a sequence $(x_i)_{i\in \omega}\subseteq\omega$ such that $x_i\in F$, $x_i<x_{i+1}$ and $x_{i+1}-x_i>x_i-x_0$ for all $i\in \omega$. Let $E=\{x_i:\ i\in\omega\}\in[F]^{\omega}$. 

Define $a_i=x_{i+1}-x_i$ for all $i\in\omega$ and observe that $a_i=x_{i+1}-x_i>x_i-x_0=\sum_{j<i}a_j$. Put $A=\{a_i:\ i\in\omega\}$. By \cite[proof of Lemma 2.2]{MR4356195} 
the set $A$ is very sparse, i.e. $A$ is sparse and 
 if $\alpha_{A}(x)\cap \alpha_{A}(y)\neq\emptyset$ then $x+y\notin \FS(A)$. 
Note that $\Delta(E)=\{\sum_{i\in I}a_i:\ I\text{ is a finite interval in }\omega\}\subseteq \FS(A)$. 

Observe that if $\Delta(\{y_n:\ n\in\omega\})\subseteq\Delta(E)$, where $y_n<y_{n+1}$ for all $n\in \omega$, then there is a partition of $\omega$ into finite intervals $(I_n)_{n\in \omega}$ such that $\max I_n<\min I_{n+1}$ and $y_{n+1}-y_n=\sum_{i\in I_n}a_i$. Indeed, as $y_{n+1}-y_n\in \Delta(\{y_n:\ n\in\omega\})\subseteq\Delta(E)\subseteq\FS(A)$, for each $n\in \omega$ there is a finite interval $I_n$ such that $y_{n+1}-y_n=\sum_{i\in I_n}a_i$ (because $A$ is sparse, we get $I_n=\alpha_{A}(y_{n+1}-y_n)$). We need to show that the intervals $I_n$ are pairwise disjoint and cover $\omega$. Suppose first that $\sup I_n+1<\inf I_{n+1}$ for some $n\in \omega$. Then $y_{n+2}-y_n=(y_{n+2}-y_{n+1})+(y_{n+1}-y_{n})=\sum_{i\in I_n\cup I_{n+1}}a_i$. On the other hand, $y_{n+2}-y_n\in\Delta(\{y_n:\ n\in\omega\})\subseteq\Delta(E)$, so $y_{n+2}-y_n=\sum_{i\in I}a_i$ for some interval $I$. This contradicts uniqueness of $\alpha_{A}(y_{n+2}-y_n)$ (because $A$ is sparse). Suppose now that $I_n\cap I_{n+1}\neq\emptyset$. Then $y_{n+2}-y_n=(y_{n+2}-y_{n+1})+(y_{n+1}-y_{n})\notin\FS(A)$ (because $A$ is very sparse), which contradicts $\Delta(\{y_n:\ n\in\omega\})\subseteq\Delta(E)\subseteq FS(A)$.

Fix any sequence $(F_k)_{k\in \omega}\subseteq[\omega]^\omega$ such that $\Delta(F_{k+1})\subseteq\Delta(F_k)\subseteq\Delta(E)$ for all $k\in \omega$. By the previous paragraph, with each $k\in \omega$ we can associate a partition of $\omega$ into finite intervals $I^k_n$, i.e., $\Delta(F_k)=\{\sum_{i\in I}a_i:\ I=I^k_j\cup I^k_{j+1}\cup \ldots\cup I^k_{j'}\text{ for some }j<j'\}$. 

Observe that actually for each $n,k\in\omega$ we have that $I^{k+1}_n=\bigcup_{i\in I}I^k_i$ for some interval $I$. Indeed, otherwise for some $n,k\in\omega$ we would have $x=\sum_{i\in I^{k+1}_n}a_i\in\Delta(F_{k+1})\subseteq\Delta(F_k)$, so $x=\sum_{i\in I}a_i$ for some $I=I^k_j\cup I^k_{j+1}\cup \ldots\cup I^k_{j'}$, which contradicts that $A$ is sparse.

Inductively pick a sequence $(n_k)_{k\in \omega}\subseteq \omega$ such that for each $k\in \omega$ we have $n_{k+1}>n_k$ (so also $a_{n_{k+1}}>a_{n_k}$) and $n_k=\min I^k_j$ for some $j\in \omega$. Define $E'=\{x_{n_k}:\ k\in\omega\}$. Notice that $x_{n_{k+1}}-x_{n_k}=\sum_{i\in[n_k,n_{k+1})} a_i$. Then for each $k\in \omega$ we have $\Delta(E'\setminus [0,x_{n_k}))=\Delta(\{x_{n_i}:\ i\geq k\})\subseteq\{\sum_{i\in I}a_i:\ I=I^k_j\cup I^k_{j+1}\cup \ldots\cup I^k_{j'}\text{ for some }j<j'\}=\Delta(F_k)$.

(\ref{prop:Plike-properties-for-known-rho:Hindman-Ramsey-Diff})
The ``hence'' part follows from Proposition~\ref{prop:Plike-basic-properties-for-ideals}. Below we show that $\I$ is not $P^-(\Lambda)$ separately for each $\I$.

\emph{Case of $\I=\Hindman$.} 
Let $A_k=\{2^k(2n+1): n\in\omega\}$ for each $k\in\omega$.
In  \cite[item (2) in the proof of Proposition~1.1]{MR4356195}, the authors showed that $A_k\in \Hindman$ for every $k\in\omega$, whereas in 
\cite[item (1) in the proof of Proposition~1.1]{MR4356195} it is shown that 
for every $B\notin\Hindman$  there is $k\in\omega$ such that $B\cap A_k$ is infinite.
Thus, the family $\{A_k:k\in\omega\}$ witnesses the fact that $\Hindman$ is not $P^-(\omega)$.

\emph{Case of $\I=\Ramsey$.} 
Let $A_n = \{\{k,i\}: i>k\geq n\}$ for every $n\in\omega$. 
Then $A_n\notin\Ramsey$, $A_0=[\omega]^2$ and $A_n\setminus A_{n+1} = \{\{n,i\}: i> n\}\in \Ramsey$.
Suppose, for sake of contradiction, that there is $B\notin\Ramsey$ such that $B\subseteq^* A_n$ for every $n\in\omega$.
Let $H=\{h_n:n\in\omega\}$ be an infinite set such that $[H]^2\subseteq B$ and $h_n<h_{n+1}$ for every $n\in \omega$.
Since $[H]^2\subseteq^*A_{h_1}$, there is a finite set $F$ such that 
$[H]^2\setminus F\subseteq A_{h_1}$.
Since $F$ is finite, there is $k>0$ such that $\{h_0,h_n\}\notin F$ for every $n\geq k$.
Then  $\{\{h_0,h_n\}:n\geq k\}\subseteq [H]^2\setminus F$ and $\{\{h_0,h_n\}:n\geq k\}\cap A_{h_1}=\emptyset$, a contradiction.

\emph{Case of $\I=\Diff$.} 
Let $A_k=\{2^k(2n+1): n\in\omega\}$ for each $k\in\omega$.
In  \cite[item (2) in the proof of Theorem 2.1]{MR4358658}, the author showed that $A_k\in \Diff$ for every $k\in\omega$, whereas in 
\cite[item (1) in the proof of Theorem 2.1]{MR4358658} it is shown that 
for every $B\notin\Diff$  there is $k\in\omega$ such that $B\cap A_k$ is infinite.
Thus, the family $\{A_k:k\in\omega\}$ witnesses the fact that $\Diff$ is not $P^-(\omega)$.
\end{proof}

The following easy observation will be useful in our considerations. 

\begin{proposition}
\label{prop:Pminus-for-rho-equivalent-condition}
If   $\rho:\cF\to[\Lambda]^\omega$ is  partition regular  with $\cF\subseteq[\Omega]^\omega$, then the following conditions are equivalent.
\begin{enumerate}
    \item 
$\rho$ is $P^-$ ($P^-(\Lambda)$, resp.).

\item For every countable family $\cB\subseteq\I_\rho$ with  $\bigcup\cB\notin \I_\rho$ ($\bigcup\cB = \Lambda$, resp.)  there exists $F\in \cF$ such that 
$\rho(F) \subseteq \bigcup\cB$ and
for every finite subfamily $\cC\subseteq \cB$ there is a finite $K\subseteq\Omega$ such that 
$\rho(F\setminus K) \cap \bigcup\cC =\emptyset$. 
\end{enumerate}
\end{proposition}

\begin{proof}
We will assume that $\rho$ is  $P^-$, as the proof in the case of $P^-(\Lambda)$ is similar.

$(1)\implies(2)$.
Let $\cB=\{B_n:n\in\omega\}$, where $\bigcup\cB\notin \I_\rho$ and $B_n\in\I_{\rho}$ for every $n\in \omega$.
For each $n\in\omega$, we define $A_n = \bigcup\cB\setminus \bigcup\{B_i:i<n\}$. 
Since $\rho$ is   $P^-$, there exists $F\in \cF$ such that $\rho(F)\subseteq^\rho A_n$ for each $n\in\omega$.
Let $\cC\subseteq\cB$ be a finite subfamily.
Let  $n\in\omega$ be such that $\cC\subseteq \{B_i:i<n\}$. 
Then $\bigcup\cC \subseteq \bigcup\{B_i:i<n\}$.
Let $K\subseteq\Omega$ be a finite set such that $\rho(F\setminus K)\subseteq A_n$.
Then $\rho(F\setminus K)\cap \bigcup\{B_i:i<n\}=\emptyset$, so 
 $\rho(F\setminus K)\cap \bigcup\cC=\emptyset$.

$(2)\implies(1)$.
Let $A_n\in \I_{\rho}^+$ be such that $A_n\supseteq A_{n+1}$ and  $A_n\setminus A_{n+1}\in \I_{\rho}$ for each $n\in\omega$.
For each $n\in\omega$ we define $B_n=A_n\setminus A_{n+1}$.
Let $\cB=\{B_n:n\in\omega\}$. Then there exists $F\in \cF$ such that 
$\rho(F) \subseteq \bigcup\cB$ and
for every finite subfamily $\cC\subseteq \cB$ there is a finite $K\subseteq\Omega$ such that 
$\rho(F\setminus K) \cap \bigcup\cC =\emptyset$. 
Thus for any $n\in\omega$, we find a finite set $K\subseteq\Omega$ such that 
$\rho(F\setminus K) \cap \bigcup\{B_i:i<n\} =\emptyset$. 
Hence $\rho(F\setminus K)\subseteq \bigcup\{B_i:i\geq n\}=A_n$,
so $\rho(F)\subseteq^\rho A_n$.
\end{proof}


\section{Kat\v{e}tov order}


\subsection{Kat\v{e}tov order between ideals}

We say that an \emph{ideal $\I_1$ on $\Lambda_1$
is above an ideal $\I_2$ on $\Lambda_2$ in the Kat\v{e}tov order} (in short: $\I_2\leq_K \I_1$) \cite{MR250257} if 
 there exists  a function $\phi:\Lambda_1\to \Lambda_2$ 
 such that $\phi[A]\notin \I_2$ for each $A\notin \I_1$.
If $\Lambda_1=\Lambda_2$ and  $\I_2\subseteq \I_1$, then obviously the identity function on $\Lambda_1$ witnesses that $\I_2\leq_K\I_1$.

There are known relationships between Kat\v{e}tov order, P-like properties and topological complexity.

\begin{proposition}[{\cite[Theorem~3.8]{MR3692233}}]
\label{prop:Pminus-versus-FinSQRD-for-ideal}
\label{prop:Pminus-versus-FinSQRD-for-rho}
Let  $\I$ be an ideal on $\Lambda$.

\begin{enumerate}

\item $\I$ is $P^-(\Lambda)$ $\iff$ $\Fin^2 \not\leq_K \I$.\label{prop:Pminus-versus-FinSQRD-for-ideal:Pminus-Lambda}

\item $\I$ is  $P^-$ $\iff$ $\Fin^2\not\leq_K \I\restriction A$ for every $A\in \I^+$.\label{prop:Pminus-versus-FinSQRD-for-ideal:Pminus}

\end{enumerate}

\end{proposition}

\begin{proposition}\ 
\label{prop:FinSQRD-below-FS-r-Delta}
\label{prop:FinSQRD-not-below-Fsigma}

\begin{enumerate}
\item $\Fin^2\leq_K\I$ for $\I\in\{\Diff,\Hindman,\Ramsey\}$.\label{prop:FinSQRD-below-FS-r-Delta:ideal}\label{prop:FinSQRD-below-FS-r-Delta:Delta-ideal}\label{prop:FinSQRD-below-FS-r-Delta:FS-ideal}\label{prop:FinSQRD-below-FS-r-Delta:r-ideal}

\item If $\I$ is an $G_{\delta\sigma\delta}$ ideal, then 
$\Fin^2 \not\leq_K \I\restriction A$ for every $A\in \I^+$ 
In particular, $\Fin^2 \not\leq_K  \vdW$ and $\Fin^2 \not\leq_K  \I_{1/n}$.\label{prop:FinSQRD-not-below-Fsigma:item}\label{prop:FinSQRD-not-below-Fsigma:vdW-summable}

\end{enumerate}   
\end{proposition}

\begin{proof}
(\ref{prop:FinSQRD-below-FS-r-Delta:ideal})
Using Proposition~\ref{prop:Pminus-versus-FinSQRD-for-ideal}(\ref{prop:Pminus-versus-FinSQRD-for-ideal:Pminus-Lambda}),  we need  to show that $\Diff$, $\Hindman$ and  $\Ramsey$ are not $P^-(\Lambda)$ ideals, but this follows from Proposition~\ref{prop:Plike-properties-for-known-rho}(\ref{prop:Plike-properties-for-known-rho:Hindman-Ramsey-Diff}).
(For $\I =\Ramsey$, this item  was earlier proved by Meza-Alc\'{a}ntara~\cite[Lemma 1.6.25]{alcantara-phd-thesis}.)

(\ref{prop:FinSQRD-not-below-Fsigma:item})
It follows from Theorem~\ref{thm:Plike-properties-for-definable-ideals}(\ref{thm:Plike-properties-for-definable-ideals:Fsigmadelta}) and Propositions~\ref{prop:Pminus-versus-FinSQRD-for-ideal}(\ref{prop:Pminus-versus-FinSQRD-for-ideal:Pminus})
and 
\ref{prop:Plike-properties-for-known-rho}(\ref{prop:Plike-properties-for-known-rho:Fsigma-known}).
\end{proof}


\subsection{Kat\v{e}tov order between partition regular operations}

The following notion will be crucial for showing when a class of sequentially compact spaces defined by $\rho_1$ is contained in a class of sequentially compact spaces defined by $\rho_2$.

\begin{definition}
\label{Katetov-order-for-rho}
Let $\rho_i:\cF_i\to[\Lambda_i]^\omega$ 
be partition regular  (with $\cF_i\subseteq [\Omega_i]^\omega$) for each $i=1,2$.
We say that \emph{$\rho_1$ is above  $\rho_2$ in the Kat\v{e}tov order}
(in short: $\rho_2\leq_K \rho_1$)
if  there is a function $\phi:\Lambda_1\to\Lambda_2$
such that 
$$\forall F_1\in \cF_1\, \exists F_2\in \cF_2\, \forall K_1\in [\Omega_1]^{<\omega}\, \exists K_2\in [\Omega_2]^{<\omega}\,(\rho_2(F_2\setminus K_2)\subseteq\phi[\rho_1(F_1\setminus K_1)])$$
or equivalently: 
$$\forall F_1\in \cF_1\, \exists F_2\in \cF_2\, \forall K_1\in [\Omega_1]^{<\omega}\, (\rho_2(F_2)\subseteq^{\rho_2}\phi[\rho_1(F_1\setminus K_1)]).$$
\end{definition}

The following  proposition reveals some basic properties of this new order on partition regular functions.

\begin{proposition}\ 
\label{prop:basic-prop-of-new-order}
\begin{enumerate}
    \item 
The relation $\leq_K$ is a preorder (a.k.a. quasi order) i.e.~it is reflexive and transitive.\label{prop:basic-prop-of-new-order:preorder}
    \item 
The preorder $\leq_K$ is upward and downward directed.\label{prop:basic-prop-of-new-order:directed}

\item Let $\rho:\cF\to[\Lambda]^\omega$
(with $\cF\subseteq [\Omega]^\omega$) be  partition regular.
\begin{enumerate}
\item 
 $\rho\leq_K \rho\restriction \rho(F)$ for every $F\in \cF$.\label{prop:basic-prop-of-new-order:restrictions}

\item 
$\rho_{\Fin(\Lambda)}\leq_K\rho$.\label{prop:basic-prop-of-new-order:minimal-element}
\end{enumerate}

\end{enumerate}
\end{proposition}

\begin{proof}
(\ref{prop:basic-prop-of-new-order:preorder})
Reflexivity of $\leq_K$ is obvious. To show transitivity, fix $\cF_i\subseteq[\Omega_i]^\omega$ and $\rho_i:\cF_i\to[\Lambda_i]^\omega$, $i=1,2,3$, and suppose that  $\rho_1\leq_K\rho_2$ is witnessed by $f$ and $\rho_2\leq_K\rho_3$ is witnessed by $g$. We claim that $\rho_1\leq_K\rho_3$ is witnessed by $h:\Lambda_3\to\Lambda_1$ given by $h(x)=f(g(x))$ for all $x\in \Lambda_3$.
Let $F_3\in\cF_3$. Then we can find $F_2\in\cF_2$ such that
for every $K\in[\Omega_3]^{<\omega}$ there exists $L_K\in[\Omega_2]^{<\omega}$ such that  $\rho_2(F_2\setminus L_K)\subseteq g\left[\rho_3(F_3\setminus K)\right].$
Then for $F_2$ we can find $F_1\in\cF_1$ such that
for every $L\in[\Omega_2]^{<\omega}$ there exists $M_L\in[\Omega_1]^{<\omega}$ such that $\rho_1(F_1\setminus M_L)\subseteq f\left[\rho_2(F_2\setminus L)\right].$
Now for a given $K\in[\Omega_3]^{<\omega}$ we have
$\rho_1(F_1\setminus M_{L_K})\subseteq f\left[\rho_2(F_2\setminus L_K)\right]\subseteq f\left[g\left[\rho_3(F_3\setminus K)\right]\right]=h\left[\rho_3(F_3\setminus K)\right],$
so the proof is finished.

(\ref{prop:basic-prop-of-new-order:directed})
Let $\rho_i:\cF_i\to[\Lambda_i]^\omega$ with $\cF_i\subseteq[\Omega_i]^\omega$ be partition regular for $i=0,1$. 
We define the following partition regular functions 
$\pi : \{F_0\times F_1: F_0\in\cF_0, F_1\in\cF_1\}\to [\Lambda_0\times \Lambda_1]^{\omega}$
by
$\pi (F_0\times F_1)=\rho_0(F_0)\times \rho_1(F_1)$
and 
$\sigma: \cF_0\oplus\cF_1 \to [\Lambda_0\oplus\Lambda_1]^{\omega}$ by 
$\sigma ((F_0\times\{0\}) \cup (F_1\times\{1\}))= (\rho_0(F_0)\times\{0\}) \cup (\rho_1(F_1)\times \{1\}).$

Then 
$\sigma \leq_K \rho_i$ ($i=0,1$) is witnessed by a function $\phi_i:\Lambda_i\to \Lambda_0\oplus\Lambda_1$ given by $\phi_i(x) = (x,i)$,
whereas 
$\rho_i\leq_K \pi$ ($i=0,1$) is witness by 
a function $\psi_i:\Lambda_0\times \Lambda_1\to \Lambda_i$ given by $\psi_i(x_0,x_1) = x_i$. 

(\ref{prop:basic-prop-of-new-order:restrictions})
Let $F\in \cF$. We claim that $\phi: \rho(F)\to\Lambda$ given by $\phi(\lambda)=\lambda$ is a witness for $\rho\leq_K \rho\restriction \rho(F)$. Let $F_1\in \cF\restriction \rho(F)$. Then $F_2=F_1$ is such that for every finite set $K_1\subseteq \Omega$ we take $K_2=K_1$ and see  that $\rho(F_2\setminus K_2)\subseteq \phi(\rho(F_1\setminus K_1))$. 

(\ref{prop:basic-prop-of-new-order:minimal-element})
We claim that  $\phi:\Lambda\to\Lambda$ given by $\phi(\lambda)=\lambda$ is a witness for $\rho_{\Fin(\Lambda)}\leq_K \rho$.
Let $F\in \cF$.
Let $\Omega =\{o_n:n\in\omega\}$.
Since $\rho(F\setminus\{o_i:i<n\})$ is infinite for every $n\in\omega$, we can pick a one-to-one sequence $(a_n:n\in\omega)$ such that 
$a_n\in \rho(F\setminus\{o_i:i<n\})$ for every $n\in\omega$.
Then $A=\{a_n:n\in\omega\}\in \Fin(\Lambda)^+$ is an infinite set.
For a finite set  $K\subseteq\Omega$ 
there is $n\in\omega$ such that $K\subseteq\{o_i:i<n\}$.
Then  $L=\{a_i:i<n\}$ is finite subset of $\Lambda$ and
$A\setminus L \subseteq \rho(F\setminus\{o_i:i<n\})\subseteq \rho(F\setminus K)$.
\end{proof}

Now we compare the relation $\leq_K$ between partition regular operations with the relation $\leq_K$ between ideals.

\begin{proposition}\ 
\label{prop:Katetov-for-rho-functions}
\label{prop:Katetov-for-rho-function-and-ideal-rho}
\label{prop:Katetov-for-ideal-rho}
Let $\rho_i:\cF_i\to[\Lambda_i]^\omega$ for each $i=1,2$ and $\rho:\cF\to[\Lambda]^\omega$ be partition regular. Let $\I$ be  ideal.

\begin{enumerate}
    \item $\rho_2\leq_K\rho_1 \implies \I_{\rho_2}\leq_K\I_{\rho_1}$ with the same witnessing function.\label{prop:Katetov-for-rho-functions:implication}

\item 
\begin{enumerate}
\item If $\rho_2$ is $P^+$ (in particular, if $\rho_2=\rho_\I$ and $\I$ is $P^+$), then $\rho_2\leq_K\rho_1 \iff  \I_{\rho_2}\leq_K\I_{\rho_1}$;\label{prop:Katetov-for-rho-functions:equivalence}\label{prop:Katetov-for-ideal-rho:idealP+}

\item $\rho\leq_K\rho_{\I} \iff \I_{\rho}\leq_K\I$;
\label{prop:Katetov-for-rho-function-and-ideal-rho:sparse}\label{prop:Katetov-for-rho-function-and-ideal-rho:Pminus}
\label{prop:Katetov-for-ideal-rho:item}
\end{enumerate}

\end{enumerate}
\end{proposition}

\begin{proof}
(\ref{prop:Katetov-for-rho-functions:implication})
Let $\cF_i\subseteq[\Omega_i]^\omega$ for $i=1,2$. Let $\phi$ be a witness for $\rho_2\leq_K\rho_1$.
We claim that $\phi$ is also a witness for $ \I_{\rho_2}\leq_K\I_{\rho_1}$.
Let $A\notin \I_{\rho_1}$.
Then there is $F_1\in \cF_1$ with $\rho_1(F_1)\subseteq A$.
Since $\rho_2\leq_K\rho_1$, there is $F_2\in \cF_2$ and a finite set $K_2\subseteq\Omega_2$ such that $\rho_2(F_2\setminus K_2)\subseteq \phi[\rho_1(F_1\setminus\emptyset)]=\phi[\rho_1(F_1)]$.
Since $F_2\setminus K_2\in \cF_2$ and $\rho_2(F_2\setminus K_2)\subseteq \phi[A]$, we obtain that $\phi[A]\notin \I_{\rho_2}$. Thus the proof of this item is finished.

(\ref{prop:Katetov-for-rho-functions:equivalence})
The ``in particular'' part  follows  from  Propositions~\ref{prop:rho-versus-ideal}(\ref{prop:rho-versus-ideal:ideal-gives-rho}) and \ref{prop:Plike-basic-properties}(\ref{prop:properties-of-ideal-versus-rho:ideal-rho}).

We only have to show the implication ``$\impliedby$'', because the reversed implication is true by item (\ref{prop:Katetov-for-rho-functions:implication}).
Let $\phi:\Lambda_1\to \Lambda_2$  be a witness for 
$\I_{\rho_2}\leq_K \I_{\rho_1}$.
We claim that $\phi$ is also a witness for $\rho_2\leq_K\rho_1$.
Let $F_1\in \cF_1$, $\Omega_1=\{o_n:n\in\omega\}$ and $B_n = \phi[\rho_1(F_1\setminus \{o_i:i<n\})]$ for each $n\in\omega$.
Then  $B_n\notin\I_{\rho_2}$, $B_n\supseteq B_{n+1}$ for each $n\in\omega$, and since $\I_{\rho_2}$ is $P^+$, there is $F_2\in \cF_2$ such that for each $n\in\omega$ there is a finite set $L_n\subseteq \Omega_2$ with $\rho_2(F_2\setminus L_n)\subseteq B_n$.
Now, for any   finite set $K_1\subseteq\Omega_1$ there is $n\in\omega$ such that $K_1\subseteq\{o_i:i<n\}$.
Let $K_2=L_n$.
Then 
$\rho_2(F_2\setminus K_2) \subseteq B_n \subseteq  
\phi[\rho_1(F_1\setminus K_1)]
$. Thus the proof of this item is finished.

(\ref{prop:Katetov-for-rho-function-and-ideal-rho:sparse}) The implication ``$\implies$'' follows from item (\ref{prop:Katetov-for-rho-functions:implication})
and 
Proposition~\ref{prop:rho-versus-ideal}(\ref{prop:rho-versus-ideal:ideal-gives-rho}),
so below we show the reverse implication.

Suppose that $\I$ is an ideal on $\Lambda$ and $\cF\subseteq[\Omega]^\omega$.
Let $\phi:\Lambda\to\Lambda$ be a witness of $\I_{\rho}\leq_K\I$.
We claim that the same $\phi$ is also a witness for $\rho\leq_K\rho_{\I}$.
Indeed, for $A\notin\I$ we find $E\in \cF$ such that $\rho(E)\subseteq \phi[A]$.
Using Proposition~\ref{prop:finite-support}, we can find a set $F\in \cF$ such that  $F\subseteq E$ and  for any finite set  $K\subseteq \Lambda_1$  there exists a finite set $L\subseteq\Omega$ with  
$\rho(F\setminus L)\subseteq \rho(F)\setminus \phi[K]$.
Consequently, $\rho(F\setminus L)\subseteq \phi[A]\setminus \phi[K] \subseteq\phi[A\setminus K]$.
\end{proof}

The following example shows that in general $\rho_2\leq_K\rho_1$ and $\I_{\rho_2}\leq_K\I_{\rho_1}$ are not equivalent.

\begin{example}
\label{ex}
$\Fin^2\leq_K\Hindman$, but $\rho_{\Fin^2}\not\leq_K\FS$.
\end{example}

\begin{proof}
By Proposition~\ref{prop:FinSQRD-below-FS-r-Delta}(\ref{prop:FinSQRD-below-FS-r-Delta:ideal}) we know that  $\I_{\rho_{\Fin^2}}=\Fin^2\leq_K\Hindman = \I_{\FS}$. Thus, we only need to show that $\rho_{\Fin^2}\not\leq_K\FS$. 

Suppose that $\rho_{\Fin^2}\leq_K \FS$ and let $\phi:\omega\to \omega^2$ be a witness  for this.
For each $n\in\omega$, we define $A_n = \phi^{-1}[(\omega\setminus n)\times \omega]$.
Then $A_0=\omega$, $A_n\supseteq A_{n+1}$ and 
$A_n\setminus A_{n+1}\subseteq \phi^{-1}[\{n\}\times\omega]\in \Hindman$ for each $n\in \omega$ by Proposition~\ref{prop:Katetov-for-rho-functions}(\ref{prop:Katetov-for-rho-functions:implication}).
Since $\FS$ is $P^-(\omega)$ by Proposition~\ref{prop:Plike-properties-for-known-rho}(\ref{prop:Plike-properties-for-known-rho:FS-r-Delta:Pminus}), there is $F\in [\omega]^\omega$ such that for every $n\in\omega$ there is a finite set $K_n\subseteq \omega$ with $\FS(F\setminus K_n) \subseteq A_n$.
Now, using the fact that $\rho_{\Fin^2}\leq_K\FS$, we find $B\notin \Fin^2$ such that for every $n\in \omega$ there is a finite set $L_n\subseteq\omega^2$
with $B\setminus L_n\subseteq \phi[\FS(F\setminus K_n)] \subseteq \phi[A_n] \subseteq (\omega\setminus n)\times \omega$. In particular, sets $B\cap (\{n\}\times \omega)$ are finite for every $n$, so $B\in \Fin^2$, a contradiction.  
\end{proof}

\begin{remark}
The partition regular function $\rho_{\Fin^2}$ from Example \ref{ex} is not $P^-$. In Example \ref{example:ideal-Katetov-does-not-imply-rho-Katetov-for-Pminus} we will show that there are partition regular functions $\rho_1$ and $\rho_2$ which are $P^-$ and have small accretions such that $\I_{\rho_2}\subseteq\I_{\rho_1}$ (in particular, $\I_{\rho_2}\leq_K\I_{\rho_1}$), but $\rho_2\not\leq_K\rho_1$.
\end{remark}



\subsection{Kat\v{e}tov order between \texorpdfstring{$\FS, r, \Delta, \vdW  \text{ and } \I_{1/n}$}{FS, r, Delta, W and I(1/n)}}

\begin{theorem}\ 
\label{thm:Katetov-between-known-rho}
\begin{enumerate}
    \item 
    $\Hindman \not\leq_K \Ramsey$. In particular, $\FS\not\leq_K r$.\label{thm:Katetov-between-known-rho:Hindman-not-below-Ramsey}

   \item $\Ramsey \not\leq_K \Hindman$. In particular, $r\not\leq_K\FS$.\label{thm:Katetov-between-known-rho:Ramsey-not-below-Hindman}

\item 
 $\Delta \leq_K \FS$ and    $\Diff\subseteq\Hindman$. In particular, $\Diff\leq_K\Hindman$.\label{thm:Katetov-between-known-rho:Diff-below-Hindman}

 \item 
 $\Delta \leq_K r$. In particular, $\Diff\leq_K\Ramsey$.\label{thm:Katetov-between-known-rho:Diff-below-Ramsey}

\item $\Ramsey\not\leq_K \Diff$. In particular, $r\not\leq_K \Delta$.\label{thm:Katetov-between-known-rho:Ramsey-not-below-Diff}

\item $\Hindman\not\leq_K \Diff$. In particular, $\FS\not\leq_K \Delta$.\label{thm:Katetov-between-known-rho:Hindman-not-below-Diff}

    \item 
    $\I_{1/n} \not\leq_K \Ramsey$. In particular,  $\rho_{\I_{1/n}} \not\leq_K r$.\label{thm:Katetov-between-known-rho:Summable-not-below-Ramsey}

    \item 
    $\I_{1/n} \not\leq_K \Hindman$. In particular,  $\rho_{\I_{1/n}} \not\leq_K \FS$.\label{thm:Katetov-between-known-rho:Summable-not-below-Hindman}

    \item 
    $\I_{1/n} \not\leq_K \Diff$. In particular,  $\rho_{\I_{1/n}} \not\leq_K \Delta$.\label{thm:Katetov-between-known-rho:Summable-not-below-Diff}

    \item 
    $\I_{1/n} \not\leq_K \vdW$. In particular,  $\rho_{\I_{1/n}} \not\leq_K \rho_{\vdW}$.\label{thm:Katetov-between-known-rho:Summable-not-below-vdW}

    \item 
    $\cD \not\leq_K \I_{1/n}$. In particular,  $\Delta\not\leq_K \rho_{\I_{1/n}}$.\label{thm:Katetov-between-known-rho:Diff-not-below-Summable}

    \item 
    $\Hindman \not\leq_K \I_{1/n}$. In particular,  $\FS \not\leq_K \rho_{\I_{1/n}}$.\label{thm:Katetov-between-known-rho:Hindman-not-below-Summable}

    \item 
    $\Ramsey \not\leq_K \I_{1/n}$. In particular,  $r \not\leq_K \rho_{\I_{1/n}}$.\label{thm:Katetov-between-known-rho:Ramsey-not-below-Summable}

    \item 
     $\cD\not\leq_K\mathcal{W}$. In particular,  $\Delta \not\leq_K \rho_{\cW}$.\label{thm:Katetov-between-known-rho:Diff-not-below-van}

    \item 
     $\Hindman\not\leq_K\mathcal{W}$. In particular,  $\FS \not\leq_K \rho_{\cW}$. \label{thm:Katetov-between-known-rho:Hindman-not-below-van}

     \item 
     $\Ramsey\not\leq_K\mathcal{W}$. In particular,  $r\not\leq_K \rho_{\cW}$. \label{thm:Katetov-between-known-rho:Ramsey-not-below-van}

\end{enumerate}

\end{theorem}

\begin{proof}
    The ``in particular'' parts follow from Proposition \ref{prop:Katetov-for-rho-functions}(\ref{prop:Katetov-for-rho-functions:implication}).

    The proofs of items~(\ref{thm:Katetov-between-known-rho:Hindman-not-below-Ramsey}), (\ref{thm:Katetov-between-known-rho:Ramsey-not-below-Hindman}),
    (\ref{thm:Katetov-between-known-rho:Summable-not-below-Ramsey}), (\ref{thm:Katetov-between-known-rho:Summable-not-below-Hindman}) and (\ref{thm:Katetov-between-known-rho:Summable-not-below-vdW}) 
    can be found in~\cite{naszKatetov}.   

(\ref{thm:Katetov-between-known-rho:Diff-below-Hindman})
The inclusion is proved in \cite[Proposition~4.2.1]{Shi2003Numbers} (see also \cite[Propositions~4.2]{MR3097000}). Below, we show that 
 $\Delta \leq_K \FS$.

We claim that the identity function  $\phi:\omega\to\omega$,  $\phi(n)=n$ for every $n\in\omega$ is a witness for $\Delta \leq_K \FS$.

For any infinite set  $A\subseteq \omega$, 
we define an infinite set 
$B=\{\sum_{i\leq n}a_i:n\in\omega\}$, where 
$\{a_n:n\in\omega\}$ is the increasing enumeration of $A$.
Next, for any finite set $K$, we define a finite set 
$L=\{0,1,\dots,\sum_{i\leq k}a_i\}$, where  $k = \max\{i\in \omega:a_i\in K\}$ (for $K=\emptyset$ we take $k=0$).
Finally, we observe that  
$\Delta(B\setminus L)\subseteq \FS(A\setminus K) = \phi[\FS(A\setminus K)]$, so the proof is finished. 

(\ref{thm:Katetov-between-known-rho:Diff-below-Ramsey}) 
We claim that $\phi: [\omega]^2\to\omega$ given by the formula $\phi(\{n,k\})=n-k$, where $n>k$, is a witness for $\Delta \leq_K r$. For any infinite set  $A\subseteq \omega$, we take $B =  A$. Then for any finite set $K\subseteq \omega$, we take $L=K$. Next, we notice that $\Delta(B\setminus L)  = \Delta(A\setminus K) =   \phi[[A\setminus K]^2] = \phi[r(A\setminus K)]$, so the proof is finished.

(\ref{thm:Katetov-between-known-rho:Ramsey-not-below-Diff})
It follows from items 
(\ref{thm:Katetov-between-known-rho:Diff-below-Hindman})
and
(\ref{thm:Katetov-between-known-rho:Ramsey-not-below-Hindman}).

(\ref{thm:Katetov-between-known-rho:Hindman-not-below-Diff})
It follows from items (\ref{thm:Katetov-between-known-rho:Diff-below-Ramsey})
and
(\ref{thm:Katetov-between-known-rho:Hindman-not-below-Ramsey}).

(\ref{thm:Katetov-between-known-rho:Summable-not-below-Diff})
It follows from items (\ref{thm:Katetov-between-known-rho:Summable-not-below-Hindman}) and (\ref{thm:Katetov-between-known-rho:Diff-below-Hindman}).

(\ref{thm:Katetov-between-known-rho:Diff-not-below-Summable})
Suppose otherwise: $\cD\leq_K\I_{1/n}$. 
By Proposition~\ref{prop:FinSQRD-below-FS-r-Delta}(\ref{prop:FinSQRD-below-FS-r-Delta:FS-ideal}) $\Fin^2\leq_K \cD$, so $\Fin^2\leq_K\I_{1/n}$.
By Proposition~\ref{prop:Pminus-versus-FinSQRD-for-ideal}(\ref{prop:Pminus-versus-FinSQRD-for-ideal:Pminus-Lambda}), we obtain that $\I_{1/n}$ is not a $P^-(\omega)$ ideal, a contradiction with Proposition~\ref{prop:Plike-properties-for-known-rho}(\ref{prop:Plike-properties-for-known-rho:summable}).

(\ref{thm:Katetov-between-known-rho:Hindman-not-below-Summable})
It follows from items (\ref{thm:Katetov-between-known-rho:Diff-below-Hindman}) and (\ref{thm:Katetov-between-known-rho:Diff-not-below-Summable}).

(\ref{thm:Katetov-between-known-rho:Ramsey-not-below-Summable})
It follows from items (\ref{thm:Katetov-between-known-rho:Diff-below-Ramsey}) and (\ref{thm:Katetov-between-known-rho:Diff-not-below-Summable}).

(\ref{thm:Katetov-between-known-rho:Diff-not-below-van})
Suppose otherwise: $\cD\leq_K\cW$. 
Using Proposition \ref{prop:FinSQRD-below-FS-r-Delta}(\ref{prop:FinSQRD-below-FS-r-Delta:ideal}) we get that $\Fin^2\leq_K \cD$, so $\Fin^2\leq_K\cW$. However, since $\cW$ is $F_\sigma$ (see \cite[Example~4.12]{MR4572258}), $\Fin^2\not\leq_K\cW$ by \cite[Theorems 7.5 and 9.1]{MR2520152}. A contradiction.

(\ref{thm:Katetov-between-known-rho:Hindman-not-below-van})
It follows from items (\ref{thm:Katetov-between-known-rho:Diff-below-Hindman}) and (\ref{thm:Katetov-between-known-rho:Diff-not-below-van}).

(\ref{thm:Katetov-between-known-rho:Ramsey-not-below-van})
It follows from items (\ref{thm:Katetov-between-known-rho:Diff-below-Hindman}) and (\ref{thm:Katetov-between-known-rho:Diff-not-below-van}).

\end{proof}

\begin{question}\ 
\label{q:Katetov-between-known-rho}
 Is $\vdW  \leq_K \I$ for $\I\in \{\I_{1/n},\Hindman, \Ramsey,\Diff\}$?
\end{question}

\begin{remark}
The positive answer to Question~\ref{q:Katetov-between-known-rho} for $\I=\I_{1/n}$  is  implied by the inclusion $\vdW\subseteq\I_{1/n}$ that is known as the Erd\H{o}s conjecture on arithmetic progressions (a.k.a.~the Erd\H{o}s–Tur\'{a}n conjecture) which can be rephrased in the following manner: if the sum of the reciprocals of the elements of a set $A\subseteq\omega$ diverges, then $A$ contains arbitrarily long finite arithmetic progressions.
\end{remark}


\section{Tallness and homogeneity}


\subsection{Tallness of partition regular functions}

An ideal $\I$ on $\Lambda$ is 
\emph{tall} if for every infinite set $A\subseteq \Lambda$ there exists an infinite set  $B\subseteq A$ such that $B\in \I$ (\cite[p.~210]{MR0363911}, see also \cite[Definition~0.6]{MR491197}).
It is not difficult to see that $\I$ is not tall $\iff$  $\I\leq_K\J$ for every ideal $\J$ $\iff$
$\I\leq_K\Fin$ $\iff$ $\I\restriction A=\Fin(A)$ for some $A\in\I^+$.

The following  proposition  characterizes  tallness of the ideal $\I_\rho$ in terms of $\rho$ and serves as a definition of tallness of partition regular functions.

\begin{proposition}
\label{prop:prop-of-talness}
Let $\rho:\cF\to[\Lambda]^\omega$ 
be partition regular  (with $\cF\subseteq [\Omega]^\omega$).
The following conditions are equivalent.
\begin{enumerate}

\item $\I_{\rho}$ is tall.\label{prop:prop-of-talness:ideal}

\item There exists  a partition regular function $\tau$ such that $\rho\not\leq_K \tau$.\label{prop:prop-of-talness:definition}

\item 
$\I_{\rho}\restriction \rho(F) \neq \Fin(\rho(F))$ for every $F\in \cF$.\label{prop:prop-of-talness:via-sparseness}

\item $\rho\not\leq_K\rho_{\Fin(\Lambda)}$.\label{prop:prop-of-talness:below-FIN}

\end{enumerate}
\end{proposition}

\begin{proof}
(\ref{prop:prop-of-talness:ideal})$\implies$(\ref{prop:prop-of-talness:definition}) If  $\I_\rho$ is tall, there is an ideal $\J$ such that $\I_\rho\not\leq_K\J$. Then $\rho\not\leq_K\rho_\J$ by Proposition \ref{prop:Katetov-for-rho-functions}(\ref{prop:Katetov-for-rho-functions:implication}).

(\ref{prop:prop-of-talness:definition})$\implies$(\ref{prop:prop-of-talness:via-sparseness})
Suppose that there is $F\in \cF$ such that $\I_{\rho}\restriction \rho(F) =  \Fin(\rho(F))$.
We will show that $\rho\leq_K \tau$
for every  partition regular function $\tau$ .

Take 
any  partition regular  function $\tau:\cG\to[\Sigma]^\omega$ 
with $\cG\subseteq [\Gamma]^\omega$.

Let $\phi:\Sigma\to\Lambda$ be a one-to-one function such that $\phi[\Sigma]=\rho(F)$.
We claim that $\phi$ is a witness for $\rho\leq_K\tau$.

Let $G\in \cG$ and $\Gamma=\{\gamma_n:n\in\omega\}$.
Since 
$\phi[\tau(G\setminus \{\gamma_i:i<n\})]$ is infinite for every $n\in\omega$,
we can pick a one-to-one sequence $(b_n:n\in\omega)$ such that
$b_n\in \phi[\tau(G\setminus \{\gamma_i:i<n\})]$ for each $n\in\omega$.
Define $B=\{b_n:n\in\omega\}$.
Since $B$ is infinite,  $B\subseteq \rho(F)$ and $\I_{\rho}\restriction\rho(F)=\Fin(\rho(F))$, 
there is $H\in \cF$ such that $\rho(H)\subseteq B$. 
Using Proposition~\ref{prop:finite-support}, there is  $E\in\cF$ with $E\subseteq H$ such that for any $n\in\omega$ 
there is a finite set $L\subseteq\Omega$ such that 
$\rho(E\setminus L)\subseteq \rho(E)\setminus \{b_i:i<n\}$.
Consequently, for any finite set $K\subseteq \Gamma$ there is $n\in\omega$ such that $K\subseteq \{\gamma_i:i<n\}$, so we can find a finite set $L\subseteq\Omega$ such that 
$\rho(E\setminus L)
\subseteq 
\rho(E)\setminus \{b_i:i<n\}
\subseteq 
B\setminus \{b_i:i<n\}
\subseteq
\phi[\tau(G\setminus \{\gamma_i:i<n\})]
\subseteq \phi[\tau(G\setminus K)]
$.

(\ref{prop:prop-of-talness:via-sparseness})$\implies$(\ref{prop:prop-of-talness:below-FIN})
Let $\phi:\Lambda\to\Lambda$ be a witness for $\rho\leq_K\rho_{\Fin(\Lambda)}$.
Since $\phi^{-1}[\{\lambda\}]\in \Fin(\Lambda)$ for every $\lambda\in \Lambda$
and $\phi[\Lambda]$ is infinite, there is an infinite set $A\subseteq\Lambda$, such that $\phi\restriction A$ is one-to-one.
Then we can find $F\in \cF$ such that $\rho(F)\subseteq\phi[A]$.
We claim that 
$\I_{\rho}\restriction \rho(F) = \Fin(\rho(F))$. Indeed, let $B\subseteq \rho(F)$ be infinite and observe that $\phi^{-1}[B]$ is infinite, so $B=\phi[\phi^{-1}[B]]\notin\I_\rho$.

(\ref{prop:prop-of-talness:below-FIN})$\implies$(\ref{prop:prop-of-talness:ideal}) If $\rho\not\leq_K\rho_{\Fin(\Lambda)}$ then by Proposition~\ref{prop:Katetov-for-rho-function-and-ideal-rho}(\ref{prop:Katetov-for-rho-function-and-ideal-rho:sparse}), $\I_\rho\not\leq_K\Fin(\Lambda)$, and consequently $\I_\rho$ is tall. 
\end{proof}

\begin{definition}
    We say that a partition regular function $\rho$ is \emph{tall} if any item of Proposition~\ref{prop:prop-of-talness} holds.    
\end{definition}

\begin{proposition}
The ideals $\Hindman$, $\Ramsey$, $\cD$, $\vdW$ and $\I_{1/n}$ are tall (hence, $FS$, $r$, $\Delta$, $\rho_\cW$ and $\rho_{\I_{1/n}}$ are tall).
\end{proposition}

\begin{proof}
For the case of $\vdW$ and $\I_{1/n}$ see \cite[p. 3-4]{MR3868039}. For other cases, see \cite [under Lemma 1.6.24]{alcantara-phd-thesis} and \cite[Proposition 4.3 and text above Lemma 3.2]{MR3097000}). Tallness of the listed partition regular functions follows then from Proposition \ref{prop:prop-of-talness}.
\end{proof}


\subsection{Homogeneity  of partition regular functions}

Let $\I_i$ be an ideal on $\Lambda_i$ for each $i=1,2$.
Ideals $\I_1$ and $\I_2$ are \emph{isomorphic} 
(in short: $\I_1\approx \I_2$) if 
 there exists  a bijection $\phi:\Lambda_1\to \Lambda_2$ 
 such that $ A\in \I_1 \iff \phi[A] \in \I_2$ for each $A\subseteq \Lambda_1$. An ideal  $\I$ on $\Lambda$ is \emph{homogeneous} if 
 the ideals $\I$ and $\I\restriction A$ are isomorphic for every $A\in \I^+$ (\cite[Definition~1.3]{MR3594409}, see also \cite{Fremlin-countable-type}).
We say that $\I$ is \emph{K-homogeneous} if 
$\I\restriction A \leq_K \I$ for every $A\in \I^+$ (in \cite[p.~37]{MR2777744}, the author uses the name \emph{K-uniform} in this case).
Note that we always have $\I\leq_K \I\restriction A$ for every $A\in \I^+$ (see e.g.~\cite[p.~46]{MR2777744}).

\begin{definition}
Let $\rho:\cF\to[\Lambda]^\omega$ be partition regular. 
We say that $\rho$ is \emph{K-homogeneous} if 
$\rho\restriction A \leq_K \rho$ for every $A\in \I_\rho^+$
(note that we always have $\rho\leq_K \rho\restriction A$ for every $A\in \I_\rho^+$ by Proposition~\ref{prop:basic-prop-of-new-order}(\ref{prop:basic-prop-of-new-order:restrictions})).
\end{definition}

\begin{proposition}\
\label{prop:Katetov-for-rho-functions:homo:homo:rho-via-ideal}
\begin{enumerate}
    \item If a partition regular function $\rho$ is K-homogeneous then $\I_\rho$ is K-homogeneous.\label{prop:Katetov-for-rho-functions:homo:homo:rho-via-ideal:item1}
    \item An ideal  $\I$ is K-homogeneous $\iff$ $\rho_\I$ is K-homogeneous.\label{prop:Katetov-for-rho-functions:homo:homo:rho-via-ideal:item2}
\end{enumerate}
\end{proposition}

\begin{proof} 
(\ref{prop:Katetov-for-rho-functions:homo:homo:rho-via-ideal:item1}) It follows from Propositions~\ref{prop:Katetov-for-rho-functions}(\ref{prop:Katetov-for-rho-functions:implication}) and \ref{prop:restrictions-of-rho}.

(\ref{prop:Katetov-for-rho-functions:homo:homo:rho-via-ideal:item2}) Observe that if $A\in\I^+$ then $\rho_{\I\restriction A}=\rho_{\I}\restriction A$. Thus, it follows from Propositions~\ref{prop:Katetov-for-rho-functions}(\ref{prop:Katetov-for-ideal-rho:item}) and \ref{prop:restrictions-of-rho}.
\end{proof}

We need the following lemma to show that $FS$ and $r$ are K-homogeneous.

\begin{lemma}
\label{lem:K-homogeneous}
Let $\rho:\cF\to[\Lambda]^\omega$ be partition regular (with $\cF\subseteq[\Omega]^\omega$). If $\I_\rho$ is homogeneous and $\rho$ is $P^-$ and has small accretions then $\rho$ is K-homogeneous.
\end{lemma}

\begin{proof} 
Let $A\in\I_\rho^+$. Since $\I_\rho$ is homogeneous, $\I_\rho\restriction A$ and $\I_\rho$ are isomorphic. Let $f:\Lambda\to A$ be a bijection witnessing it. We claim that $f$ witnesses $\rho\restriction A\leq_K\rho$.

Let $F\in\cF$. Since $\rho$ has small accretions, there is $G\in\cF$ such that $G\subseteq F$ and $G$ has small accretions. Enumerate $\Omega=\{o_n:n\in\omega\}$ and define $K_n=\{o_i:i\leq n\}$ and $A_n=f[\rho(G\setminus K_n)]$ for all $n\in\omega$. Then $A_n\supseteq A_{n+1}$. Since $G$ has small accretions and $f$ is a bijection and witnesses that $\I_\rho\restriction A$ and $\I_\rho$ are isomorphic, $A_n\in (\I_\rho\restriction A)^+$ and $A_n\setminus A_{n+1}\subseteq f[\rho(G\setminus K_n)\setminus \rho(G\setminus K_{n+1})]\subseteq f[\rho(G)\setminus \rho(G\setminus K_{n+1})]\in\I_\rho\restriction A$. Using the fact that $\rho$ is $P^-$, we can find $H\in\cF\restriction A$ such that $\rho(H)\subseteq^\rho A_n=f[\rho(G\setminus K_n)]$ for all $n\in\omega$. Hence, given any finite set $K\subseteq\Omega$ there are $n\in\omega$ and finite $L\subseteq\Omega$ such that $K\subseteq K_n$ and $\rho(H\setminus L)\subseteq A_n=f[\rho(G\setminus K_n)]\subseteq f[\rho(G\setminus K)]$.
\end{proof}

\begin{proposition}\ 
\label{prop:rho-is-K-homo}
\begin{enumerate}

\item The ideals $\Hindman$, $\Ramsey$ and $\vdW$ are  homogeneous (hence,  K-homogeneous).\label{prop:rho-is-K-homo:ideals-H-R-W-are-homo} 

\item The functions $\FS$ and $r$ are K-homogeneous.\label{prop:rho-is-K-homo:FS}\label{prop:rho-is-K-homo:r} 

\end{enumerate}
\end{proposition}

\begin{proof}
(\ref{prop:rho-is-K-homo:ideals-H-R-W-are-homo})
See   
\cite[Examples 2.5 and 2.6]{MR3594409}.

(\ref{prop:rho-is-K-homo:FS})
It follows from item (\ref{prop:rho-is-K-homo:ideals-H-R-W-are-homo}), Lemma~\ref{lem:K-homogeneous} and Propositions~\ref{prop:Plike-properties-for-known-rho}(\ref{prop:Plike-properties-for-known-rho:FS-r-Delta:weakPplus}) and \ref{prop:ideal-rho-is-sparse}.
\end{proof}

\begin{question}\ 
\label{q:is-Delta-K-homo}
\label{q:is-summable-ideal-K-homo}
\begin{enumerate}
    \item 
Is the function  $\Delta$ K-homogeneous? \label{q:is-Delta-K-homo:item}
 
\item 
Is the ideal $\I_{1/n}$ K-homogeneous? \label{q:is-summable-K-homo:item}
 
\end{enumerate}
\end{question}


\part{\texorpdfstring{$\FinBW$}{FinBW } spaces}
\label{part:FinBW-spaces}

In this part we define the main object of our studies -- classes of sequentially compact spaces defined with the aid of partition regular functions (Definition \ref{def:FinBW-for-rho}). Next, we prove some general results about those classes of spaces (Theorem \ref{thm:FinBW-hFinBW-homogeneous}).


\section{A convergence with respect to partition regular functions}

\begin{definition}
\label{def:rho-convergence}
Let $X$ be a topological space.
Let $\rho:\cF\to[\Lambda]^\omega$ 
be partition regular with $\cF\subseteq [\Omega]^\omega$.
\begin{enumerate}
    \item 
For $F\in \cF$, a function $f:  \rho(F) \to X$ is called a \emph{$\rho$-sequence} in $X$.

\item 
A $\rho$-sequence  $f:\rho(F) \to X$ is \emph{$\rho$-convergent} 
to a point $x\in X$ if 
for every neighbourhood $U$ of $x$ there is a finite set $K\subseteq\Omega$ such that 
$$f[\rho(F\setminus K)] \subseteq U.$$
\end{enumerate}
\end{definition}

\begin{remark}
Various kinds of convergences considered in the literature can be described in terms of $\rho$-convergence.

\begin{enumerate}

\item 
If   $\rho=\FS$, then $\rho$-convergence coincides  with  
\emph{$IP$-convergence}  (see \cite{MR603625}, \cite{MR531271} or \cite{MR1887003}).

    \item 
If   $\rho=r$, then $\rho$-convergence coincides  with  
the \emph{$\mathcal{R}$-convergence}  (see \cite{BergelsonZelada}, \cite{MR2948679}, \cite[Definition~2.1]{MR4552506}).

\item 
If   $\rho=\Delta$, then $\rho$-convergence coincides  with  
the \emph{differential convergence} (see \cite[Definition~4.2.4]{Shi2003Numbers} or \cite[p.~2010]{MR3097000}).

\item 
If $\I$ is an ideal on $\Lambda$ and $\rho_\I$ is defined as in Proposition~\ref{prop:rho-versus-ideal}(\ref{prop:rho-versus-ideal:ideal-gives-rho}), then $\rho_\I$-convergence coincides with the ordinary convergence.

\end{enumerate}
\end{remark}

The following  proposition reveals relationships 
between $\rho$-convergence and convergence.

\begin{proposition}\ 
\label{prop:basic-relationships-between-rho-like-convergence}
\label{prop:rho-con-implies-suseq-ordinar-comvergence}
\begin{enumerate}
\item 
Let $\rho:\cF\to[\Lambda]^\omega$ 
be partition regular with $\cF\subseteq [\Omega]^\omega$.
Let $F\in \cF$ and $f:\rho(F)\to X$.
\begin{enumerate}
\item If  $f$ is convergent to $L$, then $f\restriction\rho(E)$ is $\rho$-convergent  to $L$ for some $E\in\cF\restriction\rho(F)$.\label{prop:basic-relationships-between-rho-like-convergence:implication} 

\item  
If $f$ is $\rho$-convergent to $L$, then $f\restriction A$ is convergent to $L$ 
for some infinite set $A \subseteq \rho(F)$.\label{prop:rho-con-implies-suseq-ordinar-comvergence:item}

\end{enumerate}

\item 
Let  $\I$ be  an ideal on  $\Lambda$
and  $f:A\to X$ for some  $A\in \I^+$.
Then 
$f$ is convergent  to $L$  $\iff$ $f$ is $\rho_{\I}$-convergent to $L$.\label{prop:basic-relationships-between-rho-like-convergence:ideal-rho} 

\end{enumerate}
\end{proposition}

\begin{proof}
(\ref{prop:basic-relationships-between-rho-like-convergence:implication})
Let $E\in\cF$ with $E\subseteq F$ be as in Proposition~\ref{prop:finite-support} and let $U$ be a neighborhood of $L$. 
Then there exists a finite set  $K$ such  that $f(n)\in U$ for every $n\in \rho(F)\setminus K$.
There is a finite set $L$ such that $\rho(E\setminus L)\subseteq \rho(E)\setminus K\subseteq \rho(F)\setminus K$. Consequently, $f(n)\in U$ for every $n\in \rho(E\setminus L)$.

(\ref{prop:rho-con-implies-suseq-ordinar-comvergence:item})
Let $\Omega=\{o_n:n\in\omega\}$.
For each $n\in\omega$, we pick 
$\lambda_n \in \rho(F\setminus\{o_i:i<n\})\setminus \{\lambda_i:i<n\}$.
Let $A=\{\lambda_n:n\in\omega\}$. We claim that $f\restriction A$ is convergent to $L$.
Indeed, if $U$ is a neighborhood of $L$, then there is a finite set $K\subseteq \Omega$ such that 
$f[\rho(F\setminus K)]\subseteq U$.
Let $n\in\omega$ be such that $K\subseteq \{o_i:i<n\}$.
Then 
$
f[A\setminus \{\lambda_i:i<n\}] 
\subseteq 
f[\rho(F\setminus \{o_i:i<n\})] 
\subseteq f[\rho(F\setminus K)]
\subseteq U
$.

(\ref{prop:basic-relationships-between-rho-like-convergence:ideal-rho})
It is straightforward. 
\end{proof}


\section{\texorpdfstring{$\FinBW$}{FinBW } spaces}
\label{sec:FinBW}

Let $\I$ be an ideal on a countable infinite set $\Lambda$.
The following classes of topological spaces were extensively examined in the literature 
(see e.g.~\cite{MR2320288, MR2471564, MR4584767}):
\begin{enumerate}
    \item 
$\FinBW(\I)$ is the class of all topological spaces $X$ such that for every sequence $f:\Lambda\to X$ there exists $A\in \I^+$ such that $f\restriction A$ converges 
(in \cite{MR2471564}, spaces from  $\FinBW(\I)$ are called \emph{$\I$-spaces});
    \item 
$\hFinBW(\I)$ is   the class of all topological spaces $X$ such that for every 
 $B\in \I^+$ and every sequence $f:B\to X$ 
there exists $A\in \I^+$ such that $A\subseteq B$ and $f\restriction A$ converges.
\end{enumerate}

\begin{remark}
The classes $\FinBW(\I_{\rho})$ for $\rho\in \{\vdW,\I_{1/n}\}$
were  examined in the literature under other names:
\begin{enumerate}
    \item 
in \cite[Definition~3]{MR1866012}, spaces from $\FinBW(\vdW)$ are called  
 \emph{van der Waerden spaces};

\item in \cite[Definition~2.1]{MR2471564},  spaces from $\FinBW(\I_{1/n})$
are called \emph{$\I_{1/n}$-spaces}.
\end{enumerate}
\end{remark}

\begin{definition}
\label{def:FinBW-for-rho}
Let $\rho:\cF\to[\Lambda]^\omega$ 
be a partition regular function.
\begin{enumerate}
 \item 
$\FinBW(\rho)$ is  the class of all topological spaces $X$ such that 
for every sequence $f:\Lambda\to X$ there exists $F\in \cF$ such that $f\restriction \rho(F)$ $\rho$-converges.
 \item 
$\hFinBW(\rho)$ is the class of all topological spaces $X$ such that for 
every $\rho$-sequence $f:\rho(E)\to X$ 
there exists $F\in \cF$ such that $\rho(F)\subseteq \rho(E)$ and $f\restriction \rho(F)$ $\rho$-converges.
\end{enumerate}
\end{definition}

\begin{remark}
The classes $\FinBW(\rho)$ for $\rho\in \{\FS,r,\Delta\}$
were examined in the literature under other names:
\begin{enumerate}

    \item 
in \cite[Definition~4]{MR1887003}, 
 spaces from $\FinBW(\FS)$ are called \emph{Hindman spaces};

\item 
in \cite{MR2948679} (see also \cite[Definition~2.1]{MR4552506}), 
 spaces from $\FinBW(r)$ are called 
\emph{spaces with the Ramsey property}, and we will call them \emph{Ramsey  spaces} in short;

\item 
in \cite[Definition~4.2.4]{Shi2003Numbers} (see also \cite[p.~2010]{MR3097000}),
 spaces in $\FinBW(\Delta)$ 
are called \emph{differentially compact spaces}.

\end{enumerate}
\end{remark}

\begin{remark}
Recall that if $(\Lambda,<)$ is a directed set, then  any function $f:\Lambda\to X$ is called a \emph{net} in $X$. 
A net $f:\Lambda\to X$ in a topological space $X$  \emph{converges} to $x\in X$ if for every neighborhood $U$ of $x$ there is $\lambda_0\in\Lambda$ such that $f(\lambda)\in U$ for every $\lambda>\lambda_0$ (see e.g.~\cite[p.~49]{MR1039321}).
In \cite[Remark 2.6]{MR3461181}, the authors notice that 
if $\cB$ is a coideal basis on  $(\Lambda,<)$, then 
$(B,<\cap(B\times B))$ is a directed set and 
$f\restriction B$ is a subnet of $f$ for every $B\in \cB$.
 Furthermore, they examine topological spaces $X$ having the property that every  net $f:\Lambda\to X$ has a convergent subnet $f\restriction B$ with some $B\in \cB$ (\cite[p.~418]{MR3461181}). 
It is not difficult to see that the class of spaces they examine coincides with the class  
$\FinBW(\rho_{\cB})$ with $\rho_\cB$ defined as in Proposition~\ref{prop:rho-versus-ideal-on-directed-sets}(\ref{prop:rho-versus-ideal-on-directed-sets:ideal-gives-rho:coideal-basis}). 
\end{remark}

The following proposition reveals relationships 
between $\FinBW$-like spaces defined with the aid of partition regular functions and ideals.

\begin{proposition}
\label{prop:basic-relationships-between-FinBW-like-spaces}
 Let $\rho:\cF\to[\Lambda]^\omega$ 
be partition regular  with $\cF\subseteq [\Omega]^\omega$.
Let  $\I$ be  an ideal on $\Lambda$.
\begin{enumerate}

\item \label{prop:basic-relationships-between-FinBW-like-spaces:hFinBW-equals-intersection-of-FinBW}
\begin{enumerate}
\item $\hFinBW(\rho) = \bigcap\{\FinBW(\rho\restriction \rho(F)):F\in \cF\}$.\label{prop:basic-relationships-between-FinBW-like-spaces:hFinBW-equals-intersection-of-FinBW:rho}
\item $\hFinBW(\I) = \bigcap\{\FinBW(\I\restriction A):A\in \I^+\}$.
\end{enumerate}

\item \label{prop:basic-relationships-between-FinBW-like-spaces:inclusion}
\begin{enumerate}
\item
$\hFinBW(\rho) \subseteq \FinBW(\rho)$.
\item 
$\hFinBW(\I) \subseteq \FinBW(\I)$.
\end{enumerate}

\item 
$\FinBW(\I_\rho) \subseteq \FinBW(\rho)$  and $\hFinBW(\I_\rho) \subseteq  \hFinBW(\rho)$.\label{prop:basic-relationships-between-FinBW-like-spaces:rho-ideal}
    \item $\FinBW(\I) = \FinBW(\rho_\I)$ and  $\hFinBW(\I) = \hFinBW(\rho_\I)$.\label{prop:basic-relationships-between-FinBW-like-spaces:ideal-rho}

    \item \label{prop:basic-relationships-between-FinBW-like-spaces:K-homo}\label{thm:FinBW-hFinBW-homogeneous:-equality}
\begin{enumerate}
\item
If $\rho$ is K-homogeneous, then 
$\hFinBW(\rho)=\FinBW(\rho)$.\label{prop:basic-relationships-between-FinBW-like-spaces:K-homo:rho}
    \item If $\I$ is K-homogeneous, then 
$\hFinBW(\I)=\FinBW(\I)$.\label{prop:basic-relationships-between-FinBW-like-spaces:K-homo:ideal}
\end{enumerate}

\end{enumerate}
\end{proposition}

\begin{proof}
(\ref{prop:basic-relationships-between-FinBW-like-spaces:hFinBW-equals-intersection-of-FinBW}) and (\ref{prop:basic-relationships-between-FinBW-like-spaces:inclusion}) Straightforward.

(\ref{prop:basic-relationships-between-FinBW-like-spaces:rho-ideal})
It follows from Proposition~\ref{prop:basic-relationships-between-rho-like-convergence}(\ref{prop:basic-relationships-between-rho-like-convergence:implication}) (the other inclusion does not follow from Proposition~\ref{prop:basic-relationships-between-rho-like-convergence}(\ref{prop:rho-con-implies-suseq-ordinar-comvergence:item}) as it gives us only an infinite set $A$, not necessarily $A\in\I_\rho^+$).

(\ref{prop:basic-relationships-between-FinBW-like-spaces:ideal-rho})
It follows from Proposition~\ref{prop:basic-relationships-between-rho-like-convergence}(\ref{prop:basic-relationships-between-rho-like-convergence:ideal-rho}). 

(\ref{prop:basic-relationships-between-FinBW-like-spaces:K-homo:rho})
We only need to show 
$\FinBW(\rho)\subseteq \hFinBW(\rho)$.
Let $X\in\FinBW(\rho)$ and $f:\rho(E)\to X$ be a $\rho$-sequence in $X$ for some $E\in \cF$.
Let $\phi:\Lambda\to\rho(E)$ be a witness for 
$\rho \restriction (\cF\restriction\rho(E))\leq_K \rho$. 
Since $f\circ \phi:\Lambda\to X$, there is $F\in \cF$ such that $f\circ\phi\restriction\rho(F)$ is $\rho$-convergent to some $x\in X$.
Since 
$\rho \restriction (\cF\restriction\rho(E))\leq_K \rho$,
there is $G\in \cF\restriction \rho(E)$  such that
for every finite set $K\subseteq\Omega$ there is a finite set $L\subseteq \Omega$
with 
$\rho(G\setminus L) \subseteq \phi[\rho(F\setminus K)]$.
We claim that $f\restriction \rho(G)$ is $\rho$-convergent to $x$.
Let $U$ be a neighborhood of $x$.
Then there is a finite set $K\subseteq\Omega$ such that $(f\circ\phi)[\rho(F\setminus  K)]\subseteq U$.
We pick a finite set $L\subseteq\Omega$ such that 
$\rho(G\setminus L) \subseteq \phi[\rho(F\setminus K)]$.
Then 
$f[\rho(G\setminus L)] \subseteq f[\phi[\rho(F\setminus K)]] \subseteq U$, so the proof is finished.

(\ref{prop:basic-relationships-between-FinBW-like-spaces:K-homo:ideal})
It follows from items (\ref{prop:basic-relationships-between-FinBW-like-spaces:K-homo:rho}) and 
(\ref{prop:basic-relationships-between-FinBW-like-spaces:ideal-rho})
and Proposition~\ref{prop:Katetov-for-rho-functions:homo:homo:rho-via-ideal}(\ref{prop:Katetov-for-rho-functions:homo:homo:rho-via-ideal:item2}).
\end{proof}

\begin{remark}
In Theorem~\ref{prop:basic-relationships-between-FinBW-like-spaces}(\ref{prop:basic-relationships-between-FinBW-like-spaces:rho-ideal}), we cannot replace inclusion with equality in general because in \cite[Theorems~3 and 10]{MR1887003} the author proved that $\FinBW(\Hindman)$ contains only finite Hausdorff spaces, whereas $\FinBW(\FS)$ contains infinite (even uncountable) ones.
\end{remark}

\begin{corollary}\ 
\begin{enumerate}
\item
\cite[Proposition~4]{MR1866012} $\hFinBW(\vdW)=\FinBW(\vdW)$, and consequently the product of two van der Waerden spaces is van der Waerden.
\item  
\cite[Lemma~8]{MR1887003} $\hFinBW(\FS)=\FinBW(\FS)$, and consequently the product of two Hindman spaces is Hindman.
\item 
\cite[Theorem~3.4]{MR4552506} $\hFinBW(r)=\FinBW(r)$, and consequently the product of two Ramsey spaces is Ramsey.
\end{enumerate}
\end{corollary}

\begin{proof}
It follows from Theorem~\ref{prop:basic-relationships-between-FinBW-like-spaces}(\ref{prop:basic-relationships-between-FinBW-like-spaces:K-homo}) and 
Proposition~\ref{prop:rho-is-K-homo}.
\end{proof}

\begin{question}[{\cite{Flaskova-slides-2008, Flaskova-slides-2014} and \cite[Question~4.2.3]{Shi2003Numbers}}]\ 
\label{q:FinBW=hFinBW-for-summable}
\label{q:product-of-summable-spaces}
\label{q:FinBW=hFinBW-for-differentially-compact}
\label{q:product-of-diff-comp-spaces}
\begin{enumerate}
    \item 
    \begin{enumerate}
    \item Does $\FinBW(\I_{1/n})=\hFinBW(\I_{1/n})$?\label{q:FinBW=hFinBW-for-summable:item}

\item Is the product of two $\I_{1/n}$-spaces an $\I_{1/n}$-space?\label{q:product-of-summable-spaces:item}

\end{enumerate}

\item 
\begin{enumerate}
    \item Does $\FinBW(\Delta)=\hFinBW(\Delta)$?\label{q:FinBW=hFinBW-for-differentially-compact:item}

\item Is the product of two differentially compact spaces a differentially compact space?\label{q:product-of-diff-comp-spaces:item}
\end{enumerate}
\end{enumerate}
\end{question}

Note that the positive answer to the question in item (\ref{q:FinBW=hFinBW-for-summable:item})  gives the  positive answer to the question in item (\ref{q:product-of-summable-spaces:item}), and similarly for the questions in the  second item. Moreover, the positive answer to the Question \ref{q:is-Delta-K-homo}(\ref{q:is-Delta-K-homo:item}) (Question \ref{q:is-Delta-K-homo}(\ref{q:is-summable-K-homo:item}), resp.)   gives the  positive answer to Question~\ref{q:FinBW=hFinBW-for-differentially-compact}(\ref{q:FinBW=hFinBW-for-differentially-compact:item}) (Question~\ref{q:FinBW=hFinBW-for-summable}(\ref{q:FinBW=hFinBW-for-summable:item}), resp.).

Let us now turn to one of the main result of this paper.

\begin{theorem}
\label{thm:FinBW-hFinBW-homogeneous}
\label{thm:compact-metric-implies-FinBW(rho)}
\label{thm:Pminus-implies-various-spaces-in-FinBW(rho)}
Let $\rho:\cF\to[\Lambda]^\omega$ 
be partition regular with $\cF\subseteq [\Omega]^\omega$.

\begin{enumerate}

\item 
$\FinBW(\rho)$ contains all finite spaces and is a subclass of the class of all sequentially compact spaces.\label{thm:FinBW-hFinBW-homogeneous:SEQ-COMP}\label{thm:FinBW-hFinBW-homogeneous:inclusion}\label{thm:compact-metric-implies-FinBW(rho):skonczone}

\item $\rho$ is not tall $\iff$  $\FinBW(\rho)$ coincides with the class of  all sequentially compact spaces.\label{thm:FinBW-hFinBW-homogeneous:-equality-for-SeqComp:equivalence}\label{thm:FinBW-hFinBW-homogeneous:-equality-for-SeqComp}

\item The following conditions are equivalent.\label{thm:compact-metric-implies-FinBW(rho):Pminus-is-necessary}
\begin{enumerate}
    \item $\rho$ is  $P^-(\Lambda)$.
    \item There are Hausdorff compact spaces of arbitrary cardinality  that belong to $\FinBW(\rho)$.
    \item There exists an infinite Hausdorff topological space $X\in  \FinBW(\rho)$.
\end{enumerate}

\item The following conditions are equivalent.\label{thm:compact-metric-implies-FinBW(rho):Pminus-is-necessary-hFinBW}
\begin{enumerate}
    \item $\rho$ is  $P^-$.
    \item There are Hausdorff compact spaces of arbitrary cardinality  that belong to $\hFinBW(\rho)$.
    \item There exists an infinite Hausdorff topological space $X\in  \hFinBW(\rho)$.
\end{enumerate}

    \item 
If $\rho$ is  $P^-$, then  
\begin{enumerate}
    \item 
the  uncountable non-compact Hausdorff  space $X=\omega_1$ with the order topology belongs to $\FinBW(\rho)$,\label{thm:Pminus-implies-various-spaces-in-FinBW(rho):omega-1}

\item assuming  Continuum Hypothesis (CH) there are Hausdorff compact and separable spaces of cardinality $\continuum$  that belong to $\FinBW(\rho)$.\label{thm:Pminus-implies-various-spaces-in-FinBW(rho):Mrowka-separable}
\end{enumerate}

    \item 
If $\rho$ is weak  $P^+$, then  
every compact metric space is in $\hFinBW(\rho)$.\label{thm:compact-metric-implies-FinBW(rho):metric-compact-spaces}

    \item 
If $\rho$ is $P^+$, then every Hausdorff topological space with the property $(*)$ belongs to 
$\hFinBW(\rho)$.\label{thm:compact-metric-implies-FinBW(rho):spaces-with-star-property}

(A topological space $X$ has the \emph{property (*)} if for every countable set $D\subseteq X$ the closure $\closure_X(D)$ is compact and first-countable -- see \cite{MR1866012}.) 

\end{enumerate}

\end{theorem}

\begin{proof}
(\ref{thm:FinBW-hFinBW-homogeneous:SEQ-COMP})
It follows from Proposition~\ref{prop:rho-con-implies-suseq-ordinar-comvergence}(\ref{prop:rho-con-implies-suseq-ordinar-comvergence:item}).

(\ref{thm:FinBW-hFinBW-homogeneous:-equality-for-SeqComp:equivalence}) 
The implication ``$\impliedby$'' will follow from Theorem \ref{thm:MAD-Mrowka-nie-jest-Fspejsem}(\ref{thm:MAD-Mrowka-nie-jest-Fspejsem:MAD-in-ideal}). 
To prove the implication ``$\implies$'' we only need to show that every sequentially compact space belongs to $\FinBW(\rho)$.
Fix a sequentially compact space $X$ and $f:\Lambda\to X$.
Let $\phi:\Lambda\to\Lambda$ be a witness for 
$\rho\leq_K \rho_{\Fin(\Lambda)}$.
Since $f\circ \phi:\Lambda\to X$,  there is an infinite set $A\subseteq \Lambda$ such that $(f\circ \phi)\restriction A$ is convergent to some $x\in X$. 
Then  there is $F\in \cF$ with $\rho(F)\subseteq\phi[A]$. 
We claim that $f\restriction \rho(F)$ is $\rho$-convergent to $x$. Let $U$ be any  neighbourhood of $x$. Then there is a finite set $L\subseteq \Lambda$ such that $f[\phi[A\setminus L]] \subseteq  U$. 
Now, we can find a finite set $K \subseteq \Omega$ such that $\rho(F\setminus K) \subseteq\phi[A\setminus L]$. 
Thus, $f[\rho(F\setminus K)]\subseteq U$.

(\ref{thm:compact-metric-implies-FinBW(rho):Pminus-is-necessary})
(a)$\implies$(b) Let $\kappa$ be an infinite  cardinal number. 
Let $X = \kappa\cup\{\infty\}$ be the Alexandroff one-point compactification of the discrete space $\kappa$. 
Then $X$ is Hausdorff, compact and has cardinality $\kappa$.
Moreover,  open neighborhoods of $\infty$ are of the form $X \setminus S$ where $S$ is a compact (hence finite)  subset of $\kappa$.
We show that $X$ is in $\FinBW(\rho)$.
Let $f:\Lambda\to X$.
If there is $x\in X$ with $f^{-1}[\{x\}]\notin\I_\rho$, then we take $F\in \cF$ such that $\rho(F)\subseteq f^{-1}[\{x\}]$ and see that $f\restriction \rho(F)$ is $\rho$-convergent to $x$.
Now, we assume that $f^{-1}[\{x\}]\in \I_\rho$ for every $x\in X$.
By Proposition~\ref{prop:Pminus-for-rho-equivalent-condition}, there is $F\in \cF$ such that 
for every finite set $S\subseteq X$ there is a finite set $K_S$ 
such that $\rho(F\setminus K_S)\cap f^{-1}[S]=\emptyset$.
We claim that $f\restriction \rho(F)$ is $\rho$-convergent to $\infty$. 
Let $U$ be an open neighborhood of $\infty$. Let $S\subseteq\kappa$ be a finite set with $U=X\setminus S$.  
Then $f[\rho(F\setminus K_S)] \subseteq X\setminus S=U$.

(b)$\implies$(c) Obvious.

(c)$\implies$(a)
Suppose that $\rho$ is not $P^-(\Lambda)$ and let $A_n\in\I^+_\rho$ be the witnessing sequence, i.e., $A_0=\Lambda$, $A_{n+1}\subseteq A_n$, $A_n\setminus A_{n+1}\in\I_\rho$ and for each $F\in\cF$ there is $n\in\omega$ such that $\rho(F)\not\subseteq^\rho A_n$. Note that $\bigcap_{n\in\omega}A_n\in\I_\rho$.

Let $X$ be an infinite Hausdorff topological space. We will show that $X\notin\FinBW(\rho)$. If $X$ is not sequentially compact, then $X\notin\FinBW(\rho)$ by item (\ref{thm:FinBW-hFinBW-homogeneous:SEQ-COMP}). If $X$ is sequentially compact, then find any one-to-one sequence $\{x_n:n\in\omega\}$ in $X$ converging to some $x\in X$. Without loss of generality we may assume that $x\neq x_n$ for all $n\in\omega$. Define $f:\Lambda\to X$ by $f\restriction\bigcap_{n\in\omega}A_n=x_0$ and $f(\lambda)=x_{n+1}$, where $n$ is such that $\lambda\in A_n\setminus A_{n+1}$. Suppose for the sake of contradiction that there are $L\in X$ and $F\in\cF$ such that $f\restriction\rho(F)$ $\rho$-converges to $L$. By Proposition \ref{prop:basic-relationships-between-rho-like-convergence}(\ref{prop:rho-con-implies-suseq-ordinar-comvergence:item}) we get that either $L=x_n$ for some $n\in\omega$ or $L=x$. 

If $L=x_n$ for some $n\in\omega$, find open $U$ and $V$ such that $x_n\in U$, $x\in V$ and $U\cap V=\emptyset$. Since $x_m\in V$ (so $x_m\notin U$) for almost all $m\in\omega$ and $f^{-1}[\{x_m\}]\in\I_\rho$ for all $m\in\omega$, $f^{-1}[U]\in\I_\rho$. Hence, $f\restriction\rho(F)$ cannot $\rho$-converge to $x_n$.

If $L=x$, we can find $n\in\omega$ such that $\rho(F)\not\subseteq^\rho A_n$. Since $X$ is Hausdorff, there is an open neighbourhood $U$ of $L$ such that $x_{i+1}\notin U$ for all $i<n$. Since $f\restriction\rho(F)$ $\rho$-converges to $L$, there should be a finite $K\subseteq\Omega$ such that $f[\rho(F\setminus K)]\subseteq U$, however $\rho(F\setminus K)\setminus A_n\neq\emptyset$ (by $\rho(F)\not\subseteq^\rho A_n$), so $f[\rho(F\setminus K)]\cap\{x_{i+1}:i<n\}\neq\emptyset$, which contradicts $x_{i+1}\notin U$ for all $i<n$.

(\ref{thm:compact-metric-implies-FinBW(rho):Pminus-is-necessary-hFinBW})
(a)$\implies$(b) Notice that if $X$ is the space defined in the proof of the implication  $(\ref{thm:compact-metric-implies-FinBW(rho):Pminus-is-necessary}a)\implies (\ref{thm:compact-metric-implies-FinBW(rho):Pminus-is-necessary}b)$ then $X\in\FinBW(\rho)$ for every $\rho$ that is $P^-(\Lambda)$ (the definition of $X$ did not depend on $\rho$). Thus, if $\rho$ is $P^-$ then $\rho\restriction \rho(F)$ is $P^-(\rho(F))$ for every $F\in\cF$ and consequently $X\in\bigcap_{F\in\cF}\FinBW(\rho\restriction \rho(F))=\hFinBW(\rho)$ (by Proposition \ref{prop:basic-relationships-between-FinBW-like-spaces}(\ref{prop:basic-relationships-between-FinBW-like-spaces:hFinBW-equals-intersection-of-FinBW:rho})).

(b)$\implies$(c) Obvious.

(c)$\implies$(a) If $\rho$ is not $P^-$ then $\rho\restriction \rho(F)$ is not $P^-(\rho(F))$ for some $F\in\cF$. Hence, by item (\ref{thm:compact-metric-implies-FinBW(rho):Pminus-is-necessary}), $\FinBW(\rho\restriction\rho(F))$ contains only finite  Hausdorff spaces. Since $\hFinBW(\rho)\subseteq\FinBW(\rho\restriction\rho(F))$ by Proposition \ref{prop:basic-relationships-between-FinBW-like-spaces}(\ref{prop:basic-relationships-between-FinBW-like-spaces:hFinBW-equals-intersection-of-FinBW:rho}), $\hFinBW(\rho)$ also contains only finite  Hausdorff spaces.

(\ref{thm:Pminus-implies-various-spaces-in-FinBW(rho):omega-1})
Let $f:\Lambda\to \omega_1$. If there is $\alpha<\omega_1$ with $f^{-1}[\{\alpha\}]\notin\I_\rho$, then we take $F\in \cF$ such that $\rho(F)\subseteq f^{-1}[\{\alpha\}]$ and see that $f\restriction \rho(F)$ is $\rho$-convergent to $\alpha$.
Now, we assume that $f^{-1}[\{\alpha\}]\in \I_\rho$ for every $\alpha <\omega_1$.
Since $\Lambda$ is countable and the cofinality of $\omega_1$ is uncountable, there is $\alpha<\omega_1$ with  $f^{-1}[\alpha]\notin \I_\rho$.
Let $\alpha_0$ be the smallest $\alpha$ such that $f^{-1}[\alpha]\notin \I_\rho$.
Note that $\alpha_0$ is a limit ordinal. Indeed, if $\alpha_0=\alpha+1$, then $\alpha<\alpha_0$ and $f^{-1}[\alpha] = f^{-1}[\alpha_0]\setminus f^{-1}[\{\alpha\}]\notin \I_\rho$, a contradiction.
Since $\alpha_0$ is a countable limit ordinal, there is an increasing sequence $\{\beta_n:n\in\omega\}$ such that $\sup \{\beta_n:n\in\omega\}=\alpha_0$.
By Proposition~\ref{prop:Pminus-for-rho-equivalent-condition}, there is $F\in \cF$ such that
$\rho(F)\subseteq f^{-1}[\alpha_0]$
and for each $n\in\omega$ there is a finite set $K_n$ such that $\rho(F\setminus K_n)\cap f^{-1}[\beta_n]=\emptyset$.
We claim that $f\restriction \rho(F)$ is $\rho$-convergent to $\alpha_0$.
Indeed, let $U$ be a neighborhood of $\alpha_0$. Without loss of generality, we can assume that $U = (\alpha_0+1)\setminus \beta_n$ for some $n\in\omega$. 
Then $f[\rho(F\setminus K_n)] \subseteq \alpha_0\setminus \beta_n \subseteq U$.

(\ref{thm:Pminus-implies-various-spaces-in-FinBW(rho):Mrowka-separable})
Spaces with these properties are  constructed in Theorem~\ref{thm:TWIERDZENIE-for-Fspaces-in-FinBW}.

(\ref{thm:compact-metric-implies-FinBW(rho):metric-compact-spaces})
Let   $f:\rho(E)\to X$ be a $\rho$-sequence in a metric compact space $X$. 

Since $\rho$ is  weak $P^+$, there exists $F\in \cF$ such that $\rho(F)\subseteq \rho(E)$ and for every  sequence $\{F_n: n\in\omega\}\subseteq \cF$ such that $\rho(F) \supseteq \rho(F_n)\supseteq \rho(F_{n+1})$ for each $n\in\omega$ 
there exists $G\in\cF$ 
such that $\rho(G)\subseteq\rho(F)$ and  $\rho(G)\subseteq^\rho \rho(F_n)$
for each $n\in\omega$.

For $x\in X$ and $r>0$ we write  $B(x,r)$ and $\overline{B}(x,r)$ to denote an open and closed ball of radius  $r$ centered at a point $x$, respectively.

Since $X$ is compact metric, there are finitely many $x_i^0\in X$, $i<n_0$ such that 
$X=\bigcup\{B(x_i^0,1):i<n_0\}$.
Then there exists  $i_0<n_0$  such that  
$\rho(F)\cap f^{-1}[B(x^0_{i_0},1)] \notin \I_{\rho}$, 
and consequently there is $F_0\in \cF$ such that 
$\rho(F_0)\subseteq \rho(F)\cap f^{-1}[B(x^0_{i_0},1)]$.

Since $\overline{B}(x^0_{i_0},1)$ is compact metric, there are finitely many $x_i^1\in X$, $i<n_1$ such that 
$\overline{B}(x^0_{i_0},1)\subseteq \bigcup\{B(x_i^1,\frac{1}{2}):i<n_1\}$.
Then there exists  $i_1<n_1$  such that  
$\rho(F_0)\cap f^{-1}[B(x^1_{i_1},\frac{1}{2})] \notin \I_{\rho}$, 
and consequently there is $F_1\in \cF$ such that 
$\rho(F_1)\subseteq \rho(F_0)\cap f^{-1}[B(x^1_{i_1},\frac{1}{2})]$.

If we continue the above procedure, we obtain $F_n\in \cF$ and $x^n_{i_n}\in X$ such that 
$\rho(F_{n})\subseteq \rho(F_{n-1}) \cap f^{-1}[B(x^n_{i_n},\frac{1}{n+1})]$
for each $n\in\omega$ (assuming that $F_{-1}=F$).

Let $x\in \bigcap\{\overline{B}(x^n_{i_n},\frac{1}{n+1}):n\in\omega\}$.

Since $\rho$ is  weak $P^+$, we have  $G\in \cF$ such that $\rho(G)\subseteq\rho(F)$ and  $\rho(G)\subseteq^\rho 
\rho(F_n)$ for each $n\in\omega$.

We claim that $f\restriction \rho(G)$ is $\rho$-convergent to $x$.
Let $U$ be a neighborhood of $x$. 
Since the sequence $(x^n_{i_n})_{n\in\omega}$ is  convergent to $x$, there is $n_0\in\omega$ such that $B(x^n_{i_n},\frac{1}{n+1})\subseteq U$ for every $n\geq n_0$.
Consequently, there is $n\in \omega$ with $B(x^n_{i_n},\frac{1}{n+1})\subseteq U$.
Let $K\subseteq\Omega$ be a finite set such that $\rho(G\setminus K)\subseteq\rho(F_n)$.
Then 
$f[\rho(G\setminus K)]\subseteq f[\rho(F_n)] \subseteq B(x^n_{i_n},\frac{1}{n+1}) \subseteq U,$
so the proof is finished.

(\ref{thm:compact-metric-implies-FinBW(rho):spaces-with-star-property})
Let $E\in \cF$ and $f:\rho(E)\to X$ be a sequence in a Hausdorff topological space $X$ having the property $(*)$.
Since  the set $D=\{f(\lambda):\lambda\in \rho(E)\}$ is countable, the closure  $\closure_X(D)$ is compact and first-countable.
We claim that  there exists $L\in \closure_X(D)$ such that 
$f^{-1}[U] \in \I_{\rho}^+$
for every neighborhood $U$ of $L$.

Suppose, for sake of contradiction, that for every $x\in \closure_X(D)$ there is a neighborhood $U_x$ of $x$ such that 
$f^{-1}[U_x]\in \I_{\rho}$.
Since $\closure_X(D)$ is compact,  there are finitely many $x_i\in \closure_X(D)$ for $i<n$ with 
$\closure_X(D) \subseteq \bigcup\{ U_{x_i} : i<n\}.$
Then 
$\rho(E) =  \bigcup \{f^{-1}[U_{x_i}]:i<n\} \in \I_{\rho},$
a contradiction, so the claim is proved.

Let $\{U_n:n\in\omega\}$ be a base at $L$. Without loss of generality, we can assume that $U_n\supseteq U_{n+1}$ for each $n\in \omega$.
For each $n\in \omega$, we define $A_n=\{\lambda\in \rho(E): f(\lambda)\in U_n\}$.
Since $A_n\in \I_{\rho}^+$ and $A_n\supseteq A_{n+1}$ for each $n\in \omega$, using the fact that $\rho$ is $P^+$, there exists $F\in \cF$ such that $\rho(F)\subseteq \rho(E)$ and 
$\rho(F)\subseteq^\rho A_n$ for each $n\in \omega$.
We claim  that $f\restriction \rho(F)$ is $\rho$-convergent to $L$.

Take any neighborhood $U$ of $L$. Then there exists $n_0\in\omega$ with $U_{n_0}\subseteq U$.
Since $\rho(F)\subseteq^\rho A_{n_0}$, there exists a finite set $K\subseteq\Omega$ such that $\rho(F\setminus K) \subseteq A_{n_0}$.
Thus $f[\rho(F\setminus K)] \subseteq U_{n_0}\subseteq U$, so the proof is finished.
\end{proof}

The following series of corollaries shows that many known earlier results can be easily derived from Theorem~\ref{thm:FinBW-hFinBW-homogeneous}.

\begin{corollary}[{\cite[Proposition~2.4]{MR3276758}}]
If an ideal $\I$ is not tall, then $\FinBW(\I)$ coincides with the class of  all sequentially compact spaces.
\end{corollary}

\begin{proof}
It follows from Theorem~\ref{thm:FinBW-hFinBW-homogeneous}(\ref{thm:FinBW-hFinBW-homogeneous:-equality-for-SeqComp})
and 
Propositions~\ref{prop:basic-relationships-between-FinBW-like-spaces}(\ref{prop:basic-relationships-between-FinBW-like-spaces:ideal-rho}) and \ref{prop:prop-of-talness}.
\end{proof}

\begin{corollary}[{\cite[Theorem~6.5]{MR4584767}}]
    $\Fin^2\leq_K\I$ $\iff$ $\FinBW(\I)$ coincides with the class of all finite spaces in the realm of Hausdorff spaces. 
\end{corollary}

\begin{proof}
It follows from Theorem~\ref{thm:compact-metric-implies-FinBW(rho)}(\ref{thm:compact-metric-implies-FinBW(rho):Pminus-is-necessary}) and 
Propositions~\ref{prop:Pminus-versus-FinSQRD-for-ideal}(\ref{prop:Pminus-versus-FinSQRD-for-ideal:Pminus-Lambda}),\ref{prop:properties-of-ideal-versus-rho}(\ref{prop:properties-of-ideal-versus-rho:ideal-rho}),\ref{prop:basic-relationships-between-FinBW-like-spaces}(\ref{prop:basic-relationships-between-FinBW-like-spaces:ideal-rho}).
\end{proof}

\begin{corollary}
Every metric compact space belongs to $\hFinBW(\rho)$ in case when
\begin{enumerate}

 \item 
$\rho=\rho_\I$ and $\I$ is $P^+$ ideal,

 \item 
 \cite[Theorem~2.3]{MR2471564}
$\rho=\rho_\I$ and $\I$ is an $F_\sigma$ ideal,

\item  
\cite[Theorem~10]{MR1866012}
$\rho=\rho_\I$ and $\I = \vdW$,

\item  
\cite[Theorem~2.3]{MR2471564}
$\rho=\rho_\I$ and $\I = \I_{1/n}$,

\item 
\cite[Theorem~2.5]{MR531271}
$\rho=\FS$,

\item
\cite[Theorem~1]{MR2948679} (see also \cite[Theorem~1.16]{BergelsonZelada})
$\rho = r$,

\item 
\cite[Corollary~4.8]{MR3097000}
$\rho = \Delta$.

\end{enumerate}
\end{corollary}

\begin{proof}
It follows from Theorem~\ref{thm:compact-metric-implies-FinBW(rho)}(\ref{thm:compact-metric-implies-FinBW(rho):metric-compact-spaces}),
Propositions~\ref{prop:Plike-properties-for-known-rho} and \ref{prop:properties-of-ideal-versus-rho}.
\end{proof}

\begin{corollary}
Every Hausdorff space with the property $(*)$  belongs to $\hFinBW(\rho)$ in case when
\begin{enumerate}

 \item 
$\rho=\rho_\I$ and $\I$ is $P^+$ ideal,

 \item 
 \cite[Theorem~2.3]{MR2471564}
$\rho=\rho_\I$ and $\I$ is an  $F_\sigma$ ideal,

\item  
\cite[Theorem~10]{MR1866012}
$\rho=\rho_\I$ and $\I = \vdW$,

\item  
\cite[Theorem~2.3]{MR2471564}
$\rho=\rho_\I$ and $\I = \I_{1/n}$.

\end{enumerate}
\end{corollary}

\begin{proof}
It follows from Theorem~\ref{thm:compact-metric-implies-FinBW(rho)}(\ref{thm:compact-metric-implies-FinBW(rho):spaces-with-star-property}), Propositions~\ref{prop:Plike-properties-for-known-rho} and \ref{prop:properties-of-ideal-versus-rho}.
\end{proof}

In \cite[Theorem~11]{MR1887003} (\cite[Corollary~3.2]{MR4552506}, resp.), the authors proved that 
every Hausdorff space with the property $(*)$  belongs to $\hFinBW(\FS)$ ($\hFinBW(r)$, resp.).
However, their proofs use properties very specific  to $\FS$ and $r$. For instance, the proof for $\FS$ uses idempotent ultrafilters whereas the proof for $r$ uses the bounding number $\mathfrak{b}$. 

\begin{problem}
Find a property $W$ of partition regular functions such that 
both  $\FS$ and $r$ have the property $W$
and if $\rho$ has the property $W$ then 
every Hausdorff space with the property $(*)$  belongs to $\hFinBW(\rho)$.
\end{problem}

In \cite[Corollary~4.8]{MR3097000}, the author proved that 
every Hausdorff space with the property $(*)$  belongs to $\FinBW(\Delta)$.

\begin{question}
\label{q:star-is-hFinBW-Delta}
Does every Hausdorff space with the property $(*)$  belong to $\hFinBW(\Delta)$?
\end{question}

Note that the positive answer to Question \ref{q:FinBW=hFinBW-for-summable}(\ref{q:FinBW=hFinBW-for-differentially-compact:item}) gives the positive answer to Question~\ref{q:star-is-hFinBW-Delta}.


\section{Inclusions between  \texorpdfstring{$\FinBW$}{FinBW } classes}

\begin{theorem}
\label{thm:Katetove-implies-inclusion-for-FinBW}
Let $\rho_i:\cF_i\to[\Lambda_i]^\omega$ be  partition regular with $\cF_i\subseteq [\Omega_i]^\omega$ for each $i=1,2$. Let $\I$ be an ideal on $\Lambda$.
\begin{enumerate}
\item 
$\rho_2\leq_K \rho_1 \implies 
\FinBW(\rho_1) \subseteq   \FinBW(\rho_2)$.\label{thm:Katetove-implies-inclusion-for-FinBW:for-rho}

\item 
\begin{enumerate}
    \item If $\rho_2$ is $P^+$, then \label{thm:Katetove-implies-inclusion-for-FinBW:for-rho-ideals:Pplus}
$$\I_{\rho_2}\leq_K \I_{\rho_1} \implies \FinBW(\rho_1) \subseteq   \FinBW(\rho_2).$$
\item $\I_{\rho_2}\leq_K \I \implies \FinBW(\I)\subseteq \FinBW(\rho_2)$.\label{thm:Katetove-implies-inclusion-for-FinBW:for-ideals}
\end{enumerate}

\end{enumerate}
\end{theorem}

\begin{proof}
(\ref{thm:Katetove-implies-inclusion-for-FinBW:for-rho})
Let $\phi:\Lambda_1\to \Lambda_2$ be a witness for 
$\rho_2\leq_K \rho_1$.
Let $X\in \FinBW(\rho_1)$.
If  $f:\Lambda_2\to X$, then $f\circ \phi : \Lambda_1\to X$, so there is $F_1\in \cF_1$ such that 
$\rho_1(F_1)\subseteq \Lambda_1$ 
and 
$(f\circ \phi)\restriction \rho_1(F_1)$ 
is $\rho_1$-convergent to some $x\in X$.

Let   $F_2\in \cF_2$ be 
such that for every finite $K_1\subseteq\Omega_1$ there is a finite $K_2\subseteq \Omega_2$ with $\rho_2(F_2\setminus K_2)\subseteq \phi[\rho_1(F_1\setminus K_1)]$.

We claim that $f\restriction\rho(F_2)$ is $\rho_2$-convergent to  $x$.
Let  $U$ be a neighborhood of $x$. Since $(f\circ \phi)\restriction \rho_1(F_1)$ 
is $\rho_1$-convergent to $x$, then there is a finite set $K_1\subseteq \Omega_1$ such that $(f\circ \phi)[\rho_1(F_1\setminus K_1)]\subseteq U$. Hence, we can find a finite $K_2\subseteq \Omega_2$ with $\rho_2(F_2\setminus K_2)\subseteq \phi[\rho_1(F_1\setminus K_1)]$. 
Then $f[\rho_2(F_2\setminus K_2)] \subseteq 
f[\phi[\rho_1(F_1\setminus K_1)]]
\subseteq U$.
That finishes the proof.

(\ref{thm:Katetove-implies-inclusion-for-FinBW:for-rho-ideals:Pplus})
It follows from
Proposition~\ref{prop:Katetov-for-rho-functions}(\ref{prop:Katetov-for-rho-functions:equivalence}) and 
item
(\ref{thm:Katetove-implies-inclusion-for-FinBW:for-rho}).

(\ref{thm:Katetove-implies-inclusion-for-FinBW:for-ideals})
It follows from  item (\ref{thm:Katetove-implies-inclusion-for-FinBW:for-rho}) and Propositions~\ref{prop:Katetov-for-ideal-rho}(\ref{prop:Katetov-for-rho-function-and-ideal-rho:sparse}) and \ref{prop:basic-relationships-between-FinBW-like-spaces}(\ref{prop:basic-relationships-between-FinBW-like-spaces:ideal-rho}).
\end{proof}

The following series of corollaries shows that many known earlier results as well as some new  one can be easily derived from Theorem~\ref{thm:Katetove-implies-inclusion-for-FinBW}.

\begin{corollary}\ 
\label{cor:inclusion-between-known-spaces}
\begin{enumerate}
        \item \cite[p.~39]{Shi2003Numbers} Every Hindman space is differentially compact.\label{cor:inclusion-between-known-spaces:Hindman-Diff}
        \item Every Ramsey space is differentially compact.\label{cor:inclusion-between-known-spaces:Ramsey-Diff}
        \item Let $\I$ be a $P^+$ ideal.\label{cor:inclusion-between-known-spaces:I-rho}
        \begin{enumerate}
            \item \cite[Proposition 2.6]{MR4356195} If $\I\leq_K\Hindman$ then every Hindman space is in $\FinBW(\I)$.
            \item If $\I\leq_K\Ramsey$ then every Ramsey space is in $\FinBW(\I)$.
            \item If $\I\leq_K\Diff$ then every differentially compact space is in $\FinBW(\I)$.
        \end{enumerate} 
    \item \cite[Corollary~10.2(a)]{MR4584767} If $\I_i$ are ideals for $i=1,2$  and $\I_2\leq_K\I_1$, then $\FinBW(\I_1)\subseteq \FinBW(\I_2)$.\label{cor:inclusion-between-known-spaces:ideals}
    \end{enumerate}
\end{corollary}

\begin{proof}
    (\ref{cor:inclusion-between-known-spaces:Hindman-Diff})
It follows from Theorems~\ref{thm:Katetove-implies-inclusion-for-FinBW}(\ref{thm:Katetove-implies-inclusion-for-FinBW:for-rho}) and \ref{thm:Katetov-between-known-rho}(\ref{thm:Katetov-between-known-rho:Diff-below-Hindman}).

(\ref{cor:inclusion-between-known-spaces:Ramsey-Diff})
It follows from Theorems~\ref{thm:Katetove-implies-inclusion-for-FinBW}(\ref{thm:Katetove-implies-inclusion-for-FinBW:for-rho}) and \ref{thm:Katetov-between-known-rho}(\ref{thm:Katetov-between-known-rho:Diff-below-Ramsey}).

(\ref{cor:inclusion-between-known-spaces:I-rho})
It follows from Theorem~\ref{thm:Katetove-implies-inclusion-for-FinBW}(\ref{thm:Katetove-implies-inclusion-for-FinBW:for-rho-ideals:Pplus}) and Propositions~\ref{prop:properties-of-ideal-versus-rho}(\ref{prop:properties-of-ideal-versus-rho:ideal-rho}) and \ref{prop:basic-relationships-between-FinBW-like-spaces}(\ref{prop:basic-relationships-between-FinBW-like-spaces:ideal-rho}).

(\ref{cor:inclusion-between-known-spaces:ideals}) It follows from Theorem~\ref{thm:Katetove-implies-inclusion-for-FinBW}(\ref{thm:Katetove-implies-inclusion-for-FinBW:for-ideals}) and Proposition~\ref{prop:basic-relationships-between-FinBW-like-spaces}(\ref{prop:basic-relationships-between-FinBW-like-spaces:ideal-rho}). 
\end{proof}

\begin{corollary}\
\label{cor:compact-metric-implies-FinBW(rho):metric-compact-spaces}
\label{cor:compact-metric-implies-FinBW(rho):spaces-with-star-property}

\begin{enumerate}

\item Let $\rho:\cF_i\to[\Lambda]^\omega$ be a partition regular function.\label{cor:compact-metric-implies-FinBW(rho):metric-compact-spaces:FinBW}\label{cor:compact-metric-implies-FinBW(rho):spaces-with-star-property:FinBW}
\begin{enumerate}
\item If $\rho\leq_K\rho'$ for some weak $P^+$ partition regular function $\rho'$, then every compact metric space belongs to $\FinBW(\rho)$.

\item If $\rho\leq_K\rho'$ for some $P^+$ partition regular function $\rho'$, then every Hausdorff topological space with the property $(*)$ belongs to $\FinBW(\rho_2)$.
\end{enumerate}

\item Let $\I$ be an ideal. If an ideal $\I$ can be extended to a $P^+$ ideal, then:\label{cor:compact-metric-implies-FinBW(I)}\label{cor:star-implies-FinBW(I)}

\begin{enumerate}
\item every compact metric space belongs to $\FinBW(\I)$;

\item \cite[Corollary 5.6]{MR2961261} every Hausdorff topological space with the property $(*)$ belongs to $\FinBW(\I)$.
\end{enumerate}

\end{enumerate}
\end{corollary}

\begin{proof}
(\ref{cor:compact-metric-implies-FinBW(rho):metric-compact-spaces:FinBW}) 
It follows from Theorems~\ref{thm:Katetove-implies-inclusion-for-FinBW}(\ref{thm:Katetove-implies-inclusion-for-FinBW:for-rho}), \ref{thm:compact-metric-implies-FinBW(rho)}(\ref{thm:compact-metric-implies-FinBW(rho):metric-compact-spaces})(\ref{thm:compact-metric-implies-FinBW(rho):spaces-with-star-property})
and Proposition~\ref{prop:basic-relationships-between-FinBW-like-spaces}(\ref{prop:basic-relationships-between-FinBW-like-spaces:inclusion}).

(\ref{cor:compact-metric-implies-FinBW(I)})
It follows from item (\ref{cor:compact-metric-implies-FinBW(rho):metric-compact-spaces:FinBW}), 
Propositions~\ref{prop:properties-of-ideal-versus-rho}(\ref{prop:properties-of-ideal-versus-rho:ideal-rho}),
\ref{prop:Katetov-for-ideal-rho}(\ref{prop:Katetov-for-ideal-rho:item}),
\ref{prop:basic-relationships-between-FinBW-like-spaces}(\ref{prop:basic-relationships-between-FinBW-like-spaces:ideal-rho}).
\end{proof}

The following proposition shows that when comparing classes $\FinBW(\rho_1)$ and $\FinBW(\rho_2)$ for distinct functions $\rho_1$ and $\rho_2$ we can in fact assume that both $\rho_1$ and $\rho_2$ ``live'' on the same sets $\Omega$ and  $\Lambda$.

\begin{proposition}
\label{prop:partition-regular-functions-live-on-the-same-set}
Let $\rho:\cF\to[\Lambda]^\omega$ be  partition regular with $\cF\subseteq [\Omega]^\omega$. Suppose that $\Gamma$ and $\Sigma$ are  countable infinite sets and $\phi:\Omega\to \Gamma$ and $\psi:\Lambda\to \Sigma$ are bijections.
Let $\cG=\{\phi[F]:F\in \cF\}$ and $\tau:\cG\to[\Sigma]^\omega$ be given by
$\tau(G) = \psi[\rho(\phi^{-1}[G])]$.
Then 
\begin{enumerate}
    \item $\tau$ is partition regular,\label{prop:partition-regular-functions-live-on-the-same-set:tau}
    \item $\rho$ and $\tau$ are Kat\v{e}tov equivalent: $\rho\leq_K\tau$ and $\tau\leq_K\rho$,\label{prop:partition-regular-functions-live-on-the-same-set:Kat-equivalent}
    \item $\rho$ and $\tau$ are hereditary Kat\v{e}tov equivalent:\label{prop:partition-regular-functions-live-on-the-same-set:hereditary} 
\begin{enumerate}
\item $\forall F\in \cF\, \exists G\in \cG \, (\rho\restriction  \rho(F) \leq_K \tau\restriction  \tau(G)),$
\item $\forall G\in \cG\, \exists F\in \cF \, (\tau\restriction  \tau(G) \leq_K \rho\restriction  \rho(F)),$
\end{enumerate}
 \item $\FinBW(\tau)=\FinBW(\rho)$ and $\hFinBW(\tau)=\hFinBW(\rho)$.\label{prop:partition-regular-functions-live-on-the-same-set:FinBW}
\end{enumerate}
\end{proposition}

\begin{proof}
Items (\ref{prop:partition-regular-functions-live-on-the-same-set:tau})--(\ref{prop:partition-regular-functions-live-on-the-same-set:hereditary}) are straightforward,
whereas item~(\ref{prop:partition-regular-functions-live-on-the-same-set:FinBW}) follows from previous items and Proposition~\ref{thm:Katetove-implies-inclusion-for-FinBW}.
\end{proof}


\part{Distinguishing between \texorpdfstring{$\FinBW$}{FinBW } classes}
\label{part:Distinguishing-FinBW-spaces}

In this part we are  interested in finding spaces that are in  $\FinBW(\rho_1)$, but are not in $\FinBW(\rho_2)$. Similar investigations concerning the classes $\FinBW(\I)$ were conducted in \cite{MR4584767}. In that paper all the examples (showing that under some set-theoretic assumption $\FinBW(\I)\setminus\FinBW(\J)\neq\emptyset$ for some ideals $\I$ and $\J$) were inspired by \cite{MR1950294} and are of one specific type -- they are defined using \emph{maximal} almost disjoint families. 
It turns out (see Theorem~\ref{thm:MAD-Mrowka-nie-jest-Fspejsem}) that, in general, we cannot use \emph{maximal} almost disjoint families to distinguish between $\FinBW(\rho)$ classes with the aid of spaces defined as in \cite{MR4584767}.
Fortunately, we managed to use \emph{not necessary maximal} almost disjoint families to  prove two main results of this part (Theorems \ref{thm:TWIERDZENIE-for-Fspaces:Pminus-Mrowka} and \ref{thm:TWIERDZENIE-for-Fspaces:Pminus-Mrowka:rho}), which give us $\FinBW(\rho_1)\setminus\FinBW(\rho_2)\neq\emptyset$ for certain $\rho_1$ and $\rho_2$. Our methods were inspired by \cite{MR4552506}.


\section{Mr\'{o}wka spaces and their compactifications}

For an \emph{infinite} almost disjoint family $\cA$ on a countable set $\Lambda$, we define 
a set  
$$\Psi(\cA) = \Lambda\cup \cA$$
and introduce a topology on $\Psi(\cA)$ as follows:
 the points of $\Lambda$
are isolated and a basic neighborhood of $A\in\cA$ has the form $\{A\}\cup
(A\setminus F)$ with $F$ finite. 

Topological spaces of the form  $\Psi(\cA)$ were  introduced by Alexandroff and Urysohn in
\cite[chapter V, paragraph 1.3]{zbMATH02573679} (as noted in \cite[p.~182]{MR1039321}, \cite[p.~1380]{MR2610447} and \cite[p.~605]{MR3205494}) and its topology is known as  the \emph{rational sequence topology} (see \cite[Example 65]{MR1382863}; the same topology was later described by Kat\v{e}tov in \cite[p.~74]{MR48531}).
Spaces $\Psi(\cA)$ with \emph{maximal} (with respect to the inclusion) almost disjoint families $\cA$  were first examined by  
Mr\'{o}wka (see \cite{MR63650}) and Isbell (as noted in \cite[p.~269]{MR0116199}).
It seems that the notation $\Psi$ for these kind of spaces was  used for the first time in \cite[Problem 5I, p.~79]{MR0116199}.

Spaces of the form $\Psi(\cA)$ are known under
many names, including \emph{$\Psi$-spaces}, \emph{Isbell-Mr\'{o}wka spaces} and \emph{Mr\'{o}wka spaces}. 
Recent surveys on these spaces and their numerous applications can be found in \cite{MR3205494,MR3822423}.

It is known that $\Psi(\cA)$ is Hausdorff, regular, locally compact,
first countable and separable, but it is  not
compact nor sequentially compact (see \cite{MR63650} or \cite[Section~11]{MR776622}).
It is not difficult to see that 
$A\cup\{A\}$ is compact in $\Psi(\cA)$ for every $A\in\cA$ and for every 
compact set $K\subseteq\Psi(\cA)$ both sets 
$K\cap \cA$ and $(K \cap \Lambda) \setminus \bigcup \{A : A\in K\cap \cA\}$  are finite.
In particular, for every  compact set $K\subseteq\Psi(\cA)$ there are finitely many sets  $A_i\in \cA$ and a finite set $F$ such that 
$ K\subseteq \{A_i:i<n\}\cup \bigcup\{ A_i \cup F: i<n\}$.  

Let 
$$\Phi(\cA)  = \Psi(\cA)\cup \{\infty\} = \Lambda \cup \cA \cup \{\infty\}$$ 
be the Alexandroff one-point compactification of $\Psi(\cA)$ (recall that open neighborhoods of $\infty$ are of the form $\Phi(\cA)\setminus K$ for compact sets $K\subseteq \Psi(\cA)$).
It is not difficult to see  
that  $\Phi(\cA)$ is Hausdorff, compact, sequentially compact, separable and first countable at every point of $\Phi(\cA)\setminus\{\infty\}$.

Topological spaces of the form $\Phi(\cA)$ with \emph{maximal} almost disjoint families $\cA$ were first used by Franklin \cite[Example~7.1]{MR222832} where the author used the notation $\Psi^*$ instead of $\Phi$.
Later, these spaces were considered in \cite{MR2375170} where the authors use the notation $\cF(\cA)$ and call them 
the \emph{Franklin compact spaces} associated to   $\cA$, whereas in  \cite{MR4330212} the authors use the notation $Fr(\cA)$ and call them the \emph{Franklin spaces of $\cA$}.
The notation $\Phi(\cA)$ for these spaces is used in the following  papers \cite{MR3097000,MR4356195,MR4358658,MR4584767}
Recently, spaces of the form $\Phi(\cA)$ were also considered  for \emph{non} maximal almost disjoint families \cite{MR4552506,corral-guzman-lopez}.

It also makes sense to define $\Phi(\cA)$ for infinite families $\cA$ that  are \emph{not}  almost disjoint, but then $\Phi(\cA)$ is no longer  Hausdorff (almost  disjointness of $\cA$ is a necessary and sufficient condition for a space $\Phi(\cA)$ to be Hausdorff). 

The following lemma (which will be used repeatedly in the sequel) shows that 
a sequence in a space $\Phi(\cA)$ may fail to have a $\rho$-convergent $\rho$-subsequence only in one specific case. Hence, checking whether $\Phi(\cA)\in \FinBW(\rho)$ will be reduced to considering only sequences of this one specific kind.

\begin{lemma}
\label{lem:TWIERDZENIE-for-Fspaces-in-FinBW}
Let $\rho:\cF\to[\Lambda]^\omega$ be  partition regular with $\cF\subseteq [\Omega]^\omega$.
Let $\cA$ be an infinite almost disjoint family  on $\Lambda$.
For every sequence  
$f:\Lambda\to \Phi(\cA)$,  the following five cases can only occur: 
\begin{enumerate}
    \item $f^{-1}(\infty)\notin\I_{\rho}$,
    \item $f^{-1}[\cA] \notin\I_{\rho}$,
    \item $f^{-1}(\infty)\in\I_{\rho}$, $f^{-1}[\cA] \in\I_{\rho}$, 
    $f^{-1}[\Lambda]\in \I_{\rho}^*$ and
\begin{enumerate}
    \item $f^{-1}(\lambda)\notin \I_{\rho}$ for some $\lambda\in \Lambda$,
    \item $f^{-1}(\lambda)\in \I_{\rho}$ for every $\lambda\in \Lambda$
    and
    $f^{-1}[A]\notin \I_{\rho}$ for some $A \in \cA$,
    \item 
    $f^{-1}(\lambda)\in \I_{\rho}$ for every $\lambda\in \Lambda$
    and 
    $f^{-1}[A]\in \I_{\rho}$ for every  $A \in \cA$.
\end{enumerate}
\end{enumerate}

If $\rho$ is  $P^-$, 
then in cases (1), (2), (3a) and (3b) there is $F\in \cF$
such that $f\restriction \rho(F)$ is $\rho$-convergent.
\end{lemma}

\begin{proof}
\emph{Case (1).} 
There is $F\in \cF$ such that 
$f\restriction \rho(F)$ is constant (with the value $\infty$), hence it is $\rho$-convergent.

\emph{Case (2).}  
We find $F\in \cF$ with $\rho(F)\subseteq f^{-1}[\cA]$.
Then we enumerate $f[\rho(F)] = \{A_{n} : n \in\omega\}$
and define $E_n = f^{-1}[\{A_{n}\}]$ for each $n\in\omega$.

If there is $n_0\in\omega$ such that $E_{n_0}\notin \I_{\rho}$, then we 
find $F'\in \cF$ with   
$\rho(F')\subseteq E_{n_0}$, and we see that  $f\restriction\rho(F')$ is constant, so it is $\rho$-convergent.

Now assume that $E_n\in \I_{\rho}$ for each $n\in\omega$.
Since $\rho(F)\subseteq \bigcup\{E_n:n\in\omega\}$, we can use 
Proposition~\ref{prop:Pminus-for-rho-equivalent-condition} to find  $E\in\cF$ 
such that 
for each $n\in\omega$ there is a finite set $K\subseteq \Omega$ with $\rho(E\setminus K) \subseteq   \rho(F) \cap \bigcup \{ E_i: i\geq n\}$.  
We claim that  $f\restriction \rho(E)$  is $\rho$-convergent to $\infty$.
Let $U$ be a neighborhood of $\infty$. Without loss of generality, we can assume that 
$U = \Phi(\cA)\setminus (\{A_{i}:i<n\}\cup \bigcup_{i<n}A_{i})$
for some $n\in\omega$.
Let $K\subseteq \Omega$ be a finite set such that  $\rho(E\setminus K) \subseteq  \rho(F)\cap \bigcup \{ E_i: i\geq n\}$. 
Then 
$f[\rho(E\setminus K)] \cap \{A_{i}:i<n\}=\emptyset$, and consequently
$f[\rho(E\setminus K)] \subseteq U$.

\emph{Case (3a).} 
There is $F\in \cF$ such that 
$f\restriction \rho_1(F)$ is constant, hence it is $\rho$-convergent.

\emph{Case (3b).} 
Let $A\in \cA$ be such that $f^{-1}[A]\notin \I_{\rho}$.
Using Proposition~\ref{prop:Pminus-for-rho-equivalent-condition},
we find $F\in \cF$ 
such that 
$\rho(F)\subseteq f^{-1}[A]$
and 
for every finite $S\subseteq A$ there is a finite set $K\subseteq \Omega$ with 
$\rho(F\setminus K)\subseteq f^{-1}[A]\setminus f^{-1}[S]=f^{-1}[A\setminus S]$.
We  claim that $f\restriction \rho(F)$ is $\rho$-convergent to $A$.
Indeed, let $U$ be a neighborhood of $A$. Without loss of generality, we can assume that $U=\{A\}\cup (A\setminus S)$ where $S$ is a finite subset of $A$.
Then we take a finite $K\subseteq\Omega$ such that 
$\rho(F\setminus K)\subseteq f^{-1}[A\setminus S]$, so
$f[\rho(F\setminus K)]\subseteq A\setminus S \subseteq U$.
\end{proof}

\begin{proposition}
\label{prop:countable-Mrowka-is-FinBW}
Let $\rho:\cF\to[\Lambda]^\omega$ be  partition regular with $\cF\subseteq [\Omega]^\omega$.
Let $\cA$ be an infinite almost disjoint family  
of infinite subsets of $\Lambda$.
If $\rho$ is  $P^-$ 
and $\cA$ is \emph{countable}, 
then $\Phi(\cA)\in \FinBW(\rho)$.
\end{proposition}

\begin{proof}
Let 
$f:\Lambda\to \Phi(\cA)$.
By Lemma~\ref{lem:TWIERDZENIE-for-Fspaces-in-FinBW}, 
we can assume that 
 $f^{-1}(\infty)\in\I_{\rho}$, 
 $f^{-1}[\cA] \in\I_{\rho}$,
$f^{-1}[\Lambda]\in \I_{\rho}^*$,  
$f^{-1}(\lambda)\in \I_{\rho}$ for every $\lambda\in \Lambda$
and 
$f^{-1}[A]\in \I_{\rho}$ for every  $A \in \cA$.

Since $f^{-1}[A]\in \I_\rho$ for every $A\in \cA$, we can use  Proposition~\ref{prop:Pminus-for-rho-equivalent-condition} to  find 
$F\in \cF$ such that 
$\rho(F)\subseteq f^{-1}[\Lambda]$
and 
for any finite set  $\cA'\subseteq \cA$
there is a finite set $K\subseteq \Omega$ with 
$$
\rho(F\setminus K) \cap f^{-1}\left[\bigcup \cA'\right] 
=
\rho(F\setminus K) \cap \left( \bigcup_{A\in \cA'} f^{-1}[A] \right) = \emptyset
.$$
We claim that $f\restriction \rho(F)$ is $\rho$-convergent to $\infty$.
Indeed, let $U$ be a neighborhood of $\infty$.
Without loss of generality, we can assume that there is a finite set $\cA' \subseteq \cA$ such that 
$U = \Phi(\cA)\setminus \left( \cA' \cup\bigcup \cA'\right).$
Then we have a finite set 
 $K\subseteq \Omega$ such that 
$\rho(F\setminus K) \cap f^{-1}\left[\bigcup \cA'\right]  = \emptyset
$,
and consequently
$f[\rho(F\setminus K)] \subseteq U$, so 
 the proof is finished.
\end{proof}


\section{Mr\'{o}wka for maximal almost disjoint families}

In \cite{MR4584767} the author extensively studied $\FinBW(\I)$ spaces. In particular, for a large class of ideals, assuming the continuum hypothesis, he characterized in terms of Kat\v{e}tov order when there is a space in $\FinBW(\I)$ that is not in $\FinBW(\J)$. In his proofs the right space is always of the form $\Phi(\cA)$ for some \emph{maximal} almost disjoint family. In our paper we want to generalize results of \cite{MR4584767} so that they will apply also for Hindman spaces, Ramsey spaces and differentially compact spaces. As we will see at the end of this section, our generalization requires going beyond \emph{maximal} almost disjoint families (as always $\Phi(\cA)\notin\FinBW(\rho)$ for maximal $\cA$ and $\rho\in\{FS,r,\Delta\}$ -- see Corollary \ref{cor:nonFINBW-for-MAD}) and working with almost disjoint families that are not necessarily maximal.

\begin{lemma}
\label{lem:MAD-Mrowka-nie-jest-Fspejsem:technical}
Let $\rho:\cF\to[\Lambda]^\omega$ 
be partition regular with $\cF\subseteq [\Omega]^\omega$.
Let $\cA$ be  an almost disjoint family on $\Lambda$.
\begin{enumerate}
\item 
If  
$\cA\subseteq\I_\rho$ 
and $$\forall F\in \cF\,\exists A\in \cA\,\forall K\in [\Omega]^{<\omega}\,(A\cap \rho(F\setminus K)\neq\emptyset),$$
then 
$\Phi(\cA) \notin \FinBW(\rho).$\label{lem:MAD-Mrowka-nie-jest-Fspejsem:technical:item}

\item If $\cA\subseteq \I_\rho$ 
and $\cA$ is  a  \emph{maximal} 
almost disjoint family,
then \label{lem:MAD-Mrowka-nie-jest-Fspejsem:technical:mad-in-ideal}
$\Phi(\cA) \notin \FinBW(\rho).$
\end{enumerate}
\end{lemma}

\begin{proof} 
(\ref{lem:MAD-Mrowka-nie-jest-Fspejsem:technical:item})
Let $f:\Lambda\to\Phi(\cA)$ be given by $f(\lambda)=\lambda$.
We claim that there is no $F\in\cF$ such that $f\restriction \rho(F)$ is  $\rho$-convergent.
Assume, for sake of contradiction, that there is $F\in \cF$ such that 
$f\restriction \rho(F)$ is $\rho$-convergent to some $L\in \Phi(\cA)$.
We have three cases: 
   (1) $L\in \Lambda$, (2) $L\in \cA$, (3) $L=\infty$.

\emph{Case (1).} The set $U=\{L\}$ is a neighborhood of $L$.
But  for any finite set $K\subseteq\Omega$ we have 
$f[\rho(F\setminus K)] = \rho(F\setminus K) \not\subseteq \{L\}=U.$
Thus, $f\restriction \rho(F)$ is not $\rho$-convergent to $L$, a contradiction.

\emph{Case (2).}
The set   $U=L \cup\{L \}$ is a neighborhood of $L$. 
Since $f\restriction \rho(F)$ is $\rho$-convergent to $L$, there is 
a finite set $K\subseteq\Omega$ such that $\rho(F\setminus K)=f[\rho(F\setminus K)]\subseteq U$.
Thus, $\rho(F\setminus K)\subseteq L$, so $L\notin \I_{\rho}$, but $\cA\subseteq \I_{\rho}$, a contradiction.

\emph{Case (3).}
Let $A\in \cA$ be such that 
$A\cap \rho(F\setminus K)\neq\emptyset$ for every finite set $K\subseteq \Omega$.
The set $U = \Phi(\cA)\setminus (A\cup\{A\})$ is a neighborhood of $\infty$. Since $f\restriction \rho(F)$ is $\rho$-convergent to $L$, there is a finite set $K\subseteq\Omega$ such that 
$ \rho(F\setminus K) = f[\rho(F\setminus K)]\subseteq U$.
Hence, $A\cap \rho(F\setminus K)=\emptyset$, a contradiction.

(\ref{lem:MAD-Mrowka-nie-jest-Fspejsem:technical:mad-in-ideal})
Let $F\in \cF$ and enumerate $\Omega=\{o_n:n\in\omega\}$.
We  pick inductively a point $b_{n}\in \rho(F\setminus \{o_j:j<n\})\setminus\{b_j:j<n\}$ for each $n\in\omega$.
Then using maximality 
of $\cA$ we can find  $A\in \cA$ such that $A\cap \{b_n:n\in\omega\}$ is infinite.
Thus, the condition form  item (\ref{lem:MAD-Mrowka-nie-jest-Fspejsem:technical:item}) 
is satisfied, so the proof is finished.
\end{proof}

\begin{theorem}
\label{thm:MAD-Mrowka-nie-jest-Fspejsem}  
Let $\rho:\cF\to[\Lambda]^\omega$ 
be partition regular with $\cF\subseteq [\Omega]^\omega$.
\begin{enumerate}
    \item 
 If 
$\rho$ is tall
(equivalently, $\I_\rho$ is a tall ideal),
then 
there exists an  
infinite (even of cardinality $\continuum$)  \emph{maximal} almost disjoint family $\cA$ on $\Lambda$ such that\label{thm:MAD-Mrowka-nie-jest-Fspejsem:MAD-in-ideal}
$\Phi(\cA) \notin \FinBW(\rho).$

\item If 
$\I_\rho$ is \emph{not} $P^-(\Lambda)$ 
(equivalently, $\Fin^2\leq_K\I_\rho$), 
then 
$\Phi(\cA) \notin \FinBW(\rho)$
for every infinite 
\emph{maximal} 
almost disjoint family $\cA$ on $\Lambda$.\label{thm:MAD-Mrowka-nie-jest-Fspejsem:ideal-over-FINsqrd}  

\end{enumerate}
\end{theorem}

\begin{proof}
(\ref{thm:MAD-Mrowka-nie-jest-Fspejsem:MAD-in-ideal})
It follows from Lemma~\ref{lem:MAD-Mrowka-nie-jest-Fspejsem:technical}(\ref{lem:MAD-Mrowka-nie-jest-Fspejsem:technical:mad-in-ideal}), because in \cite[Proposition~2.2]{MR3276758}, the authors proved that if an ideal $\I$ is tall, then there exists 
an infinite maximal almost disjoint family $\cA$  of infinite subsets of $\Lambda$ such that $\cA\subseteq \I$. If necessary, we can  make $\cA$ to be of cardinality $\continuum$ (just take one set $A\in \cA$, construct your favourite almost disjoint family $\cB$ of cardinality $\continuum$ on $A$, then any maximal almost disjoint family extending $\cA\cup \cB$ is the required family).
The equivalence of $\rho$ being tall
and $\I_\rho$ being a tall ideal follows from Proposition~\ref{prop:prop-of-talness}.

(\ref{thm:MAD-Mrowka-nie-jest-Fspejsem:ideal-over-FINsqrd}) 
The equivalence of 
$\I_\rho$ not being $P^-(\Lambda)$ 
and $\Fin^2\leq_K\I_\rho$ follows from Proposition \ref{prop:Pminus-versus-FinSQRD-for-rho}(\ref{prop:Pminus-versus-FinSQRD-for-ideal:Pminus-Lambda}).

Let $\phi:\Lambda\to\omega^2$ be a witness for $\fin^2\leq_K\I_\rho$.
In \cite{MR3034318}, the authors proved that we can assume that $\phi$ is a bijection.
For each $n\in\omega$, we define $P_n=\phi^{-1}[\{n\}\times\omega]$.
Then $\{P_n:n\in\omega\}$ is a partition of $\Lambda$ and $P_n\in \I_\rho\cap [\Lambda]^\omega$ for each $n\in\omega$.
Let $\cA=\{A_\alpha:\alpha<|\cA|\}$.
Since $\cA$ is infinite, $|\cA|\geq \omega$.
Let $f:\Lambda\to \Phi(\cA)$ be a bijection such that 
$f[P_n]=A_n \setminus  \bigcup \{ A_i:i<n\}$ for each $n\in\omega.$
We claim that $f$ does not have a $\rho$-convergent subsequence.
Assume, for the sake of contradiction, that $f\restriction\rho(F)$ is $\rho$-convergent to some $L\in \Phi(\cA)$ for some $F\in \cF$.
We have three cases: 
     (1) $L\in \Lambda$, (2) $L\in \cA$,  (3) $L=\infty$.

\emph{Case (1).} 
The set $U=\{L\}$ is a neighborhood of $L$.
But  for any finite set $K\subseteq\Omega$ we have 
$f[\rho(F\setminus K)]  \not\subseteq \{L\}=U.$
Thus, $f\restriction \rho(F)$ is not $\rho$-convergent to $L$, a contradiction.

\emph{Case (2).} 
We have two subcases: 
(2a) $\exists n\in \omega\, (L=A_n)$,   (2b) $\exists \alpha\in |\cA|\setminus \omega\, (L=A_\alpha)$. 

\emph{Case (2a).} 
The set $U=\{A_n\}\cup A_n$ is a neighborhood of $L$, so there is a finite set $K\subseteq\Omega$ such that 
$f[\rho(F\setminus K)] \subseteq U$.
Then 
$f[\rho(F\setminus K)] \subseteq A_n$,
so
$\rho(F\setminus K) \subseteq f^{-1}[A_n]\subseteq \bigcup_{i\leq n}P_i\in \I_\rho$, a contradiction.

\emph{Case (2b).} 
The set  $U=A_\alpha\cup\{A_\alpha\}$
 is a neighborhood of $L$, so there is a finite set $K\subseteq\Omega$ such that 
$f[\rho(F\setminus K)] \subseteq U$.
Then 
$f[\rho(F\setminus K)] \subseteq A_\alpha$,
so 
$\rho(F\setminus K) \subseteq f^{-1}[A_\alpha]$.
Thus $f^{-1}[A_\alpha]\notin\I_\rho$, 
and consequently $\phi[f^{-1}[A_\alpha]]\notin \fin^2$.
On the other hand, 
$A_\alpha\cap A_n$ is finite for each $n\in\omega$, so  
$f^{-1}[A_\alpha\cap A_n]$ is finite, and consequently 
$f^{-1}[A_\alpha]\cap P_n$ is finite for every $n\in\omega$.
Thus, $\phi[f^{-1}[A_\alpha]]\in \fin^2$, a contradiction.

\emph{Case (3).} Using Proposition~\ref{prop:rho-con-implies-suseq-ordinar-comvergence}(\ref{prop:rho-con-implies-suseq-ordinar-comvergence:item}) we find an infinite set $B\subseteq\Lambda$ such that $f\restriction B$ is convergent to $\infty$.
Since $f$ is a bijection, $f[B]$ is infinite.
Thus, using maximality 
of $\cA$, we find $\alpha$ such that 
$A_\alpha\cap f[B]$ is infinite.
Since $U=\Phi(\cA)\setminus(\{A_\alpha\}\cup A_\alpha)$ is  a neighborhood of $\infty$, 
there is  a finite set $K\subseteq\Lambda$ such that 
$f[B\setminus  K] \subseteq U$.
Then 
$A_\alpha\cap f[B \setminus K]  = \emptyset$, a contradiction.
\end{proof}

\begin{corollary}
\label{cor:nonFINBW-for-MAD}
 Let $\cA$ be an infinite \emph{maximal} almost disjoint family. 
\begin{enumerate}

\item Hindman spaces.

\begin{enumerate}
\item 
\cite[Theorem~10]{MR1887003}
If $\cA\subseteq\Hindman$, 
then $\Phi(\cA)$ is not a Hindman space.

\item 
\cite[Proposition~1.1]{MR4356195}
$\Phi(\cA)$ is not a Hindman space.\label{cor:nonFINBW-for-MAD:Hindman}

\end{enumerate}
  
\item Ramsey spaces.

\begin{enumerate}

\item 
\cite[Example 4.1]{MR4552506}
If $\{\{n,k\}: k\in\omega\setminus\{n\}\}\in \cA$ for every $n\in\omega$, 
then $\Phi(\cA)$ is not a Ramsey space.

\item $\Phi(\cA)$ is not a Ramsey space.
\end{enumerate}

\item  Differentially compact  spaces.

\begin{enumerate}

\item 
\cite[4.2.2]{Shi2003Numbers} or  \cite[Theorem~4.9]{MR3097000}
If $\cA\subseteq\Diff$, 
then $\Phi(\cA)$  is not a differentially compact space.

\item 
\cite[Theorem~2.1]{MR4358658}
$\Phi(\cA)$   is not a differentially compact space.
\end{enumerate}
                
\item van der Waerden spaces.

\begin{enumerate}
\item 
\cite[Theorem~6]{MR1866012}
If $\cA\subseteq\vdW$, 
then $\Phi(\cA)$   is not a van der Waerden space.\label{cor:nonFINBW-for-MAD:vdW}

\end{enumerate}
                
\item $\I_{1/n}$-spaces.

\begin{enumerate}
\item 
\cite[Proposition~2.2]{MR2471564}
If $\cA\subseteq\I_{1/n}$, 
then $\Phi(\cA)$   is not a $\I_{1/n}$-space.\label{cor:nonFINBW-for-MAD:summable}
\end{enumerate}

\item $\FinBW(\I)$.

\begin{enumerate}

\item 
\cite[Proposition~2.2]{MR2471564}
If $\I$ is a tall $F_\sigma$-ideal on $\Lambda$ and $\cA\subseteq\I$, 
then $\Phi(\cA)\notin\FinBW(\I)$.\label{cor:nonFINBW-for-MAD:Fsigma}

\item 
\cite[Proposition~2.3]{MR3276758}
If $\I$ is a tall ideal and  $\cA\subseteq \I$, 
then 
$\Phi(\cA)\notin\FinBW(\I)$.\label{cor:nonFINBW-for-MAD:tall} 

\end{enumerate}

\end{enumerate}

\end{corollary}

\begin{proof}
(1)--(6)
In cases when we assume that $\cA\subseteq\I$, it follows from Lemma~\ref{lem:MAD-Mrowka-nie-jest-Fspejsem:technical}(\ref{lem:MAD-Mrowka-nie-jest-Fspejsem:technical:mad-in-ideal}) along with  Proposition~\ref{prop:basic-relationships-between-FinBW-like-spaces}(\ref{prop:basic-relationships-between-FinBW-like-spaces:ideal-rho}) in some cases. 
In other cases, it follows from Theorem~\ref{thm:MAD-Mrowka-nie-jest-Fspejsem}(\ref{thm:MAD-Mrowka-nie-jest-Fspejsem:ideal-over-FINsqrd}) and 
Proposition~\ref{prop:FinSQRD-below-FS-r-Delta}.
\end{proof}


\section{Distinguishing between \texorpdfstring{$\FinBW$}{FinBW } classes via Kat\v{e}tov order on ideals}

In this section we prove first of the two main results of this part and show its various particular cases and consequences. We will need the following two lemmas.

\begin{lemma}
\label{lem:TWIERDZENIE-for-Fspaces-in-FinBW-technical}
Let $\rho:\cF\to[\Lambda]^\omega$ be  partition regular with $\cF\subseteq [\Omega]^\omega$.
Let $\cA$ be an infinite almost disjoint family on $\Lambda$ such that  
for every $\I_\rho$-to-one 
function $f:\Lambda\to\Lambda$
there is $E\in \cF$ such that  
the family 
$$\{A\in \cA: \forall K\in [\Omega]^{<\omega}\, (|A\cap f[\rho(E\setminus K)]|=\omega)\}$$
is at most countable.
If $\rho$ is  $P^-$,
then 
$\Phi(\cA)\in \FinBW(\rho).$
\end{lemma}

\begin{proof}
Let 
$f:\Lambda\to \Phi(\cA)$.
By Lemma~\ref{lem:TWIERDZENIE-for-Fspaces-in-FinBW}, 
we can assume that 
 $f^{-1}(\infty)\in\I_{\rho}$, $f^{-1}[\cA] \in\I_{\rho}$,
    $f^{-1}[\Lambda]\in \I_{\rho}^*$,  
    $f^{-1}(\lambda)\in \I_{\rho}$ for every $\lambda\in \Lambda$
    and 
    $f^{-1}[A]\in \I_{\rho}$ for every  $A\in \cA$.
Then $B = f^{-1}[\Lambda]\in \I^*_{\rho}$ and $f\restriction B :B\to\Lambda$ is $\I_\rho$-to-one.

We fix an element $\lambda_0\in \Lambda$ and define
 $g:\Lambda\to\Lambda$ by $g(\lambda)=f(\lambda)$ for $\lambda\in B$ and $g(\lambda)=\lambda_0$ otherwise.
Then $g$ is $\I_\rho$-to-one, so there is $E\in \cF$ such that 
the family 
$$\cC = \{A\in \cA: \forall K\in [\Omega]^{<\omega}\, (|A\cap g[\rho(E\setminus K)]|=\omega)\}$$
is at most countable.

 Since $\rho(E)\cap B\notin \I_\rho$ and 
$\rho$ is $P^-$, we can use Proposition~\ref{prop:Pminus-for-rho-equivalent-condition} to  find 
$F\in \cF$ such that 
$\rho(F)\subseteq \rho(E)\cap B$
and 
for any finite sets  $S\subseteq \Lambda$
and  $\cT\subseteq \cC$ 
there is a finite set $L\subseteq \Omega$ with 
$$
\rho(F\setminus L) \cap \left( 
f^{-1}[S] 
\cup \bigcup\left\{f^{-1}\left[A\right]:A \in \cT\right\}\right) = \emptyset
.$$
We claim that $f\restriction \rho(F)$ is $\rho$-convergent to $\infty$.
Indeed, let $U$ be a neighborhood of $\infty$.
Without loss of generality, we can assume that there is a finite set $\Gamma \subseteq\cA$ such that 
$U = \Phi(\cA)\setminus \left(\{A: A\in \Gamma\} \cup\bigcup\{A:A\in \Gamma\}\right).$

For each $A\in \Gamma\setminus \cC$, there is a finite set $K_A\subseteq\Omega$ such that $A\cap f[\rho(F\setminus K_A)] = A\cap g[\rho(F\setminus K_A)]$ is finite.
Then $K=\bigcup\{K_A:A\in \Gamma\setminus \cC\}\subseteq\Omega$ is a finite set such that $A\cap f[\rho(F\setminus K)]$ is finite for every $A\in \Gamma\setminus \cC$.

Then both 
$S = f[\rho(F\setminus K)]\cap \bigcup \{A: A \in \Gamma\setminus \cC\}$
 and 
$\cT = \Gamma\cap\cC$
are finite, 
so we can find a finite set $L\subseteq \Omega$ such that 
$$\rho(F\setminus L) \cap \left( 
f^{-1}[S] 
\cup \bigcup\left\{f^{-1}\left[A\right]:A\in \cT\right\}\right) = \emptyset,
$$
and consequently we obtain a finite set $K\cup L$ such that 
\begin{equation*}
    \begin{split}
        f[\rho(F\setminus (K\cup L))]
&\subseteq
\Lambda \setminus \bigcup\{A:A\in \Gamma \}
\subseteq U.
\end{split}
\end{equation*}
That finishes the proof.
\end{proof}

Recall that $\pnumber$ is the smallest cardinality of a family $\cF$ of infinite subsets of $\omega$ with the  strong finite intersection property (i.e.~intersection of finitely many sets from $\cF$ is infinite) that does not have a pseudointersection (i.e.~there is no infinite set $A\subseteq\omega$ such that $A\setminus F$ is finite for each $F\in \cF$; see e.g.~\cite{MR776622}).

\begin{lemma}[Assume $\pnumber=\continuum$]
\label{lem:LEMAT}
Let $\rho_i:\cF_i\to[\Lambda_i]^\omega$ be  partition regular with $\cF_i\subseteq [\Omega_i]^\omega$ for each $i=1,2$.
Let $\{f_\alpha: \alpha<\continuum\}$ be an enumeration of all  functions $f:\Lambda_1\to \Lambda_2$ 
and
 $\cF_2=\{F_\alpha:\alpha<\continuum\}$.
 
If $\I_{\rho_2}\not\leq_K \I_{\rho_1}$, then  there exist 
families 
$\cA = \{A_\alpha:\alpha<\continuum\}$
 and  $\cC=\{C_\alpha : \alpha<\continuum\}$
 such that for every $\alpha<\continuum$:
 \begin{enumerate}

\item $C_\alpha\in \cF_1$,

\item $f_\alpha[\rho_1(C_\alpha)]\in \I_{\rho_2}$,

\item $A_\alpha \in \I_{\rho_2}\cap [\Lambda_2]^\omega$,
\item $\forall\beta<\alpha\, (|A_\alpha\cap A_\beta|<\omega)$,\label{lem-item4}

\item $\forall\gamma >\alpha\, (|A_\gamma \cap f_\alpha [\rho_1(C_\alpha)]|<\omega)$,\label{lem-item5}

\item $\forall L\in [\Omega_2]^{<\omega}\, (A_\alpha \subseteq^* \rho_2(F_\alpha\setminus L))$.\label{lem-item6}

\end{enumerate}
\end{lemma}

\begin{proof}
Suppose that $A_\beta$ and $C_\beta$ have been constructed for $\beta<\alpha$ and satisfy items (1)-(6).

First, we construct the set $C_\alpha$. 
Since   
$\I_{\rho_2}\not\leq_K \I_{\rho_1}$,
 there is a set 
$C_\alpha\in \cF_1$
 such that 
$f_\alpha[\rho_1(C_\alpha)]\in \I_{\rho_2}$.

Now, we turn to the construction of the set $A_\alpha$.
Let 
$$\cD
= 
\{\Lambda_2\setminus f_\beta[\rho_1(C_\beta)]:\beta<\alpha\} 
\cup 
\{\Lambda_2\setminus A_\beta:\beta<\alpha\}
\cup
\{
\rho_2(F_\alpha\setminus L): L\in [\Omega_2]^{<\omega}
\}.
$$
Since 
$\bigcap\{\rho_2(F_\alpha\setminus L_i):i<n\} \in\I_{\rho_2}^+$ for every $n\in\omega$ and finite sets $L_i\subseteq\Omega$, 
and
$\Lambda_2\setminus f_\beta[\rho_1(C_\beta)]\in \I_{\rho_2}^*$ and  $\Lambda_2\setminus A_\beta\in \I_{\rho_2}^*$ for every $\beta<\alpha$, we obtain that the intersection of finitely many sets from $\cD$ is in $\I_{\rho_2}^+$. In particular, this intersection is infinite, so $\cD$ has the strong finite intersection property.
Since  $|\cD|<\continuum=\pnumber$, there exists an infinite set $A\subseteq \Lambda_2$ such that $A\subseteq^* D$ for every $D\in \cD$.
Since 
$\I_{\rho_2}\not\leq_K\I_{\rho_1}$, we obtain that the ideal $\I_{\rho_2}$ is tall, and consequently there is an infinite set $A_\alpha\subseteq A$ such that $A_\alpha\in \I_{\rho_2}$. 

It is not difficult to see, that the sets $A_\alpha$ and $C_\alpha$ satisfy all the required conditions, so the proof of the lemma is finished.
\end{proof}

We are ready for the main result of this section.

\begin{theorem}[Assume CH]
\label{thm:TWIERDZENIE-for-Fspaces:Pminus-Mrowka}
Let $\rho_i:\cF_i\to[\Lambda_i]^\omega$ be  partition regular for each $i=1,2$. 
If  
$\rho_1$ is  $P^-$  and  
$\I_{\rho_2}\not\leq_K \I_{\rho_1}$,
then there exists an almost disjoint family $\cA$ such that 
$|\cA|=\continuum$
and 
$\Phi(\cA)\in \FinBW(\rho_1) \setminus \FinBW(\rho_2).$ In particular, there is a Hausdorff compact and separable space of cardinality $\continuum$ in $\FinBW(\rho_1) \setminus \FinBW(\rho_2).$
\end{theorem} 

\begin{proof}
Using Proposition~\ref{prop:partition-regular-functions-live-on-the-same-set} we can assume that $\Lambda_1=\Lambda_2=\Lambda$.

Let $\{f_\alpha: \alpha<\continuum\}$ be an enumeration of all  functions $f:\Lambda\to \Lambda$
and
 $\cF_2=\{F_\alpha:\alpha<\continuum\}$.
By Lemma~\ref{lem:LEMAT},  
 there exist  families $\cA = \{A_\alpha:\alpha<\continuum\}$ and  $\cC=\{C_\alpha : \alpha<\continuum\}$
 such that for every $\alpha<\continuum$ all conditions of Lemma~\ref{lem:LEMAT} are satisfied. 
We claim that $\cA$ is the required family.

First, we see that $\cA$ is an almost disjoint family on $\Lambda$ by item (\ref{lem-item4}) of Lemma~\ref{lem:LEMAT}, $|\cA|=\continuum$ and $\cA\subseteq\I_{\rho_2}$. 
Second,  CH together with item (\ref{lem-item5}) of Lemma~\ref{lem:LEMAT} 
ensures that 
$$|\{\beta< \continuum: \forall K\in [\Omega]^{<\omega}\, (|A_\beta\cap f_\alpha[\rho_1(C_\alpha\setminus K)]|=\omega)\}|\leq|\alpha+1|\leq\omega$$
for each $\alpha<\continuum$, so knowing that $\rho_1$ is  $P^-$  we can use Lemma~\ref{lem:TWIERDZENIE-for-Fspaces-in-FinBW-technical} 
to see that 
$\Phi(\cA)\in \FinBW(\rho_1)$. 
Third, we use item (\ref{lem-item6}) of Lemma~\ref{lem:LEMAT}  and Lemma~\ref{lem:MAD-Mrowka-nie-jest-Fspejsem:technical}(\ref{lem:MAD-Mrowka-nie-jest-Fspejsem:technical:item}) to see that $\Phi(\cA)\notin\FinBW(\rho_2)$.
\end{proof}

Now we want to show various applications of Theorem \ref{thm:TWIERDZENIE-for-Fspaces:Pminus-Mrowka}. Those applications can be divided into three parts. The first part concerns existence of a Hausdorff compact and separable space in $\FinBW(\rho)$. Before applying Theorem \ref{thm:TWIERDZENIE-for-Fspaces:Pminus-Mrowka}, we need to prove one more result.

\begin{proposition}
\label{prop:ideal-not-below-in-KAT-another-ideal}
For every ideal $\I$ there is an ideal $\J$ such that $\J\not\leq_K\I$.     
\end{proposition}

\begin{proof}
Suppose for the sake of contradiction that there is an ideal $\I$ on $\Lambda$ such that $\J\leq_K\I$ for every ideal $\J$.
Then for every maximal (with respect to inclusion) ideal $\J$ on $\omega$ there exists a function $f_\J:\Lambda\to\omega$ such that $f_\J^{-1}[A]\in \I$ for every $A\in \J$.
Let   $\K(f_\J) = \{A\subseteq \omega: f_{\J}^{-1}[A]\in\I\}$.
Then $\K(f_\J)$ is an ideal and $\J\subseteq\K(f_\J)$.
Since $\J$ is maximal,  $\K(f_\J)=\J$.
There are $2^\continuum$ pairwise distinct ultrafilters on $\omega$ (see e.g.~\cite[Theorem~7.6]{MR1940513}), so there are $2^\continuum$ pairwise distinct maximal ideals on $\omega$ (given an ultrafilter $\cU$ on $\omega$, the family $\{A\subseteq\omega:A\notin\cU\}$ is a maximal ideal). However, there are only $\continuum$ many function from $\omega$ into $\omega$, a contradiction. 
\end{proof}

\begin{theorem}[Assume CH]
\label{thm:TWIERDZENIE-for-Fspaces-in-FinBW}
Let $\rho:\cF\to[\Lambda]^\omega$ be  partition regular with $\cF\subseteq [\Omega]^\omega$.
If $\rho$ is  $P^-$,
then there exists an almost disjoint family $\cA$ such that $|\cA|=\continuum$ and \label{thm:TWIERDZENIE-for-Fspaces-in-FinBW:item-main}
$ \Phi(\cA)\in \FinBW(\rho).$ In particular, there is a Hausdorff compact and separable space of cardinality $\continuum$  in $\FinBW(\rho)$.
\end{theorem}

\begin{proof}
By Proposition~\ref{prop:ideal-not-below-in-KAT-another-ideal} there is an ideal $\J$ such that $\J\not\leq_K\I_\rho$, so Theorem~\ref{thm:TWIERDZENIE-for-Fspaces:Pminus-Mrowka} gives us an almost disjoint family $\cA$ such that  $|\cA|=\continuum$ and $ \Phi(\cA)\in \FinBW(\rho)\setminus\FinBW(\rho_\J)$.
\end{proof}

\begin{corollary}[Assume CH]
There exists (for each item distinct) an almost disjoint family $\cA$ for which $\Phi(\cA)$ is a Hausdorff compact and separable space of cardinality $\continuum$ such that: 
\begin{enumerate}
    \item $\Phi(\cA)$ is a Hindman space,\label{cor-H}
\item \cite[Theorem~4.7]{MR4552506} $\Phi(\cA)$ is a Ramsey space,\label{cor-R}
    \item $\Phi(\cA)$ is a differentially compact space,\label{cor-D}
    \item \cite[Theorem~5.3]{MR4584767} $\Phi(\cA)\in\FinBW(\I)$, where $\I$ is  a $P^-$ ideal (in particular, if $\I$ is a $G_{\delta\sigma\delta}$ ideal).\label{cor-I}
\end{enumerate}
\end{corollary}

\begin{proof}
Items (\ref{cor-H}), (\ref{cor-R}) and (\ref{cor-D}) follow from Theorem~\ref{thm:TWIERDZENIE-for-Fspaces-in-FinBW}
and 
Proposition~\ref{prop:Plike-properties-for-known-rho}(\ref{prop:Plike-properties-for-known-rho:FS-r-Delta:Pminus}). Item (\ref{cor-I}) follows from Theorem~\ref{thm:TWIERDZENIE-for-Fspaces-in-FinBW}
and 
Propositions~\ref{prop:properties-of-ideal-versus-rho}(\ref{prop:properties-of-ideal-versus-rho:ideal-rho})
and 
\ref{prop:basic-relationships-between-FinBW-like-spaces}(\ref{prop:basic-relationships-between-FinBW-like-spaces:ideal-rho}), and the ``in particular'' part follows from Proposition~\ref{thm:Plike-properties-for-definable-ideals}(\ref{thm:Plike-properties-for-definable-ideals:Fsigmadelta}).
\end{proof}

The second part of applications of Theorem \ref{thm:TWIERDZENIE-for-Fspaces:Pminus-Mrowka} concerns a special case when $\I_{\rho_1}$ is $P^-(\Lambda_1)$ while $\I_{\rho_2}$ is \emph{not} $P^-(\Lambda_2)$.

\begin{theorem}[Assume CH]
\label{thm:dist-between-FinBW}
Let $\rho_i:\cF_i\to[\Lambda_i]^\omega$ be  partition regular functions for each $i=1,2$.
If 
 $\rho_1$ is  $P^-$,
$\I_{\rho_1}$ is $P^-(\Lambda_1)$ (equivalently, $\Fin^2\not\leq_K\I_{\rho_1}$)
and 
$\I_{\rho_2}$ is \emph{not} $P^-(\Lambda_2)$ (equivalently,  $\Fin^2\leq_K \I_{\rho_2}$),
then 
there exists an infinite almost disjoint family $\cA$ of cardinality $\continuum$ such that 
$\Phi(\cA)\in \FinBW(\rho_1) \setminus \FinBW(\rho_2).$ In particular, there is a Hausdorff compact and separable space of cardinality $\continuum$ in $\FinBW(\rho_1) \setminus \FinBW(\rho_2).$
\end{theorem}

\begin{proof}
The equivalence of $\I_{\rho_1}$ being $P^-(\Lambda_1)$ and $\Fin^2\not\leq_K\I_{\rho_1}$ follows from Proposition~\ref{prop:Pminus-versus-FinSQRD-for-rho}(\ref{prop:Pminus-versus-FinSQRD-for-ideal:Pminus-Lambda}).

Since $\Fin^2\leq_K \I_{\rho_2}$ and $\Fin^2\not\leq_K\I_{\rho_1}$, we know that $\I_{\rho_2}\not\leq_K \I_{\rho_1}$, so Theorem~\ref{thm:TWIERDZENIE-for-Fspaces:Pminus-Mrowka} finishes the proof.
\end{proof}

\begin{corollary}[Assume CH] 
\label{cor:Ispace-not-rho-space-for-known-rho}
There exists (for each item distinct) an almost disjoint family $\cA$ for which $\Phi(\cA)$ is a Hausdorff compact and separable space of cardinality $\continuum$ and the following holds. 

\begin{enumerate}

\item $\Phi(\cA)$ is in $\FinBW(\I)$, where $\I$ is a $P^-$ ideal (in particular, if $\I$ is a $G_{\delta\sigma\delta}$ ideal), but $\Phi(\cA)$ is not a:\label{item:Ispace-not-rho-space-for-known-rho}

\begin{enumerate}
\item \cite[Corollary~11.5]{MR4584767} Hindman space,
\item Ramsey space,
\item differentially compact space.
\end{enumerate}

\item 

\begin{enumerate}

\item \cite[Theorem~3]{MR1950294}
 $\Phi(\cA)$ is a van der Waerden space that  is not a Hindman space.
             
\item \cite[Theorem~4.4]{MR2471564}
 $\Phi(\cA)$ is an $\I_{1/n}$-space that is not a Hindman space.

\end{enumerate}

\item 

\begin{enumerate}

\item \cite[Theorem~4.2.2]{Shi2003Numbers}
 $\Phi(\cA)$ is a van der Waerden space that  is not a  differentially compace  space.

\item \cite[Theorem~3.5]{MR4358658}
$\Phi(\cA)$ is in $\FinBW(\I)$ but it is not a differentially compact space for any $P^+$ ideal $\I$  (in particular, for any $F_\sigma$ ideal). For instance, 
  $\Phi(\cA)$ is an $\I_{1/n}$-space that is not a differentially compact space.
\end{enumerate}
\end{enumerate}
\end{corollary} 

\begin{proof}
Item (\ref{item:Ispace-not-rho-space-for-known-rho}) follows from Theorem~\ref{thm:dist-between-FinBW} 
and 
Propositions~\ref{prop:FinSQRD-not-below-Fsigma}(\ref{prop:FinSQRD-not-below-Fsigma:item}), \ref{prop:basic-relationships-between-FinBW-like-spaces}(\ref{prop:basic-relationships-between-FinBW-like-spaces:ideal-rho}) and \ref{prop:properties-of-ideal-versus-rho}(\ref{prop:properties-of-ideal-versus-rho:ideal-rho}), \ref{prop:Plike-basic-properties-for-ideals}(\ref{prop:Plike-basic-properties-for-ideals:Pplus-implies-weakPplus}). Other items follow from item (\ref{item:Ispace-not-rho-space-for-known-rho}), Theorem~\ref{thm:Plike-properties-for-definable-ideals}(\ref{thm:Plike-properties-for-definable-ideals:Fsigma}) and
Propositions~\ref{prop:Plike-properties-for-known-rho}(\ref{prop:Plike-properties-for-known-rho:Fsigma-known}), \ref{prop:Plike-basic-properties-for-ideals}(\ref{prop:Plike-basic-properties-for-ideals:Pplus-implies-weakPplus}).
\end{proof}

Now we deal with the third part of applications of Theorems~\ref{thm:TWIERDZENIE-for-Fspaces:Pminus-Mrowka}, in which we need to use its full strength. 

\begin{corollary}[Assume CH]
\label{cor:distinguishing-Hindman-Ramsey-Diff-spaces}
There exists (for each item distinct) an  almost disjoint family $\cA$ for which $\Phi(\cA)$ is a Hausdorff compact and separable space of cardinality $\continuum$ such that 
\begin{enumerate}
\item 
$\Phi(\cA)$ is a Ramsey space that is not a Hindman space;\label{cor:distinguishing-Hindman-Ramsey-Diff-spaces:Ramsey-not-Hindman-space}

\item 
$\Phi(\cA)$ is a Hindman space that  is not a Ramsey space;\label{cor:distinguishing-Hindman-Ramsey-Diff-spaces:diff-not-Hindman-space:Hindman-not-Ramsey}

\item 
$\Phi(\cA)$ is a differentially compact space that is not a Hindman space;\label{cor:distinguishing-Hindman-Ramsey-Diff-spaces:diff-not-Hindman-space}

\item
$\Phi(\cA)$ is a differentially compact space that  is not a Ramsey space.\label{cor:distinguishing-Hindman-Ramsey-Diff-spaces:diff-not-Hindman-space:diff-not-Ramsey}
\end{enumerate}
\end{corollary}

\begin{proof}
It follows from Theorems~\ref{thm:TWIERDZENIE-for-Fspaces:Pminus-Mrowka} and \ref{thm:Katetov-between-known-rho}  and Proposition~\ref{prop:Plike-properties-for-known-rho}(\ref{prop:Plike-properties-for-known-rho:FS-r-Delta:Pminus}).
\end{proof}

\begin{remark}
The space from Corollary~\ref{cor:distinguishing-Hindman-Ramsey-Diff-spaces}(\ref{cor:distinguishing-Hindman-Ramsey-Diff-spaces:diff-not-Hindman-space}) yields the  negative answer to 
  \cite[Question~4.2.2]{Shi2003Numbers} (see also \cite[Problem 1]{MR3097000} and \cite[Question 3]{MR4358658}). 
\end{remark}

\begin{corollary}[Assume CH]\ 
\label{cor:Fspaces:Pminus-Mrowka}
\begin{enumerate}
\item If $\I$ is an ideal such that $\I\not\leq_K\Hindman$ ($\I\not\leq_K\Ramsey$, $\I\not\leq_K\Diff$, resp.), then there exists an  almost disjoint family $\cA$  such that the Hausdorff compact and separable space $\Phi(\cA)$ of cardinality $\continuum$
 is a Hindman (Ramsey,  differentially compact, resp.) space that is not 
in $\FinBW(\I)$.\label{cor:Fspaces:Pminus-Mrowka:Hindman-not-summable}
\item There exists an almost disjoint family $\cA$ such that the Hausdorff compact and separable space $\Phi(\cA)$ of cardinality $\continuum$ 
is a Hindman (Ramsey,  differentially compact, resp.) space that 
 is not an $\I_{1/n}$-space.\label{cor:Fspaces:Pminus-Mrowka:Hindman-not-summable:existence}
\end{enumerate}
\end{corollary}

\begin{proof}
(\ref{cor:Fspaces:Pminus-Mrowka:Hindman-not-summable}) 
It follows from Theorem \ref{thm:TWIERDZENIE-for-Fspaces:Pminus-Mrowka} and Propositions~\ref{prop:Plike-properties-for-known-rho}(\ref{prop:Plike-properties-for-known-rho:FS-r-Delta:Pminus}) and 
\ref{prop:basic-relationships-between-FinBW-like-spaces}(\ref{prop:basic-relationships-between-FinBW-like-spaces:ideal-rho}). 

(\ref{cor:Fspaces:Pminus-Mrowka:Hindman-not-summable:existence}) It follows from item~(\ref{cor:Fspaces:Pminus-Mrowka:Hindman-not-summable}) and Theorem~\ref{thm:Katetov-between-known-rho}.
\end{proof}

\begin{remark}
In \cite[Theorem~2.5]{MR4356195}, the authors constructed, assuming  CH and $\I\not\leq_K\Hindman$, a \emph{non}-Hausdorff Hindman space that is not in $\FinBW(\I)$. Corollary~\ref{cor:Fspaces:Pminus-Mrowka}(\ref{cor:Fspaces:Pminus-Mrowka:Hindman-not-summable}) strengthens this result to the case of Hausdorff spaces. Taking $\I=\I_{1/n}$, they obtained a  positive answer to the  question posed in  \cite{Flaskova-slides-2007}, namely they constructed a (non-Hausdorff) Hindman space which is not $\I_{1/n}$-space. In Corollary~\ref{cor:Fspaces:Pminus-Mrowka}(\ref{cor:Fspaces:Pminus-Mrowka:Hindman-not-summable:existence}), we obtained a \emph{Hausdorff} answer to the above mentioned question.
\end{remark}

\begin{corollary}[Assume CH]\ 
\label{cor:TWIERDZENIE-for-Fspaces:two-ideals}
\begin{enumerate}
    \item 
\cite[Theorem~9.3]{MR4584767}
If  $\I_1$ and $\I_2$ are ideals such that \label{cor:TWIERDZENIE-for-Fspaces:two-ideals:item}
$\I_1$ is $P^-$ (in particular, if $\I_1$ is a $G_{\delta\sigma\delta}$ ideal) and 
$\I_{2}\not\leq_K \I_{1}$, 
then  
there exists an almost disjoint family $\cA$ such that the Hausdorff compact and separable space $\Phi(\cA)$ of cardinality $\continuum$ belongs to $\FinBW(\I_1) \setminus \FinBW(\I_2).$

\item 
{\cite[Theorem~3.3]{MR2471564}}
There exists an almost disjoint family $\cA$ such that the Hausdorff compact and separable space $\Phi(\cA)$ of cardinality $\continuum$ is a van der Waerden space that 
 is not an $\I_{1/n}$-space.\label{cor:TWIERDZENIE-for-Fspaces:two-ideals:item-vdW}   
\end{enumerate}

\end{corollary} 

\begin{proof}
(\ref{cor:TWIERDZENIE-for-Fspaces:two-ideals:item})
It follows from Theorem~\ref{thm:TWIERDZENIE-for-Fspaces:Pminus-Mrowka} and Propositions~\ref{prop:properties-of-ideal-versus-rho}(\ref{prop:properties-of-ideal-versus-rho:ideal-rho}), 
\ref{prop:basic-relationships-between-FinBW-like-spaces}(\ref{prop:basic-relationships-between-FinBW-like-spaces:ideal-rho}).

(\ref{cor:TWIERDZENIE-for-Fspaces:two-ideals:item-vdW})
It follows from item (\ref{cor:TWIERDZENIE-for-Fspaces:two-ideals:item}),  Theorem~\ref{thm:Katetov-between-known-rho}(\ref{thm:Katetov-between-known-rho:Summable-not-below-vdW})
and
Proposition~\ref{prop:Plike-properties-for-known-rho}(\ref{prop:Plike-properties-for-known-rho:vdW}).
\end{proof}


\section{Distinguishing between \texorpdfstring{$\FinBW$}{FinBW } classes via Kat\v{e}tov order on partition regular functions}

In this section we prove the second of the main results of Part \ref{part:Distinguishing-FinBW-spaces}. Then we compare it with Theorem \ref{thm:TWIERDZENIE-for-Fspaces:Pminus-Mrowka} and show that none of them can be derived from the other one. We start with a technical lemma.

\begin{lemma}[Assume CH]
\label{lem:LEMAT:rho}
Let $\rho_i:\cF_i\to[\Lambda_i]^\omega$ be  partition regular with $\cF_i\subseteq [\Omega_i]^\omega$ for each $i=1,2$.
Let $\{f_\alpha: \alpha<\continuum\}$ be an enumeration of all functions $f:\Lambda_1\to \Lambda_2$ 
and
 $\{F_\alpha:\alpha<\continuum\}$
 be an enumeration of all sets $F\in \cF_2$ having small accretions.

 If 
 $\rho_2$ is  $P^-$ and 
 $\rho_2\not\leq_K \rho_1$, 
then  there exist 
families 
$\cA = \{A_\alpha:\alpha<\continuum\}$
and
$\cC=\{C_\alpha : \alpha<\continuum\}$
 such that for every $\alpha<\continuum$:
 \begin{enumerate}

\item $A_\alpha=\emptyset\lor A_\alpha \in \I_{\rho_2}\cap [\Lambda_2]^\omega$,
\item $\forall\beta<\alpha\, (|A_\alpha\cap A_\beta|<\omega)$,\label{i3}

\item $C_\alpha\in \cF_1$,

\item $\forall F\in \cF_2\,\exists K\in [\Omega_1]^{<\omega}\, \forall  L\in [\Omega_2]^{<\omega}\, (\rho_2(F\setminus L) \not\subseteq f_\alpha[\rho_1(C_\alpha\setminus K)])$,\label{i4}   

\item $\forall\gamma >  \alpha \, \exists K\in[\Omega_1]^{<\omega} \,(|A_\gamma  \cap f_\alpha[\rho_1(C_\alpha \setminus K)]|<\omega)$,\label{i5}

\item $\exists \beta \leq \alpha  \, \forall L\in  [\Omega_2]^{<\omega}\,  (A_\beta \cap \rho_2(F_\alpha\setminus L)\neq\emptyset)$.\label{i6} 

\end{enumerate}
\end{lemma}

\begin{proof}
Suppose that $A_\beta$ and  $C_\beta$ have been constructed for $\beta<\alpha$ and satisfy all the required conditions.

First, we construct a set $C_\alpha$.
Since 
$\rho_2 \not\leq_K \rho_1$,
 there is a set 
$C_\alpha\in \cF_1$
 such that 
$$\forall F\in \cF_2\, \exists  K\in [\Omega_1]^{<\omega}\, \forall  L\in [\Omega_2]^{<\omega}\, (\rho_2(F\setminus L) \not\subseteq f_{\alpha}[\rho_1(C_\alpha\setminus K)]).$$

Now, we turn to the construction of a set $A_\alpha$.
We have two cases: 
\begin{enumerate}
 \item $\exists \beta < \alpha  \, \forall L\in  [\Omega_2]^{<\omega}\,  (A_\beta \cap \rho_2(F_\alpha\setminus L)\neq\emptyset)$.
 \item $\forall \beta < \alpha  \, \exists  L_\beta \in  [\Omega_2]^{<\omega}\,  (A_\beta \cap \rho_2(F_\alpha\setminus L_\beta) = \emptyset)$.
\end{enumerate}

\emph{Case (1).} 
We put $A_\alpha=\emptyset$. Then the sets $A_\alpha$ and $C_\alpha$ satisfy all the required conditions, so the proof of the lemma is finished in this case.

\emph{Case (2).}
Let $\alpha=\{\beta_n:n\in\omega\}$.
Let $\{L_n:n\in\omega\}$ be an increasing sequence of finite subsets of $\Omega_2$ such that 
$\bigcup\{L_{\beta_i}:i<n\} \subseteq L_n$ and
$\bigcup\{L_n:n\in\omega\} = \Omega_2$. Notice that $\rho_2(F_\alpha\setminus L_n)\cap \bigcup\{A_{\beta_i}:i<n\} = \emptyset$ for every $n\in\omega$.

We define inductively sequences $\{E_n:n\in\omega\} \subseteq \cF_2$, $\{K_n:n\in\omega\}\subseteq[\Omega_1]^{<\omega}$ and $\{a_n:n\in\omega\}\subseteq \Lambda_2$ such that for every $n\in\omega$ the following conditions hold:
\begin{enumerate}[(i)]
    \item $\rho_2(E_{n+1})\subseteq \rho_2(E_n)\subseteq \rho_2(F_\alpha)$,
    \item $\rho_2(E_n) \subseteq \rho_2(F_\alpha\setminus L_n)\setminus f_{\beta_n}[\rho_1(C_{\beta_n}\setminus K_n)]$,
    \item $a_n\in \rho_2(E_n) \setminus\{a_i:i<n\}$.
\end{enumerate}

Suppose that $E_i$, $K_i$ and $a_i$ have been constructed for $i<n$ and satisfy all the required conditions.

Since $F_\alpha$ has  small accretions, we obtain 
$\rho_2(F_\alpha\setminus L_{n-1})\setminus \rho_2(F_\alpha\setminus L_n)\in \I_{\rho_2}$,
and consequently
$\rho_2(E_{n-1})\cap \rho_2(F_\alpha\setminus L_{n})\notin\I_{\rho_2}$ (in the case of $n=0$ we put $L_{-1}=\emptyset$ and $E_{-1}=F_\alpha$).
Let $E\in \cF_2$ be such that 
$\rho_2(E)\subseteq \rho_2(E_{n-1})\cap \rho_2(F_\alpha\setminus L_{n})$.
We have 2 subcases:
\begin{enumerate}[(2a)]
    \item $\exists K_n\in [\Omega_1]^{<\omega}\,(\rho_2(E)\setminus f_{\beta_n}[\rho_1(C_{\beta_n}\setminus K_n)]\notin \I_{\rho_2}),$
    \item $\forall  K\in [\Omega_1]^{<\omega}\,(\rho_2(E)\setminus f_{\beta_n}[\rho_1(C_{\beta_n}\setminus K)]\in \I_{\rho_2}).$
\end{enumerate}

\emph{Case (2a).} 
We take $E_n\in \cF_2$ such that 
$\rho_2(E_n)\subseteq \rho_2(E)\setminus f_{\beta_n}[\rho_1(C_{\beta_n}\setminus K_n)]$ and pick any $a_n\in \rho_2(E_n) \setminus\{a_i:i<n\}$. 
Then  $E_n$, $K_n$ and $a_n$ satisfy all the required conditions.

\emph{Case (2b).} 
It will turn out, that  this subcase is impossible. 
Let $\{M_i:i\in\omega\}$ be an increasing sequence of finite subsets of $\Omega_1$ such that $\bigcup\{M_i:i\in\omega\} = \Omega_1$.   

We have 2 further subcases:
\begin{enumerate}[(2b-1)]
    \item $\bigcup\{\rho_2(E)\setminus f_{\beta_n}[\rho_1(C_{\beta_n}\setminus M_i)]: i\in\omega\}\notin\I_{\rho_2}$,
    \item $\bigcup\{\rho_2(E)\setminus f_{\beta_n}[\rho_1(C_{\beta_n}\setminus M_i)]: i\in\omega\}\in\I_{\rho_2}$.
\end{enumerate}

\emph{Case (2b-1).}
Since  $\rho_2$ is $P^-$, there is $G\in \cF_2$ such that 
$\rho_2(G)\subseteq \bigcup\{\rho_2(E)\setminus f_{\beta_n}[\rho_1(C_{\beta_n}\setminus M_i)]: i\in\omega\}$ and for every $i\in\omega$ there is a finite set $L\subseteq \Omega_2$ such that 
$\rho_2(G\setminus L)\subseteq f_{\beta_n}[\rho_1(C_{\beta_n}\setminus M_i)]$. 
On the other hand, from the inductive assumptions (more precisely: since $C_{\beta_n}$ satisfies item \ref{i4}), we know that there is a finite set $K$ such that   $\rho_2(G\setminus L)\not\subseteq f_{\beta_n}[\rho_1(C_{\beta_n}\setminus K)]$ for any  finite set $L$.
Let $i\in\omega$ be such that $K\subseteq M_i$.
Then there is a finite set $L\subseteq \Omega_2$ such that 
$\rho_2(G\setminus L)\subseteq f_{\beta_n}[\rho_1(C_{\beta_n}\setminus M_i)]\subseteq f_{\beta_n}[\rho_1(C_{\beta_n}\setminus K)]$, a contradiction.

\emph{Case (2b-2).}
In this case, 
there is $G\in \cF_2$ such that 
$\rho_2(G)\subseteq \rho_2(E)\setminus \bigcup\{\rho_2(E)\setminus f_{\beta_n}[\rho_1(C_{\beta_n}\setminus M_i)]: i\in\omega\} = \rho_2(E)\cap \bigcap\{f_{\beta_n}[\rho_1(C_{\beta_n}\setminus M_i)]: i\in\omega\}$.
From the inductive assumptions, we know that there is a finite set $K$ such that  $\rho_2(G\setminus L)\not\subseteq f_{\beta_n}[\rho_1(C_{\beta_n}\setminus K)]$ for any  finite set $L$.
Let $i\in\omega$ be such that $K\subseteq M_i$.
Then 
$\rho_2(G)\subseteq f_{\beta_n}[\rho_1(C_{\beta_n}\setminus M_i)]\subseteq f_{\beta_n}[\rho_1(C_{\beta_n}\setminus K)]$, a contradiction.

The construction of $E_n$, $K_n$ and $a_n$ is finished.

We define $A = \{a_n:n\in\omega\}$. Since $\rho_2\not\leq_K\rho_1$, we obtain  that $\rho_2$ is tall.
Thus $\I_{\rho_2}$ is a tall ideal (by Proposition~\ref{prop:prop-of-talness}). 
Since $A$ is infinite, there is an infinite set $A_\alpha \subseteq A$ such that $A_\alpha\in \I_{\rho_2}$.

It is not difficult to see, that the sets $A_\alpha$ and $C_\alpha$ satisfy all the required conditions, so the proof of the lemma is finished.
\end{proof}

The main result of this section is as follows.

\begin{theorem}[Assume CH]
\label{thm:TWIERDZENIE-for-Fspaces:Pminus-Mrowka:rho}
Let $\rho_i:\cF_i\to[\Lambda_i]^\omega$ be  partition regular for each $i=1,2$. 
If  
$\rho_1$ and $\rho_2$ are   $P^-$,
 $\rho_2$ has  small accretions and
 $\rho_2\not\leq_K \rho_1$, 
then there exists an almost disjoint family $\cA$  such that 
$|\cA|=\continuum$, $\cA\subseteq \I_{\rho_2}$
and 
$\Phi(\cA)\in  \FinBW(\rho_1) \setminus \FinBW(\rho_2).$ In particular, there is a Hausdorff compact and separable space of cardinality $\continuum$ in $\FinBW(\rho_1) \setminus \FinBW(\rho_2).$
\end{theorem} 

\begin{proof}
Using Proposition~\ref{prop:partition-regular-functions-live-on-the-same-set} we can assume that $\Lambda_1=\Lambda_2=\Lambda$. 
Let $\{f_\alpha: \alpha<\continuum\}$ be an enumeration of all  functions $f:\Lambda_1\to \Lambda_2$ 
and
 $\{F_\alpha:\alpha<\continuum\}$
  be an enumeration of all sets $F\in \cF_2$ having  small accretions.
By Lemma~\ref{lem:LEMAT:rho},  
 there exist  families $\cA = \{A_\alpha:\alpha<\continuum\}$ and  $\cC=\{C_\alpha : \alpha<\continuum\}$
 such that for every $\alpha<\continuum$ all the required conditions of Lemma~\ref{lem:LEMAT:rho} are satisfied. 
We claim that $\cA\setminus\{\emptyset\}$ is the required family.

First, we see that $\cA\setminus\{\emptyset\}$ is an almost disjoint family on  $\Lambda_2$ (by item (\ref{i3}) of Lemma~\ref{lem:LEMAT:rho}) and $\cA\setminus\{\emptyset\}\subseteq\I_{\rho_2}$.

Second, let CH together with item (\ref{i5}) of Lemma~\ref{lem:LEMAT:rho} 
ensures that 
$$|\{\beta< \continuum: \forall K\in [\Omega]^{<\omega}\, (|A_\beta\cap f_\alpha[\rho_1(C_\alpha\setminus K)]|=\omega)\}|\leq|\alpha+1|\leq\omega$$
for each $\alpha<\continuum$, so 
knowing that $\rho_1$ is  $P^-$ we can use Lemma~\ref{lem:TWIERDZENIE-for-Fspaces-in-FinBW-technical} to see that 
$\Phi(\cA)\in \FinBW(\rho_1)$. 

Third, we use Lemma~\ref{lem:MAD-Mrowka-nie-jest-Fspejsem:technical}(\ref{lem:MAD-Mrowka-nie-jest-Fspejsem:technical:item}) along with item (\ref{i6}) of Lemma~\ref{lem:LEMAT:rho} and the fact that $\rho_2$ has small accretions  to see that $\Phi(\cA)\notin\FinBW(\rho_2)$.

Finally, using  Proposition~\ref{prop:countable-Mrowka-is-FinBW} we know that $\cA$ cannot be countable, so $|\cA\setminus\{\emptyset\}|=\continuum$.
 \end{proof}

Now we want to compare Theorem~\ref{thm:TWIERDZENIE-for-Fspaces:Pminus-Mrowka:rho} with Theorem~\ref{thm:TWIERDZENIE-for-Fspaces:Pminus-Mrowka}. Next two examples show that there are partition regular functions $\rho_1$ and $\rho_2$ satisfying the assumptions of Theorem~\ref{thm:TWIERDZENIE-for-Fspaces:Pminus-Mrowka} (so it gives us, under CH, a space in $\FinBW(\rho_1)\setminus \FinBW(\rho_2)$), but not satisfying assumptions of Theorem~\ref{thm:TWIERDZENIE-for-Fspaces:Pminus-Mrowka:rho} (i.e. we cannot apply it).

\begin{example}
There exist partition regular functions 
$\rho_1$ and $\rho_2$ such that $\rho_1$ is $P^-$, $\rho_2$ is not $P^-$ (so we cannot apply Theorem~\ref{thm:TWIERDZENIE-for-Fspaces:Pminus-Mrowka:rho}) and $\I_{\rho_2}\not\leq_K \I_{\rho_1}$.
\end{example}

\begin{proof}
Let $\rho_1=\rho_{\I_{1/n}}$ and $\rho_2=\rho_\Hindman$. By Theorem~\ref{prop:Plike-properties-for-known-rho}, 
$\I_{1/n}$ is $P^+$ (hence, $P^-$) and  $\Hindman$ is not $P^-(\omega)$ (hence, not $P^-$). Applying Proposition~
\ref{prop:properties-of-ideal-versus-rho}(\ref{prop:properties-of-ideal-versus-rho:ideal-rho}), we see that $\rho_1$ is $P^-$ and $\rho_2$ is not $P^-$. By Theorem~\ref{thm:Katetov-between-known-rho}(\ref{thm:Katetov-between-known-rho:Hindman-not-below-Summable}), $\Hindman\not\leq_K \I_{1/n}$.
\end{proof}

The above example may not be satisfactory as all Hausdorff spaces from the class $\FinBW(\rho_2)$ are finite (by Theorem~\ref{thm:compact-metric-implies-FinBW(rho)}(\ref{thm:compact-metric-implies-FinBW(rho):Pminus-is-necessary}) and 
Proposition~\ref{prop:basic-relationships-between-FinBW-like-spaces}(\ref{prop:basic-relationships-between-FinBW-like-spaces:ideal-rho})), so one could just use Theorem \ref{thm:compact-metric-implies-FinBW(rho)}(\ref{thm:compact-metric-implies-FinBW(rho):Pminus-is-necessary}) instead of Theorem~\ref{thm:TWIERDZENIE-for-Fspaces:Pminus-Mrowka}. The next example is more sophisticated.

\begin{example}
There exist partition regular functions 
$\rho_1$ and $\rho_2$ such that $\rho_1$ is $P^-$, $\rho_2$ is not $P^-$ (so we cannot apply Theorem~\ref{thm:TWIERDZENIE-for-Fspaces:Pminus-Mrowka:rho}), $\I_{\rho_2}\not\leq_K \I_{\rho_1}$ and under CH there is a Hausdorff compact separable space of cardinality $\continuum$ in $\FinBW(\rho_2)$.
\end{example}

\begin{proof}
Let $\rho_1=\rho_{\I_{1/n}}$ and $\rho_2=\rho_\conv$, where $\conv$ is an ideal on $\mathbb{Q}\cap[0,1]$ consisting of those subsets of $\mathbb{Q}\cap[0,1]$ that have only finitely many cluster points in $[0,1]$. Then, $\FinBW(\rho_2)=\FinBW(\conv)$ (Proposition~\ref{prop:basic-relationships-between-FinBW-like-spaces}(\ref{prop:basic-relationships-between-FinBW-like-spaces:ideal-rho})). Applying \cite[Definition 4.3, Proposition 4.6 and Theorem 6.6]{MR4584767}, assuming CH, there is a Hausdorff compact separable space of cardinality $\continuum$ in $\FinBW(\rho_2)$. Moreover, $\rho_1$ is $P^-$ and $\rho_2$ is not $P^-$ (by Proposition~
\ref{prop:properties-of-ideal-versus-rho}(\ref{prop:properties-of-ideal-versus-rho:ideal-rho}), Theorem~\ref{prop:Plike-properties-for-known-rho} and \cite[proof of Proposition 4.10(b)]{MR4584767}). Finally, $\I_{\rho_2}\not\leq_K \I_{\rho_1}$ (\cite[Section 2]{MR3696069}).
\end{proof}

Next example shows that there are partition regular functions $\rho_1$ and $\rho_2$ satisfying the assumptions of Theorem~\ref{thm:TWIERDZENIE-for-Fspaces:Pminus-Mrowka:rho} (so it gives us, under CH, a space in $\FinBW(\rho_1)\setminus \FinBW(\rho_2)$), but not satisfying assumptions of Theorem~\ref{thm:TWIERDZENIE-for-Fspaces:Pminus-Mrowka} (i.e. we cannot apply it).

\begin{example}
\label{example:ideal-Katetov-does-not-imply-rho-Katetov-for-Pminus}
There exist  partition regular functions 
$\rho_1$ and $\rho_2$ with  small accretions which are   $P^-$ and such that  
 $\I_{\rho_2} \subseteq  \I_{\rho_1}$
(in particular,  $\I_{\rho_2} \leq_K \I_{\rho_1}$, so we cannot apply Theorem~\ref{thm:TWIERDZENIE-for-Fspaces:Pminus-Mrowka}),
but 
$\rho_2\not\leq_K\rho_1$.
\end{example}

\begin{proof}
Consider the ideal $\nwd=\{A\subseteq\mathbb{Q}\cap[0,1]:\ \overline{A}\text{ is meager}\}$. Let $\rho_2=\rho_\nwd$.

Fix an almost disjoint family $\cA$ of cardinality $\mathfrak{c}$, enumerate it as $\cA=\{A_\alpha:\ \alpha<\mathfrak{c}\}$ and denote $\cA'=\{A\setminus K:\ A\in\cA,\ K\in[\omega]^{<\omega}\}$. Let $I_n=[\frac{1}{n+2},\frac{1}{n+1})$ for all $n\in\omega$. Enumerate also the set $\cB=\{B\subseteq\mathbb{Q}\cap[0,1]:\ B\cap I_n\notin\nwd\text{ for infinitely many }n\in \omega\}$ as $\{B_\alpha:\ \alpha<\mathfrak{c}\}$. Let $\rho_1:\cA'\to[\mathbb{Q}\cap[0,1]]^\omega$ be given by $\rho_1(A_\alpha\setminus K)=B_\alpha\setminus\bigcup_{n\in K}I_n$.

Observe that $\I_{\rho_1}=\{A\subseteq\mathbb{Q}\cap[0,1]:\ \exists_{K\in\fin}\ \overline{A}\setminus\bigcup_{n\in K}I_n\text{ is meager}\}$. Thus, $\nwd\subseteq\I_{\rho_1}$ and $\I_{\rho_2}\leq_K\I_{\rho_1}$. Moreover, it is easy to see that $\rho_1$ and $\rho_2$ both have small accretions (in the case of $\rho_2$ just apply Proposition \ref{prop:ideal-rho-is-sparse}). 
Since $\nwd$ is $F_{\sigma\delta}$ (see \cite[Theorem~3]{MR1955288}), it is $P^-$ (by Proposition~\ref{thm:Plike-properties-for-definable-ideals}(\ref{thm:Plike-properties-for-definable-ideals:Fsigmadelta})) and consequently 
$\rho_2$ is $P^-$ (by Proposition~\ref{prop:properties-of-ideal-versus-rho}(\ref{prop:properties-of-ideal-versus-rho:ideal-rho})).

Now we show that $\rho_1$ is $P^-$. Suppose that $\{C_n: n\in \omega\}\subseteq\I^+_{\rho_1}$ is decreasing and such that $C_n\setminus C_{n+1}\in\I_{\rho_1}$ for all $n\in \omega$. For each $n\in \omega$ let $T_n=\{i\in\omega:\ C_n\cap I_i\notin\nwd\}$. 

Assume first that $T=\bigcap_{n\in\omega}T_n$ is infinite. Since $\nwd$ is $P^-$, for each $i\in T$ we can find $D_i\notin\nwd$, $D_i\subseteq I_i$ with $D_i\subseteq^*C_n$ for all $n\in \omega$. Then for $E=\bigcup_{i\in T}D_i\cap C_i$ we have $E\in\cB$ (as $D_i\notin\nwd$ and $D_i\setminus C_i$ is finite for all $i\in T$). Hence, $E=B_\alpha$ for some $\alpha<\mathfrak{c}$. Moreover, for each $n\in \omega$ we have $\rho_1(A_\alpha\setminus n)=E\setminus \bigcup_{i<n}I_i=\bigcup_{i\in T,i\geq n}D_i\cap C_i\subseteq C_n$.
 
Assume now that $T$ is finite. Inductively pick $i_n\in\omega$ and $D_n\notin\nwd$ such that $i_{n+1}>i_n$ and $D_n\subseteq I_{i_n}\cap C_n$ for all $n\in \omega$. Define $E=\bigcup_{n\in\omega}D_n$ and note that $E\in\cB$. Hence, $E=B_\alpha$ for some $\alpha<\mathfrak{c}$. Moreover, for each $n\in \omega$ we have $\rho_1(A_\alpha\setminus i_n)=E\setminus \bigcup_{i<i_n}I_i=\bigcup_{i\geq n}D_i\subseteq C_n$.

Finally, we will show that $\rho_2\not\leq_K\rho_1$. Fix any $f:\mathbb{Q}\cap[0,1]\to\mathbb{Q}\cap[0,1]$. For each $n\in\omega$ find $r_n\in\mathbb{Q}\cap[0,1]$ such that $f^{-1}[(r_n-\frac{1}{2^n},r_n+\frac{1}{2^n})]\cap I_n\notin\nwd$. This is possible as $[0,1]$ can be covered by finitely many intervals of the form $(r-\frac{1}{2^n},r+\frac{1}{2^n})$ and 
$I_n\cap(\mathbb{Q}\cap[0,1])\notin\nwd$. Since $[0,1]$ is sequentially compact, there is an infinite $S\subseteq \omega$ such that $(r_n)_{n\in S}$ converges to some $x\in[0,1]$. Put $F=\bigcup_{n\in S}f^{-1}[(r_n-\frac{1}{2^n},r_n+\frac{1}{2^n})]\cap I_n$. Then $F\in\cB$ (in particular, $F\in\I_{\rho_1}^+$), so $F=B_\alpha$ for some $\alpha<\mathfrak{c}$. Fix any $E\in\nwd^+$ and enumerate $S=\{s_i:\ i\in\omega\}$ in such a way that $s_i<s_j$ whenever $i<j$. Observe that $E\cap((r_{s_i}-\frac{1}{2^{s_i}},r_{s_i}+\frac{1}{2^{s_i}})\setminus \bigcup_{j>i}(r_{s_{j}}-\frac{1}{2^{s_{j}}},r_{s_{j}}+\frac{1}{2^{s_{j}}}))$ is infinite for some $i\in\omega$ as otherwise $E$ would converge to $x$, so $E\in\nwd$. 

We claim that for every finite set $L\subseteq\mathbb{Q}\cap[0,1]$ we have:
\begin{equation*}
    \begin{split}
        E\setminus L
        &=
        \rho_1(E\setminus L)\not\subseteq f[\rho_2(A_\alpha\setminus(s_i+1))]
        \\&=f\left[F\setminus\bigcup_{j\leq i}(f^{-1}[(r_{s_j}-\frac{1}{2^{s_j}},r_{s_j}+\frac{1}{2^{s_j}})]\cap I_{s_j})\right]  .
            \end{split}
\end{equation*}

Let $L\subseteq\mathbb{Q}\cap[0,1]$ be a finite set. We will show that $E\setminus L\not\subseteq f[F\setminus\bigcup_{j\leq i}(f^{-1}[(r_{s_j}-\frac{1}{2^{s_j}},r_{s_j}+\frac{1}{2^{s_j}})]$. Suppose that $E\setminus L\subseteq f[F\setminus\bigcup_{j\leq i}(f^{-1}[(r_{s_j}-\frac{1}{2^{s_j}},r_{s_j}+\frac{1}{2^{s_j}})]$. Let 
$$x\in E\cap\left(\left(r_{s_i}-\frac{1}{2^{s_i}},r_{s_i}+\frac{1}{2^{s_i}}\right)\setminus \bigcup_{j>i}\left(r_{s_{j}}-\frac{1}{2^{s_{j}}},r_{s_{j}}+\frac{1}{2^{s_{j}}}\right)\right)\setminus L\subseteq E\setminus L.$$ 
Then 
\begin{equation*}
    \begin{split}
        x
        &\in f\left[F\setminus\bigcup_{j\leq i}\left(f^{-1}\left[\left(r_{s_j}-\frac{1}{2^{s_j}},r_{s_j}+\frac{1}{2^{s_j}}\right)\right]\cap I_{s_j}\right)\right]        
        \\&=f\left[\bigcup_{j> i}\left(f^{-1}\left[\left(r_{s_j}-\frac{1}{2^{s_j}},r_{s_j}+\frac{1}{2^{s_j}}\right)\right]\cap I_{s_j}\right)\right]
        \\&\subseteq f\left[\bigcup_{j> i}\left(f^{-1}\left[\left(r_{s_j}-\frac{1}{2^{s_j}},r_{s_j}+\frac{1}{2^{s_j}}\right)\right]\right)\right]
        \\&=f\left[f^{-1}\left[\bigcup_{j> i}\left(r_{s_j}-\frac{1}{2^{s_j}},r_{s_j}+\frac{1}{2^{s_j}}\right)\right]\right]
        \\&\subseteq \bigcup_{j> i}\left(r_{s_j}-\frac{1}{2^{s_j}},r_{s_j}+\frac{1}{2^{s_j}}\right).
            \end{split}
\end{equation*}
A contradiction.
\end{proof}


\part{Characterizations}
\label{part:characterization}

In the final part we want to characterize when $\FinBW(\rho_1) \setminus  \FinBW(\rho_2) \neq\emptyset$ in the cases of $\rho_1\in\{FS, r, \Delta\}\cup\{\rho_\I:\I\text{ is an ideal}\}$. In the realm of partition regular functions that are $P^-$ and have small accretions we were able to obtain a full characterization (Theorem~\ref{thm:characterization-for-rho}(\ref{thm:characterization-for-rho:rho-rho})) using Theorem~\ref{thm:TWIERDZENIE-for-Fspaces:Pminus-Mrowka:rho}. If $\rho_1=\rho_\I$ for some $P^-$ ideal $\I$, then Theorem~\ref{thm:TWIERDZENIE-for-Fspaces:Pminus-Mrowka} gives us a complete characterization (Theorem~\ref{thm:characterization-for-rho}(\ref{thm:characterization-for-rho:ideal-Pminus-rho-finitely-supported})) and this problem for $\rho_1=\rho_\I$ in the case of non-$P^-$ ideals $\I$ is rather complicated (see \cite{MR4584767} and Example \ref{example-unboring}). However, for instance in the case of $\rho_1=FS$ and $\rho_2$ not being $P^-$, we needed another construction -- we were able to obtain a characterization (Theorem \ref{thm:characterization-for-rho:without-Pminus}), but only in the realm of spaces with unique limits of sequences (which are not necessarily Hausdorff).


\section{Characterizations of  distinguishness between \texorpdfstring{$\FinBW$}{FinBW } classes via Kat\v{e}tov order}
\label{sec:characterization}

\begin{theorem}[Assume CH]
\label{thm:characterization-for-rho}
Let $\rho_1$ and $\rho_2$ be  partition regular functions.
Let $\I_1$ be an ideal.
\begin{enumerate}
\item \label{thm:characterization-for-rho:rho-rho}\label{thm:characterization-for-rho:rho-Pminus-rho-Pminus}
    If  
$\rho_1$ is $P^-$ and  
$\rho_2$ is $P^-$ with small accretions,
then 
$$
 \rho_2 \not\leq_K \rho_1
 \iff
\FinBW(\rho_1) \setminus  \FinBW(\rho_2) \neq\emptyset.$$

\item 
\begin{enumerate}

\item 
\label{thm:characterization-for-rho:rho-Pminus-rho-Pplus}
If 
$\rho_1$ is $P^-$ and 
$\rho_2$ is $P^+$,  
then 
$$
 \I_{\rho_2}\not\leq_K \I_{\rho_1}
\iff
\FinBW(\rho_1) \setminus  \FinBW(\rho_2) \neq\emptyset.$$

\item \label{thm:characterization-for-rho:ideal-Pminus-rho-finitely-supported}
If  $\I_1$ is $P^-$, then 
$$
 \I_{\rho_2}\not\leq_K \I_1
\iff
\FinBW(\I_1) \setminus  \FinBW(\rho_2) \neq\emptyset.$$

\end{enumerate} 

\end{enumerate}

Moreover, in every item an example showing that the above difference between $\FinBW$ classes is nonempty is  of the form $\Phi(\cA)$ with $\cA$ being almost disjoint and of cardinality $\continuum$ (in particular, these examples are Hausdorff, compact, separable and of cardinality $\continuum$).
\end{theorem}

\begin{proof}
(\ref{thm:characterization-for-rho:rho-Pminus-rho-Pminus})
The implication ``$\implies$'' follows  from Theorem~\ref{thm:TWIERDZENIE-for-Fspaces:Pminus-Mrowka:rho}, whereas 
the implication ``$\impliedby$'' follows  from  Theorem~\ref{thm:Katetove-implies-inclusion-for-FinBW}(\ref{thm:Katetove-implies-inclusion-for-FinBW:for-rho}).

(\ref{thm:characterization-for-rho:rho-Pminus-rho-Pplus})
The implication ``$\implies$'' follows  from Theorem~\ref{thm:TWIERDZENIE-for-Fspaces:Pminus-Mrowka},
whereas the implication ``$\impliedby$'' follows  from  Theorem~\ref{thm:Katetove-implies-inclusion-for-FinBW}(\ref{thm:Katetove-implies-inclusion-for-FinBW:for-rho-ideals:Pplus}).

(\ref{thm:characterization-for-rho:ideal-Pminus-rho-finitely-supported})
It follows from Theorems \ref{thm:TWIERDZENIE-for-Fspaces:Pminus-Mrowka}, \ref{thm:Katetove-implies-inclusion-for-FinBW}(\ref{thm:Katetove-implies-inclusion-for-FinBW:for-ideals}), \ref{prop:basic-relationships-between-FinBW-like-spaces}(\ref{prop:basic-relationships-between-FinBW-like-spaces:ideal-rho}) and 
Proposition \ref{prop:properties-of-ideal-versus-rho}(\ref{prop:properties-of-ideal-versus-rho:ideal-rho}).
\end{proof}

Next two examples show that in Theorem~\ref{thm:characterization-for-rho} we cannot drop the assumption that $\rho_1$ is $P^-$ and obtain a characterization in the realm of Hausdorff spaces.

\begin{example}
There exist partition regular functions $\rho_1$ and $\rho_2$ with small accretions such that:
\begin{enumerate}
    \item $\rho_1$ is not  $P^-$ and $\rho_2$ is $P^+$,
    \item $\rho_{2}\not\leq_K\rho_{1}$,
    \item  there is no \emph{Hausdorff} space in $\FinBW(\rho_1)\setminus \FinBW(\rho_2)$.
\end{enumerate}
\end{example}

\begin{proof}
Let $\rho_1=\rho_\Hindman$ and $\rho_2=\rho_{\I_{1/n}}$. Then $\rho_1$ and $\rho_2$ have small accretions (by Proposition \ref{prop:ideal-rho-is-sparse}). By Theorem~\ref{prop:Plike-properties-for-known-rho}, 
$\I_{1/n}$ is $P^+$ and  $\Hindman$ is not $P^-(\omega)$ (hence, not $P^-$). Applying Proposition~
\ref{prop:properties-of-ideal-versus-rho}(\ref{prop:properties-of-ideal-versus-rho:ideal-rho}), we see that $\rho_1$ is not  $P^-$ and $\rho_2$ is $P^+$. By Theorem~\ref{thm:Katetov-between-known-rho}(\ref{thm:Katetov-between-known-rho:Summable-not-below-Hindman}), $\I_2\not\leq_K\I_1$, so $\rho_{2}\not\leq_K\rho_{1}$ (by Proposition~\ref{prop:Katetov-for-ideal-rho}(\ref{prop:Katetov-for-ideal-rho:item})). 

By Theorem~\ref{thm:compact-metric-implies-FinBW(rho)}(\ref{thm:compact-metric-implies-FinBW(rho):Pminus-is-necessary}), $\FinBW(\Hindman)$ contains only finite Hausdorff spaces. On the other hand, $\FinBW(\I_2)$ contains all finite spaces (Theorem~\ref{thm:compact-metric-implies-FinBW(rho)}(\ref{thm:compact-metric-implies-FinBW(rho):skonczone})), so there is no Hausdorff space $\FinBW(\Hindman)\setminus \FinBW(\I_{1/n})$. Applying 
Proposition~\ref{prop:basic-relationships-between-FinBW-like-spaces}(\ref{prop:basic-relationships-between-FinBW-like-spaces:ideal-rho}),
we obtain that 
there is  no Hausdorff space in $\FinBW(\rho_1)\setminus \FinBW(\rho_2)$.
\end{proof}

The above example may not be satisfactory as all Hausdorff spaces from $\FinBW(\rho_1)$ are finite.
The next example is more sophisticated.

\begin{example}
\label{example-unboring}
There exist partition regular functions $\rho_1$ and $\rho_2$ with small accretions such that:
\begin{enumerate}
    \item $\rho_1$ is not  $P^-$ and $\rho_2$ is $P^-$,
    \item assuming CH, there is a Hausdorff, compact, separable space of cardinality $\continuum$ in $\FinBW(\rho_1)$,
    \item $\rho_{2}\not\leq_K\rho_{1}$,
    \item there is no \emph{Hausdorff} space in $\FinBW(\rho_1)\setminus \FinBW(\rho_2)$.
\end{enumerate}
\end{example}

\begin{proof}
Let $\I$ and $\J$ be the ideals from \cite[Example 8.9]{MR4584767} and define $\rho_1=\rho_\I$ and $\rho_2=\rho_\J$. Then $\rho_1$ and $\rho_2$ have small accretions (by Proposition \ref{prop:ideal-rho-is-sparse}) and $\J\not\leq_K\I$, so $\rho_{2}\not\leq_K\rho_{1}$ (by Proposition~\ref{prop:Katetov-for-ideal-rho}(\ref{prop:Katetov-for-ideal-rho:item})). By \cite[Example 10.6]{MR4584767} and Proposition~\ref{prop:basic-relationships-between-FinBW-like-spaces}(\ref{prop:basic-relationships-between-FinBW-like-spaces:ideal-rho}), there is no \emph{Hausdorff} space in $\FinBW(\rho_1)\setminus \FinBW(\rho_2)$. Applying \cite[Theorem 6.6]{MR4584767} and Proposition~\ref{prop:basic-relationships-between-FinBW-like-spaces}(\ref{prop:basic-relationships-between-FinBW-like-spaces:ideal-rho}) we see that, assuming CH, there is a Hausdorff, compact, separable space of cardinality $\continuum$ in $\FinBW(\rho_1)$. Since $\J$ is $P^-$, $\rho_2$ is $P^-$ (by Proposition~
\ref{prop:properties-of-ideal-versus-rho}(\ref{prop:properties-of-ideal-versus-rho:ideal-rho})) and $\rho_1$ cannot be  $P^-$ as it would contradict Theorem \ref{thm:characterization-for-rho}(\ref{thm:characterization-for-rho:rho-rho}).
\end{proof}

\begin{question}
Can we  drop the assumption that $\rho_2$ is $P^-$ in  Theorem~\ref{thm:characterization-for-rho} and obtain the characterization in the realm of Hausdorff spaces?
\end{question}

In Theorem~\ref{thm:characterization-for-rho:without-Pminus}, we show that we can drop the assumption that  $\rho_2$ is $P^-$ in Theorem~\ref{thm:characterization-for-rho}(\ref{thm:characterization-for-rho:rho-Pminus-rho-Pminus}) and obtain a characterization in the realm of \emph{non}-Hausdorff spaces with \emph{unique limits} of sequences, but at the cost of requiring that $\rho_1$ is weak $P^+$ instead of $P^-$.

Now we present some applications of Theorem \ref{thm:characterization-for-rho}.

\begin{corollary}[{\cite[Theorem~10.4]{MR4584767}}] Assume CH.
Let $\I_1$ and $\I_2$ be ideals. If  $\I_1$ is $P^-$,  
then the following are equivalent:
\begin{enumerate}
\item $\I_{2}\not\leq_K \I_{1}$,
\item $\FinBW(\I_1) \setminus  \FinBW(\I_2)\neq\emptyset$.
\end{enumerate} 
Moreover, an example showing that the above difference between $\FinBW$ classes is nonempty is  of the form $\Phi(\cA)$ with $\cA$ being almost disjoint and of cardinality $\continuum$ (in particular, these examples are Hausdorff, compact, separable and of cardinality $\continuum$).
\end{corollary}

\begin{proof}
It follows from Theorems \ref{thm:characterization-for-rho}(\ref{thm:characterization-for-rho:ideal-Pminus-rho-finitely-supported}) and \ref{prop:basic-relationships-between-FinBW-like-spaces}(\ref{prop:basic-relationships-between-FinBW-like-spaces:ideal-rho}).
\end{proof}

\begin{question}\ 
\label{q:summable-space-is-vdWaerde-space}
Is every $\I_{1/n}$-space (Hindman space, Ramsey space) a van der Waerden space?
\end{question}

Note that under CH Theorem \ref{thm:characterization-for-rho} reduces the above question to Question \ref{q:Katetov-between-known-rho}.

%
%
%

\begin{corollary}[Assume CH]
\label{cor:I-Pplus-vs-Hindman-like}
Let $\I$ be an ideal.
\begin{enumerate}

\item 
\label{cor:I-Pplus-vs-Hindman-like:rho}
If $\I$ is  $P^-$, then the following conditions are equivalent:
\begin{enumerate}
    \item $\rho_\I\not\leq_K\FS$ ($\rho_\I\not\leq_K r$, $\rho_\I\not\leq_K\Delta$, resp.).
    \item There exists a Hindman (Ramsey,  differentially compact, resp.) space that is not in $\FinBW(\I)$.
\end{enumerate}
Moreover, if $\I$ is $P^+$ then the above are equivalent to $\I\not\leq_K\Hindman$.

\item If $\I$ is $P^-$, then the following conditions are equivalent:
\begin{enumerate}
    \item $\Hindman\not\leq_K\I$ ($\Ramsey\not\leq_K\I$, $\Diff\not\leq_K\I$, resp.).
    \item $\FS\not\leq_K\rho_\I$ ($r\not\leq_K\rho_\I$, $\Delta\not\leq_K\rho_\I$, resp.).
    \item There exists a space in $\FinBW(\I)$ that is not  a Hindman (Ramsey,  differentially compact, resp.) space.
\end{enumerate}
Moreover, in every item an example showing that the above difference between $\FinBW$ classes is nonempty is  of the form $\Phi(\cA)$ with $\cA$ being almost disjoint and of cardinality $\continuum$ (in particular, these examples are Hausdorff, compact, separable and of cardinality $\continuum$).
\end{enumerate}

\end{corollary}

\begin{proof}
It follows from Theorem~\ref{thm:characterization-for-rho}(\ref{thm:characterization-for-rho:rho-Pminus-rho-Pminus}) and Propositions \ref{prop:ideal-rho-is-sparse}, \ref{prop:Plike-basic-properties}(\ref{prop:properties-of-ideal-versus-rho:ideal-rho}), \ref{prop:Plike-properties-for-known-rho}(\ref{prop:Plike-properties-for-known-rho:FS-r-Delta:Pminus}), \ref{prop:Katetov-for-rho-functions}(\ref{prop:Katetov-for-rho-functions:equivalence})(\ref{prop:Katetov-for-rho-function-and-ideal-rho:Pminus}) and \ref{prop:basic-relationships-between-FinBW-like-spaces}(\ref{prop:basic-relationships-between-FinBW-like-spaces:ideal-rho}).
\end{proof}

 \begin{remark}
 In \cite[Corollary~2.8]{MR4356195}, the authors obtained Corollary~\ref{cor:I-Pplus-vs-Hindman-like}(\ref{cor:I-Pplus-vs-Hindman-like:rho}) in  the case of Hindman spaces and $P^+$ ideals but in the realm of \emph{non}-Hausdorff spaces.
 \end{remark}

\begin{corollary}[Assume CH]
Let $\I$ be an ideal.
\begin{enumerate}
    \item 
The following conditions are equivalent:
\begin{enumerate}
    \item $\I\not\leq_K\vdW$ ($\I\not\leq_K\I_{1/n}$, resp.).
    \item $\rho_\I\not\leq_K\rho_\vdW$ ($\rho_\I\not\leq_K\rho_{\I_{1/n}}$, resp.).
    \item There exists a van der Waerden space ($\I_{1/n}$-space) that is not in $\FinBW(\I)$.
\end{enumerate}
    \item 
If $\I$ is a $P^-$ ideal, then the following conditions are equivalent:
\begin{enumerate}
    \item $\vdW\not\leq_K\I$ ($\I_{1/n}\not\leq_K\I$, resp.).
    \item $\rho_\vdW\not\leq_K\rho_\I$ ($\rho_{\I_{1/n}}\not\leq_K\rho_{\I}$, resp.).
    \item There exists a space in $\FinBW(\I)$ that is not a van der Waerden  space ($\I_{1/n}$-space).
\end{enumerate}
\end{enumerate}
Moreover, in every item an example showing that the above difference between $\FinBW$ classes is nonempty is  of the form $\Phi(\cA)$ with $\cA$ being almost disjoint and of cardinality $\continuum$ (in particular, these examples are Hausdorff, compact, separable and of cardinality $\continuum$).
\end{corollary}

\begin{proof}
It follows from Theorem~\ref{thm:characterization-for-rho}(\ref{thm:characterization-for-rho:rho-Pminus-rho-Pminus}) and Propositions \ref{prop:ideal-rho-is-sparse}, \ref{prop:Plike-basic-properties}(\ref{prop:properties-of-ideal-versus-rho:ideal-rho}), \ref{prop:Plike-properties-for-known-rho}(\ref{prop:Plike-properties-for-known-rho:vdW}), \ref{prop:Katetov-for-rho-functions}(\ref{prop:Katetov-for-rho-function-and-ideal-rho:Pminus}) and \ref{prop:basic-relationships-between-FinBW-like-spaces}(\ref{prop:basic-relationships-between-FinBW-like-spaces:ideal-rho}).
\end{proof}



\section{Non-Hausdorff world}
\label{sec:nonHausdorf-world}

\begin{proposition}
The following conditions are equivalent for every topological space $X$.
\begin{enumerate}
\item $X$ has unique limits of sequences.
\item $\rho$-limits of sequences in $X$ are unique for every $\rho$. 
\end{enumerate}
\end{proposition}

\begin{proof}
Since $\rho_\fin$-convergence is convergence, (2)$\implies$(1) is obvious.

(1)$\implies$(2): We will show that the negation of (2) implies the negation of (1). Suppose that there are partition regular $\rho:\cF\to[\Lambda]^\omega$ with $\cF\subseteq[\Omega]^\omega$, $F\in\cF$ and $\{x_n: n\in\rho(F)\}\subseteq X$ which $\rho$-converges to $x$ and to $y$, for some $x,y\in X$, $x\neq y$. Let $\{K_n: n\in \omega\}\subseteq[\Omega]^{<\omega}$ be nondecreasing and such that $\bigcup_{n\in \omega} K_n=\Omega$. For each $n\in \omega$ inductively find any $m_n\in\rho(F\setminus K_n)\setminus\{m_i:\ i<n\}$ (this is possible as $\rho(F\setminus K_n)$ is infinite). Observe that the sequence $(x_{m_n})_{n\in \omega}$ is convergent to $x$ and to $y$. 
\end{proof}

The main result of this section is as follows.

\begin{theorem}[Assume CH]
\label{thm:characterization-for-rho:without-Pminus}
Let $\rho_i:\cF_i\to[\Lambda_i]^\omega$ be  partition regular for each $i=1,2$. 
If  $\rho_1$ is weak $P^+$ and has small accretions, then the following conditions are equivalent.
\begin{enumerate}
    \item $\rho_2\not\leq_K\rho_1$.
    \item There exists a     separable space $X$ with unique limits of sequences such that  $X\in \FinBW(\rho_1)\setminus\FinBW(\rho_2).$
\end{enumerate}\end{theorem}

\begin{proof}
$(2)\implies (1)$.
It follows from Theorem \ref{thm:Katetove-implies-inclusion-for-FinBW}(\ref{thm:Katetove-implies-inclusion-for-FinBW:for-rho}).

$(1)\implies (2)$. 
Fix a list $\{f_\alpha:\alpha<\continuum\}$ of all $\I_{\rho_1}$-to-one functions $f:\Lambda_1\to\Lambda_2$. 

We will construct a sequence $\{D_\alpha: \alpha<\continuum\}\subseteq\cF_1$ such that for every $\alpha<\continuum$ 
we have 
$$\forall E\in\cF_2\, \exists K\in[\Omega_1]^{<\omega}\, \forall L\in[\Omega_2]^{<\omega}\, \left(\rho_2(E\setminus L)\not\subseteq f_\alpha[\rho_1(D_\alpha\setminus K)]\right) $$
 and 
one of the following conditions holds:
\begin{equation*}
\label{eq:W1}
\tag{W1} 
\begin{split}
&
\forall \beta<\alpha  \, \forall M\in[\Lambda_2]^{<\omega} \, \exists K\in[\Omega_1]^{<\omega} \, (f_\alpha[\rho_1(D_\alpha\setminus K) ]\cap (M\cup f_\beta[\rho_1(D_\beta\setminus K)])=\emptyset)
\end{split}
\end{equation*}
or
\begin{equation*}
\label{eq:W2}
\tag{W2} 
\begin{split}
\exists \beta<\alpha \, \forall K\in[\Omega_1]^{<\omega} \, \forall M\in[\Lambda_2]^{<\omega} \, 
&\exists L\in[\Omega_1]^{<\omega} 
\\&
\,  (f_\alpha[\rho_1(D_\alpha\setminus L) ]\subseteq f_\beta[\rho_1(D_\beta\setminus K)]\setminus M).
\end{split}
\end{equation*}

Suppose that $\alpha<\continuum$ and that $D_\beta$ have been chosen for all $\beta<\alpha$. Since $\rho_2\not\leq_K\rho_1$, there is $D^0\in\cF_1$ such that:
\begin{equation}
\tag{A1} 
\forall E\in\cF_2\, \exists K\in[\Omega_1]^{<\omega}\, \forall L\in[\Omega_2]^{<\omega}\ \left(\rho_2(E\setminus L)\not\subseteq f_\alpha[\rho_1(D^0\setminus K)]\right).
\end{equation}
Since $\rho_1$ has small accretions, there is $D^1\in\cF_1$, $D^1\subseteq D^0$, such that for every $K\in[\Omega_1]^{<\omega}$ we have $\rho_1(D^1)\setminus\rho_1(D^1\setminus K)\in\I_{\rho_1}$. Observe that $D^1$ also has the property (A1) as $f_\alpha[\rho_1(D^0\setminus K)]\supseteq f_\alpha[\rho_1(D^1\setminus K)]$ for every $K\in[\Omega_1]^{<\omega}$.

Since $\rho_1$ is weak $P^+$, there is $D\in\cF_1$ such that $\rho_1(D)\subseteq\rho_1(D^1)$ and satisfying property:
\begin{equation*}
\tag{A2} 
\begin{split}
    &\forall \{F_n: n\in \omega\}\subseteq\cF_1 \, (\forall n\in\omega \, (\rho_1(F_{n+1})\subseteq\rho_1(F_n)\subseteq\rho_1(D) ) 
    \\& \implies \exists E'\in\cF_1 \, \forall n\in\omega \, \exists K\in[\Omega_1]^{<\omega} \, (\rho_1(E'\setminus K)\subseteq\rho_1(F_n)).
\end{split}
\end{equation*}

Now we have 2 cases:
\begin{equation*}
\label{eq:P1}
\tag{P1}
\begin{split}
&
\forall D'\in\cF_1\, \forall \beta<\alpha \, (\rho_1(D')\subseteq \rho_1(D)  
\\&\implies  \exists K\in[\Omega_1]^{<\omega}\, ( \rho_1(D')\setminus 
f_\alpha^{-1}[
f_\beta[\rho_1(D_\beta\setminus K)]]\notin \I_{\rho_1})
\end{split}
\end{equation*}
or
\begin{equation*}
\label{eq:P2}
\tag{P2}
\begin{split}
& \exists D'\in\cF_1 \, \exists \beta<\alpha \, (\rho_1(D')\subseteq \rho_1(D)  
\\& 
\land \forall K\in[\Omega_1]^{<\omega} \, ( \rho_1(D')\setminus f_\alpha^{-1}[f_\beta[\rho_1(D_\beta\setminus K)]]\in \I_{\rho_1}).
\end{split}
\end{equation*}

In the first case, let $\{K_n:n\in \omega\}\subseteq[\Omega_1]^{<\omega}$ be such that $\bigcup_{n\in \omega} K_n=\Omega_1$ and let $\alpha\times [\Lambda_2]^{<\omega} =\{(\beta_n,M_n): n\in\omega\}$, taking into account that $\alpha$ is countable (as we assumed CH). 
Using condition \eqref{eq:P1} repeatedly  
and the facts that $f_\alpha^{-1}[\{\lambda\}]\in\I_{\rho_1}$, for every $\lambda\in\Lambda_2$, and $\rho_1(D^1)\setminus\rho_1(D^1\setminus K)\in\I_{\rho_1}$, for all $K\in[\Omega_1]^{<\omega}$,
one can easily construct a sequence $\{E_n: n\in \omega\}\subseteq\cF_1$ such that 
\begin{enumerate}
	\item $\rho_1(E_0)\subseteq \rho_1(D)$,
	\item $\forall {n\in \omega}\, (\rho_1(E_{n+1})\subseteq \rho_1(E_{n}))$,
	\item $\forall {n\in \omega} \, \exists {K\in[\Omega_1]^{<\omega}} \left(\rho_1(E_n) \cap  f_\alpha^{-1}[M_n\cup f_{\beta_n}[\rho_1(D_{\beta_n}\setminus K)]] = \emptyset\right)$,
\item $\forall {n\in \omega} \, \rho_1(E_n)\cap \rho_1(D^1)\setminus\rho_1(D^1\setminus K_n)=\emptyset$. 
\end{enumerate}
Now using property (A2) we find $E'\in \cF_1$ such that 
$\rho_1(E')\subseteq\rho_1(D)\subseteq\rho_1(D^1)$
and for every  $n\in\omega$ there is $K\in[\Omega_1]^{<\omega}$ with $\rho_1(E'\setminus K) \subseteq \rho_1(E_n)$.

It is not difficult to see that $D_\alpha = E'$ satisfies (A1) and  \eqref{eq:W1}, i.e., it is as needed.

Consider the second case.
Let $D'\in\cF_1$ and $\beta<\alpha$ be such that $\rho_1(D')\subseteq \rho_1(D)$ and $\rho_1(D')\setminus f_\alpha^{-1}[f_\beta[\rho_1(D_\beta\setminus K)]]\in \I_{\rho_1}$
 for each $K\in[\Omega_1]^{<\omega}$. 
Since $f_\alpha^{-1}[\{\lambda\}]\in\I_{\rho_1}$ for every $\lambda\in\Lambda_2$,
we also have 
$\rho_1(D')\setminus f_\alpha^{-1}[f_\beta[\rho_1(D_\beta\setminus K)]\setminus M]\in \I_{\rho_1}$ for each $K\in[\Omega_1]^{<\omega}$ and $M\in[\Lambda_2]^{<\omega}$. Recall also that $\rho_1(D^1)\setminus\rho_1(D^1\setminus K)\in\I_{\rho_1}$, for all $K\in[\Omega_1]^{<\omega}$. Since $\rho_1$ is $P^-$ (by Proposition \ref{prop:Plike-basic-properties}(\ref{prop:Plike-basic-properties:implications}), as $\rho_1$ is weak $P^+$), we find an infinite set $D''\in\cF_1$
such that 
\begin{itemize}
\item $\rho_1(D'')\subseteq \rho_1(D')$,
\item for every $K\in[\Omega_1]^{<\omega}$ there is $L\in[\Omega_1]^{<\omega}$ with $\rho_1(D''\setminus L)\subseteq\rho_1(D^1\setminus K)$,
\item for every $K\in[\Omega_1]^{<\omega}$ and $M\in[\Lambda_2]^{<\omega}$ there is $L\in[\Omega_1]^{<\omega}$ with 
$\rho_1(D''\setminus L)\cap (\rho_1(D')\setminus f_\alpha^{-1}[f_\beta[\rho_1(D_\beta\setminus K)]\setminus M]) = \emptyset $.
\end{itemize}

It is not difficult to see that $D_\alpha = D''$ satisfies (A1) and  \eqref{eq:W2}.

The construction of sets $D_\alpha$ is finished.

We are ready to define the required space. Let 
$T=\{\alpha<\continuum:\text{$D_\alpha$ satisfies \eqref{eq:W1}}\}$
 and $$X=\Lambda_2\cup\{\rho_1(D_\alpha):\alpha\in T\}\cup\{\infty\}.$$ 
For every $x\in X$ we define the family $\cB(x)\subseteq \cP(X)$ as follows:
\begin{itemize}
	\item 
$\cB(\lambda)=\{\{\lambda\}\}$ for $\lambda\in \Lambda_2$,
\item $\cB(\rho_1(D_\alpha))= \{\{\rho_1(D_\alpha)\}\cup f_\alpha[\rho_1(D_\alpha\setminus K)]\setminus M:K\in[\Omega_1]^{<\omega},M\in[\Lambda_2]^{<\omega}\}$  for $\alpha\in T$,
\item $\cB(\infty)=\{\{\infty\}\cup\bigcup_{\alpha\in T\setminus F}U_\alpha : F\in[T]^{<\omega} \land U_\alpha\in\cB(\rho_1(D_\alpha)) \text{ for } \alpha\in T\setminus F\}$.
\end{itemize}
It is not difficult to check that 
the family $\cN = \{\mathcal{B}(x):x\in X\}$ 
 is a  neighborhood system  (see e.g.~\cite[Proposition 1.2.3]{MR1039321}).
We claim that $X$ with the topology generated by $\cN$  is a topological space that we are looking for.

First we will show that $X$ has unique limits of sequences. It is not difficult to see that $X\setminus\{\infty\}$ is Hausdorff. Thus, it suffices to check that if $\{x_n: n\in \omega\}\subseteq X$ converges to $\infty$ then it cannot converge to any other point in $X$. Indeed, if $\{x_n: n\in \omega\}$ would converge to some $\lambda\in\Lambda_2$ then it would have to be constant from some point on, so $\{\infty\}\cup\bigcup_{\alpha\in T}(\{\rho_1(D_\alpha)\}\cup f_\alpha[\rho_1(D_\alpha)]\setminus \{\lambda\})$ would be an open neighborhood of $\infty$ omitting almost all $x_n$'s. On the other hand, if $(x_n)_{n\in \omega}$ would converge to some $\rho_1(D_\alpha)$ for $\alpha\in T$ then using \eqref{eq:W1} for each $\beta\in T\setminus\{\alpha\}$ we could find $K_\beta\in[\Omega]^{<\omega}$ with $f_\alpha[\rho_1(D_\alpha\setminus K_\beta) ]\cap f_\beta[\rho_1(D_\beta\setminus K_\beta)]=\emptyset$. Then, denoting $M_\beta=\{x_n: n\in\omega\}\setminus f_\alpha[\rho_1(D_\alpha\setminus K_\beta) ]$ (which is a finite set, as $\{\rho_1(D_\alpha)\}\cup f_\alpha[\rho_1(D_\alpha\setminus K_\beta) ]$ is an open neighborhood of $\rho_1(D_\alpha)$ and $(x_n)_{n\in \omega}$ converges to $\rho_1(D_\alpha)$), the set $\{\infty\}\cup\bigcup_{\beta\in T\setminus\{\alpha\}}(\{\rho_1(D_\beta)\}\cup f_\beta[\rho_1(D_\beta\setminus K_\beta)]\setminus M_\beta)$ would be an open neighborhood of $\infty$ omitting all $x_n$'s. Hence, $X$ has unique limits of sequences.

Now we show that $X\in\FinBW(\rho_1)$. Fix any $f:\Lambda_1\to X$. If there is $x\in X$ with $f^{-1}[\{x\}]\notin\I_{\rho_1}$ then find $F\in\cF_1$ with $\rho_1(F)\subseteq f^{-1}[\{x\}]$ and observe that $(f(n))_{n\in \rho_1(F)}$ is $\rho_1$-convergent to $x$. Thus, we can assume that $f^{-1}[\{x\}]\in\I_{\rho_1}$ for all $x\in X$. There are two possible cases:
$f^{-1}[X\setminus\Lambda_2]\notin\I_{\rho_1}$ or 
$f^{-1}[X\setminus\Lambda_2]\in\I_{\rho_1}$.

If $f^{-1}[X\setminus\Lambda_2]\notin\I_{\rho_1}$ then we find $F\in\cF_1$ with $\rho_1(F)\subseteq f^{-1}[X\setminus\Lambda_2]$. As $f^{-1}[\{x\}]\in\I_{\rho_1}$ for all $x\in X$
and 
$f^{-1}[\{x\}] \neq\emptyset$ only for countably many $x\in X$, using the fact that $\rho_1$ is $P^-$ (by Proposition \ref{prop:Plike-basic-properties}(\ref{prop:Plike-basic-properties:weakPplus-implies-Pminus}))
 we can find $E\in\cF_1$ with $\rho_1(E)\subseteq \rho_1(F)$ and such that for each $x\in X\setminus\Lambda_2$ there is $K\in[\Omega_1]^{<\omega}$ with $\rho_1(E\setminus K)\cap f^{-1}[\{x\}]=\emptyset$. Since for each $U\in\mathcal{B}(\infty)$ there are only finitely many $\alpha\in T$ with $\rho_1(D_\alpha)\notin U$, $(f(n))_{n\in \rho_1(E)}$ $\rho_1$-converges to $\infty$.

If $f^{-1}[X\setminus\Lambda_2]\in\I_{\rho_1}$ then define $g:\Lambda_1\to\Lambda_2$ by $g(\lambda)=f(\lambda)$ for all $\lambda\in\Lambda_1\setminus f^{-1}[X\setminus\Lambda_2]$ and $g(\lambda)=x$ for all $\lambda\in f^{-1}[X\setminus\Lambda_2]$, where $x\in\Lambda_2$ is a fixed point. Then there is $\alpha<\continuum$ with $f_\alpha=g$. 
We have two subcases: $\alpha\in T$ and $\alpha\notin T$.

Assume $\alpha\in T$. Since $\rho_1$ has small accretions, there is $E\subseteq D_\alpha$, $E\in\cF_1$ such that $\rho_1(E)\setminus\rho_1(E\setminus K)\in\I_{\rho_1}$ for all $K\in[\Omega_1]^{<\omega}$. 
Using that $\rho_1$ is $P^-$ (by Proposition \ref{prop:Plike-basic-properties}(\ref{prop:Plike-basic-properties:weakPplus-implies-Pminus})),
we find $D\in\cF_1$ such that 
\begin{itemize}
\item $\rho_1(D)\subseteq \rho_1(E)\setminus f^{-1}[X\setminus\Lambda_2]$,
\item for each $K\in[\Omega_1]^{<\omega}$ there is $L\in[\Omega_1]^{<\omega}$ with $\rho_1(D\setminus L)\subseteq\rho_1(E\setminus K)\subseteq\rho_1(D_\alpha\setminus K)$,
\item for each $M\in[\Lambda_2]^{<\omega}$ there is $L\in[\Omega_1]^{<\omega}$ with $\rho_1(D\setminus L)\cap f^{-1}[M] = \emptyset$.
\end{itemize}
Since each $U\in\mathcal{B}(\rho_1(D_\alpha))$ is of the form $\{\rho_1(D_\alpha)\}\cup f_\alpha[\rho_1(D_\alpha\setminus K)]\setminus M$ for some $K\in[\Omega_1]^{<\omega}$ and $M\in[\Lambda_2]^{<\omega}$,
the subsequence 
$(f(n))_{n\in \rho_1(D)}$ $\rho_1$-converges to $\rho_1(D_\alpha)$.

Assume $\alpha\notin T$.
Then there is $\beta<\alpha$, $\beta\in T$ such that 
$$\forall {K\in[\Omega_1]^{<\omega}} \, \forall {M\in[\Lambda_2]^{<\omega}} \, \exists {L\in[\Omega_1]^{<\omega}} \,  \,\left(f_\alpha[\rho_1(D_\alpha\setminus L) ]\subseteq f_\beta[\rho_1(D_\beta\setminus K)]\setminus M\right)$$
(we take the minimal $\beta<\alpha$ satisfying property (W2)).
Since each open neighborhood of $\rho_1(D_\beta)$ is of the form $\{\rho_1(D_\beta)\}\cup f_\beta[\rho_1(D_\beta\setminus K)]\setminus M$ for some $K\in[\Omega_1]^{<\omega}$ and $M\in[\Lambda_2]^{<\omega}$, $(f_\alpha(n))_{n\in \rho_1(D_\alpha)}$ $\rho_1$-converges to $\rho_1(D_\beta)\in X$. Since $\rho_1$ has small accretions, there is $E\subseteq D_\alpha$, $E\in\cF_1$ such that $\rho_1(E)\setminus\rho_1(E\setminus K)\in\I_{\rho_1}$ for all $K\in[\Omega_1]^{<\omega}$. Then also $(f_\alpha(n))_{n\in \rho_1(E)}$ $\rho_1$-converges to $\rho_1(D_\beta)\in X$. Finally, since $f^{-1}[X\setminus\Lambda_2]\in\I_{\rho_1}$, using that $\rho_1$ is $P^-$ (by Proposition \ref{prop:Plike-basic-properties}(\ref{prop:Plike-basic-properties:weakPplus-implies-Pminus})),
we get $E'\in\cF_1$ such that $\rho_1(E')\subseteq\rho_1(E)\setminus f^{-1}[X\setminus\Lambda_2]$ and for each $K\in[\Omega_1]^{<\omega}$ there is $L\in[\Omega_1]^{<\omega}$ with $\rho_1(E'\setminus L)\subseteq\rho_1(E\setminus K)$. It is easy to see that $f_\alpha\restriction\rho_1(E')=f\restriction\rho_1(E')$ and $(f_\alpha(n))_{n\in \rho_1(E')}$ $\rho_1$-converges to $\rho_1(D_\beta)\in X$.

Finally, we check that $X\notin\FinBW(\rho_2)$. Define $f:\Lambda_2\to X$ by $f(\lambda)=\lambda$ for all $\lambda\in\Lambda_2$ and fix any $E\in\cF_2$. We claim that $(f(n))_{n\in \rho_2(E)}$ does not $\rho_2$-converge. Clearly, it cannot converge to any $x\in\Lambda_2$. Moreover, it cannot converge to any $\rho_1(D_\alpha)$ for $\alpha\in T$ as property (A1) guarantees that for some $K\in[\Omega_1]^{<\omega}$ we have $\rho_2(E\setminus L)\not\subseteq f_\alpha[\rho_1(D_\alpha\setminus K)]$ for all $L\in[\Omega_2]^{<\omega}$, so $U=\{\rho_1(D_\alpha)\}\cup f_\alpha[ \rho_1(D_\alpha\setminus K)]$ would be an open neighborhood of $\rho_1(D_\alpha)$ such that $\rho_2(E\setminus L)\subseteq U$ for no $L\in[\Omega_2]^{<\omega}$.

We will show that  $(f(n))_{n\in \rho_2(E)}$ cannot $\rho_2$-converge to $\infty$. Suppose otherwise, let $\{L_n: n\in \omega\}\subseteq[\Omega_2]^{<\omega}$ be such that $\bigcup_{n\in \omega} L_n=\Omega_2$ and inductively pick $m_n\in\rho_2(E\setminus L_n)\setminus\{m_i:i<n\}$. Then $(f(m_n))_{n\in \omega}$ is convergent to $\infty$. However, if $g:\Lambda_1\to\{f(m_n): n\in\omega\}$ is any bijection (the set $\{f(m_n): n\in\omega\}$ is infinite since $f$ is one-to-one) then $g=f_\alpha$ for some $\alpha$. If $\alpha\in T$ then in $(f(m_n))_{n\in \omega}$ we could find a subsequence converging to $\rho_1(D_\alpha)$ (in the same way as above when showing that $X\in\FinBW(\rho_1)$ in the case of $\alpha\in T$) which contradicts that $X$ has unique limits of sequences. If $\alpha\notin T$ then in $(f(m_n))_{n\in \omega}$ we could find a subsequence converging to $\rho_1(D_\beta)$ for some $\beta<\alpha$, $\beta\in T$ (in the same way as above when showing that $X\in\FinBW(\rho_1)$ in the case of $\alpha\notin T$)  which also contradicts that $X$ has unique limits of sequences.
\end{proof}


\section{Hindman (Ramsey, differentially compact) spaces that are not in \texorpdfstring{$\FinBW(\I)$}{I } and vice versa}
\label{sec:nonHausdorf-world2}

Now we turn our attention to the question when there is a space in $\FinBW(\I)$ 
that is not Hindman (Ramsey, differentially compact, resp.) and vice versa in the case when  $\I$ is an arbitrary ideal.

\begin{corollary}[Assume CH] 
For each ideal $\I$ and $\rho\in \{\FS,r,\Delta\}$ the following conditions are equivalent.
\begin{enumerate}
    \item $\rho_\I\not\leq_K\rho$.
    \item There exists a Hindman (Ramsey, differentially compact, resp.) space that is not in $\FinBW(\I)$.
\end{enumerate}
Moreover, an example showing that the above difference between $\FinBW$ classes is nonempty is separable and has unique limits of sequences. If $\I$ is $P^-$, then this example is of the form $\Phi(\cA)$ with $\cA$ being almost disjoint of cardinality $\continuum$ (in particular, it is Hausdorff, compact, separable and of cardinality $\continuum$).
\end{corollary}

\begin{proof}
It follows from Theorem \ref{thm:characterization-for-rho:without-Pminus} and Proposition \ref{prop:basic-relationships-between-FinBW-like-spaces}(\ref{prop:basic-relationships-between-FinBW-like-spaces:ideal-rho}) as each $\rho\in \{\FS,r,\Delta\}$ is weak $P^+$ (by Theorem \ref{prop:Plike-properties-for-known-rho}(\ref{prop:Plike-properties-for-known-rho:FS-r-Delta:weakPplus})) and has small accretions (by Proposition \ref{prop:ideal-rho-is-sparse}). The case of $P^-$ ideals $\I$ follows from Corollary \ref{cor:I-Pplus-vs-Hindman-like}(\ref{cor:I-Pplus-vs-Hindman-like:rho}).
\end{proof}

In \cite[Definition~4.1]{MR4584767}, the author introduced the following ideal 
$$ \mathcal{BI}=\left\{A\subseteq\omega^3: \exists k\,\left[\forall i<k \,(A_{(i)}\in \Fin^2) \land \forall i\geq k\, (A_{(i)}\in \Fin(\omega^2)) \right]\right\}$$
where $A_{(i)}=\{(x,y)\in \omega^2: (i,x,y)\in A\}$.
The ideal $\mathcal{BI}$  proved to be useful  in research of $\FinBW(\I)$ spaces (see \cite{MR4584767} for more details).

\begin{corollary}[Assume CH]
\label{cor2}
For each ideal $\I$, the following conditions are equivalent.
\begin{enumerate}
    \item $\mathcal{BI}\not\leq_K\I$.
    \item There exists a space in $\FinBW(\I)$ that is not a Hindman (Ramsey, differentially compact, resp.) space.
\end{enumerate}
Moreover, an example showing that the above difference between $\FinBW$ classes is nonempty is  of the form $\Phi(\cA)$ with $\cA$ being infinite maximal almost disjoint (in particular, it is Hausdorff, compact, separable and of cadinality $\continuum$).    
\end{corollary}

\begin{proof}
$(1)\implies(2)$
In \cite[Theorem~5.3]{MR4584767}, the author 
proved that if $\mathcal{BI}\not\leq_K\I$ then there exists an infinite  maximal almost disjoint family $\cA$ such that  $\Phi(\cA)\in\FinBW(\I)$.
Then  Corollary~\ref{cor:nonFINBW-for-MAD}
 shows that $\Phi(\cA)$ is not Hindman (Ramsey nor differentially compact).

$(2)\implies(1)$
Using \cite[Proposition~6.3 and Lemma~3.2(ii)]{MR4584767},
it is not difficult to see that if $\mathcal{BI}\leq_K\I$ then each space in $\FinBW(\I)$ satisfies property $(*)$. 
On the other hand, we know that spaces with $(*)$ property are Hindman, Ramsey and differentially compact (see \cite[Theorem~11]{MR1887003}, 
\cite[Corollary~3.2]{MR4552506}
and
\cite[Corollary~4.8]{MR3097000}, resp.).
\end{proof}

\begin{corollary}[Assume CH]
If $\rho\in\{FS, r, \Delta\}$ then $\FinBW(\rho)\neq\FinBW(\I)$ for every ideal $\I$.
\end{corollary}

\begin{proof}
Let $\rho\in\{FS, r, \Delta\}$ and $\I$ be an ideal. If $\mathcal{BI}\not\leq_K\I$ then $\FinBW(\I)\setminus\FinBW(\rho)\neq\emptyset$ by Corollary \ref{cor2}. On the other hand, if $\mathcal{BI}\leq_K\I$ then the interval $[0,1]$ is in $\FinBW(\rho)$ (by Theorem \ref{thm:FinBW-hFinBW-homogeneous}(\ref{thm:compact-metric-implies-FinBW(rho):metric-compact-spaces}) and Propositions \ref{prop:basic-relationships-between-FinBW-like-spaces}(\ref{prop:basic-relationships-between-FinBW-like-spaces:inclusion}) and \ref{prop:Plike-properties-for-known-rho}(\ref{prop:Plike-properties-for-known-rho:FS-r-Delta:weakPplus}))  and it is not in $\FinBW(\I)$ (by \cite[Proposition 4.6]{MR4584767}, \cite[Example 4.1]{MR3034318} and \cite[Section 2.7]{alcantara-phd-thesis}).
\end{proof}


\bibliographystyle{amsplain}
\bibliography{FKKU}

\end{document}